\documentclass[12pt]{amsart}
\usepackage[margin=45pt]{geometry}

\usepackage[colorlinks=true,linkcolor=purple,citecolor=green,urlcolor=blue]{hyperref}
\usepackage[foot]{amsaddr}
\usepackage{comment}
\usepackage{xcolor}
\usepackage{newtxtext}

\theoremstyle{plain}
\newtheorem{claim}{Claim}
\newtheorem{theorem}{Theorem}[section]
\newtheorem{assumption}[theorem]{Assumption}
\newtheorem{lemma}[theorem]{Lemma}

\newtheorem{proposition}[theorem]{Proposition}

\newtheorem{definition}{Definition}

\def\P{{\mathbb P}} 
\def\E{{\mathbb E}} 
\def\Z{{\mathbb Z}} 

\def \N{{\mathbb{N}}}
\def \R {{\mathbb{R}}}
\def \X {{\mathbf{X} }}

\newcommand{\Bcal}  {{\mathcal B}}
\newcommand{\Ccal}   {{\mathcal C }}

\newcommand{\Gcal}   {{\mathcal G }} 
\newcommand{\rmd}{\,\mathrm{d}}
\def\phi{\varphi}
\def\1{{\mathchoice {1\mskip-4mu\mathrm l}      
{1\mskip-4mu\mathrm l} 
{1\mskip-4.5mu\mathrm l} {1\mskip-5mu\mathrm l}}}

\newcommand{\eps}{\varepsilon}

\newcommand{\ssup}[1] {{{\scriptscriptstyle{({#1}})}}}

\title{Strongly vertex-reinforced jump process on graphs with bounded degree}
\author{Andrea Collevecchio}
\address[\sc A. Collevecchio]{School of Mathematics and Monash Data Futures Institute, Monash University, Victoria 3800, Australia}
\email{andrea.collevecchio@monash.edu}

\author{Tuan-Minh Nguyen}
\address[\sc T.M. Nguyen]{School of Mathematics, Monash University, Victoria 3800, Australia}
\email{tuanminh.nguyen@monash.edu}

\begin{document}

\begin{abstract}
   We study asymptotic behaviours of a non-linear vertex-reinforced jump process defined on an arbitrary infinite graph with bounded degree.  We prove that if the reinforcement function $w$ is reciprocally integrable and non-decreasing, then the process visits only a finite number of vertices.
   In the case where $w$ is regularly varying with index $\alpha>1$, we show that the process eventually gets stuck on a star-shaped subgraph and there is exactly one vertex with unbounded local time. 
\end{abstract}

\maketitle

\tableofcontents

\section{Introduction}\label{se:main} 
\subsection{Description of the model and the main result}
Let $\mathcal G=(V,E)$ be a  connected, undirected infinite graph. We assume that $\mathcal{G}$  has no loops, i.e., edges whose endpoints coincide. Moreover, we suppose that $\mathcal G$ is a graph with bounded degree, i.e. 
$d:=\sup_{v\in V} \deg(v)<\infty,$
where $\deg(v)$ stands for the degree of vertex $v$.
Fix a collection of positive real numbers $\ell=(\ell_v)_{v \in V}$ and set $\ell_*:=\inf_{v\in V }\ell_v$. Let $w:[\ell_*,\infty)\to (0,\infty)$ be a Borel measurable function. Denote by $\R_+$ the set of non-negative real numbers. We consider a continuous-time stochastic process ${\bf X} = (X_t)_{t\in\R_+}$, so-called \textbf{vertex-reinforced jump process} (VRJP)  with initial local times $\ell$ and reinforcement function $w$, denoted by VRJP($\ell, w$), which is defined as follows. 
\begin{enumerate}
\item[i.] It is a c\`adl\`ag process which takes values on $V$ and it jumps to nearest-neighbour vertices.
\item[ii.] For each time $t>0$, conditionally on the past  $\mathcal{F}_t = \sigma\{X_{s},s\le t\}$, the probability that  exactly one jump occurs  during  the time interval $(t,t+h]$, and is  towards a neighbour $v$ of $X_t$  is given by
\begin{equation*} \label{eq:smint}
w\left(\ell_v+\int_0^t \1_{\{X_s=v\}}\text{d}s \right)\cdot h+o(h).
\end{equation*}
The probability of more than one jump  in $(t, t+h]$ is $o(h)$.
\end{enumerate} 

We later define a strong construction for this process in Section \ref{Se:strongconstr} using a collection of i.i.d. exponential random variables, which is inspired by Rubin's construction for Pólya's urn \cite{Davis1990}. It is worth mentioning that under a mild condition on the reinforcement function $w$, the process is non-explosive (see Proposition \ref{prop:non-exp}).

For each vertex $v\in V$, set
$L(v,t):=\ell_v+\int_0^t\1_{\{X_s=v\}}\text{d}s$, that is the \textbf{local time} at vertex $v$ by time $t$. 
 For each subset $U\subset V$  define $L(U,t):=\sum_{v\in U}L(v,t)$.

Denote by $\mathsf{L}^{1}([\ell_*,\infty))$ the set of Lebesgue integrable functions on $[\ell_*,\infty)$. The process $\X$ is  called \textbf{strongly reinforced} if 
$1/w\in \mathsf{L}^{1}([\ell_*,\infty))$. Otherwise, the process is called \textbf{weakly reinforced} if $1/w\notin \mathsf{L}^{1}([\ell_*,\infty))$.

We denote by $\delta(x,y)$ the graph distance between two vertices $x$ and $y$ in $V$, i.e. the number of edges in a shortest path connecting $x$ and $y$. For two subsets $A, B\subset V$ we define $$\delta(A,B)=\inf_{x\in A, y\in B}\delta(x,y).$$ 
For $r\in \Z_+$ and $x\in V$, let
$$\Bcal_r(x)=\{ y\in V: \delta(x,y)\le r\}\quad\text{and} \quad \partial\Bcal_r(x)= \{y\in V: \delta(x,y)= r+1\}.$$
They stand, respectively,  for the discrete ball of radius $r$ with center $x$ and its outer boundary. We denote by $\mathcal{N}_x$ the set of all nearest neighbours of $x$ and  write $x\sim y$ when $y\in \mathcal{N}_x$. For a pair of real numbers $a$ and $b$, let $a\wedge b$ denote the minimum of $a$ and $b$.
We use the convention that $\inf \emptyset =\infty.$

\begin{assumption}\label{assump}
Assume that
\begin{itemize}
\item[i.] $\ell_*>0$ and $\ell^*:=\sup_{v\in V} \ell_v<\infty$;
\item[ii.] $1/w\in \mathsf{L}^{1}([\ell_*,\infty))$ and $w$ is non-decreasing.
\end{itemize}
\end{assumption}

Our aim in this paper is to prove the following localisation result.

\begin{theorem}\label{thm:main}
Assume that $\X=(X_t)_{t\in\R_{+}}$ is a ${\rm VRJP}(\ell,w)$ on $\mathcal{G}=(V,E)$ such that Assumption \ref{assump} is fulfilled. We have that:
\begin{itemize}
    \item[(a)] With probability 1, the process eventually gets stuck in a finite subgraph.
    \item[(b)] If $w$ is continuously differentiable on $[\ell_*,\infty)$ and there exist constants $\beta>1$, $C>0$ and $\lambda_0\in(0,1]$ such that
   \begin{equation}   \label{cond.w}
           \begin{aligned}
     &w(\lambda t)\le C\lambda^\beta w(t) \quad\text{for all $t\ge\ell_*$ and $0<\lambda\le\lambda_0$ such that $\lambda t\ge\ell_*$, and}\\
      &\liminf_{t\to\infty}\frac{w(ct)}{w(t)}>0\quad \text{for every fixed $c>0$.}
    \end{aligned}
       \end{equation}
        then, with probability 1, there exists a unique vertex $v\in V$ such that $L(v, \infty) = \infty$, each neighbour $u\in \mathcal{N}_v$ is visited infinitely often, and all the vertices at distance at least two from $v$ are visited finitely often.
\end{itemize}
\end{theorem}

Note that the condition \eqref{cond.w} holds if $w$ is a regularly varying function with index $\alpha>1$. We conjecture that Theorem~\ref{thm:main} still holds without assuming that $w$ is non-decreasing, differentiable and satisfying the condition \eqref{cond.w}, but rather that $w$ is a reciprocally integrable càdlàg function on $[\ell_*, \infty)$.

\subsection{Literature review}
Random processes with reinforcement form a fundamental class of random walks in which the walker prefers returning to sites they have previously visited rather than exploring unfamiliar ones. One of the most intriguing aspects of these processes is localisation, i.e. the phenomenon in which the walker eventually gets trapped in a finite region when the reinforcement increases fast enough.  

This research field started with the seminal work of Coppersmith and  Diaconis \cite{1}. They introduced a discrete-time random walk model so-called \textit{Edge-Reinforced Random Walk} (ERRW). In this model, the probability that the walk crosses at time $n$ an edge $e$ incident to the position of the walk is proportional to the number of crossings to $e$ by time $n$. Davis \cite{Davis1990} obtained  results for  ERRW with non-linear reinforcement. He proved that the strongly ERRW on $\mathbb{Z}$ eventually localizes on a single edge. This result was later extended by Sellke \cite{Sellke2008} to any bipartite graph with bounded degree (e.g. $\mathbb{Z}^2$). Sellke also conjectured that the localisation on a single edge occurs on arbitrary graphs with bounded degree. This conjecture was settled  by Cotar and Thacker in \cite{CT2017}. See also \cite{LT2007} and \cite{LT2008}.
 
Another process with reinforcement, namely  
Vertex-Reinforced Random Walk (VRRW), was introduced by Pemantle \cite{P1992}. In this model, the probability that the walk at time $n$ visit a vertex $x$ incident to the position of the walk is proportional to the number of visits to $x$ by time $n$. Remarkably, Tarrès \cite{Tarres2004} proved that linear VRRW on $\mathbb Z$ eventually localizes on exactly five consecutive vertices. Criteria for the localisation on 4 or 5 sites were further studied in \cite{BSS2014} and \cite{Schapira2021}.
Volkov \cite{Volkov2001} showed that linear VRRW defined on an infinite graph with bounded degree eventually gets stuck in a finite subgraph with positive probability. A criterion for the localisation on exactly 2 vertices of VRRW with super-linear reinforcement defined on a general graph with bounded degree was given by Cotar and Thacker in \cite{CT2017}. Additionally, phase transitions for the localisation range of strongly VRRW defined on complete graphs were shown by Benaim, Raimond and Schapira in \cite{BRS2013}.

VRJP was originally conceived by Wendelin Werner as a continuous-time counterpart of VRRW. This model was first studied by Davis and Volkov \cite{DV2002, DV2004}. In particular, it was shown in \cite{DV2004} that linear VRJP on any finite graph spends an unbounded amount of time at all vertices while the vector of all the normalized local times converges almost surely to a non-trivial limit. In \cite{ST2015}, Sabot and Tarrès proved that linear VRJP on any finite graph is actually a mixture of time-changed Markov jump processes and the mixing measure is the partition function of a super-symmetric hyperbolic sigma model called $\mathbb{H}^{2|2}$. In particular, they showed that the centered local times of the process converges in distribution to a multivariate inverse Gaussian law. 
Notably, the existence of a recurrent phase was shown by Sabot and Tarrès \cite{ST2015} and also by Angel, Crawford and
Kozma \cite{ACK2014} using different approaches. Using the connection between linear VRJP and linear ERRW, the existence of a transient phase of linear ERRW on $\Z^d$ with $d\ge 3$ was proved by Disertori, Sabot and Tarrès \cite{DST2015}; the recurrence on $\Z^2$ for any initial constant weights was shown by Sabot and
Zeng \cite{SZ2019}; and the uniqueness of the phase transition was recently demonstrated by Poudevigne \cite{P2022}.


Recent papers \cite{NR2021,CNV2022} focused on VRJP with non-linear reinforcement. 
In particular, it was shown in \cite{CNV2022} that strongly VRJP defined on $\mathbb{Z}$ with initial weights 1 eventually gets stuck on exactly 3 consecutive vertices while weakly VRJP is recurrent and all its local times are unbounded. The present paper generalizes the results in \cite{NR2021} and \cite{CNV2022} to general graphs. 
As mentioned earlier, studying  random processes with reinforcement on general graphs is typically challenging and requires innovative techniques. Our core idea for the proof of Theorem \ref{thm:main}, as outlined below, mainly relies on a martingale approach extended from the one in \cite{CNV2022} together with a novel combinatorial representation of a solution to Poisson equation associated to VRJP. 

\subsection{Strategy of the proof and outline of the paper} We  provide a strong construction of the VRJP in Section \ref{Se:strongconstr}  and show that the process  does not have infinitely many jumps in a bounded time interval (i.e. it is non-explosive)  under a mild condition for initial local times. We prove, in Section \ref{sec:finrang}, that under Assumption \ref{assump}, the range of the process is almost surely finite. Consequently, the problem is reduced to  establishing the localisation of VRJP defined on finite graphs.
In Section \ref{sec:nonconv}, we introduce  preliminary results on martingales and Markov chains, which are necessary for the proof of Theorem \ref{thm:main}.
In Section \ref{sec:f.graph}, we show that strongly VRJP defined on a finite graph eventually gets stuck on a star-shaped subgraph. To demonstrate this, we compare the local times relative to any pair of vertices $i$ and $j$ at distance at most two from each other. More specifically, we consider the following process
$$Y_{ij}(t):=\int_{\ell_i}^{L(i,t)}\frac{\rmd u}{w(u)}-\int_{\ell_j}^{L(j,t)}\frac{\rmd u}{w(u)}.$$
This process is a semimartingale which can be decomposed into a sum of a martingale and a finite variation process.  By solving a Poisson equation associated to VRJP and using the matrix-tree theorem for weighted directed graphs, we obtain bounds for the martingale part as well as the finite variation part. Using these bounds, we establish a non-convergence result for $(Y_{ij}(t))_{t\in\R_+}$. In particular, we obtain that almost surely
\begin{align}\label{ratio1}\liminf_{t\to\infty}\frac{L(i,t)\wedge L(j,t)}{t}=0,\end{align} 
for any pair of vertices $i$ and $j$ such that either they are neighbours or $\delta(i,j)=2$ and their common neighbours have finite local times. Let $V^*$ denote the set of vertices with unbounded local times. We observe that if $V^*$ is composed of more than one connected component, then for any connected component $U$, we have almost surely $\delta(U, V^*\setminus U)=2$.
We then show that almost surely
\begin{align}\label{ratio2}&\liminf_{t\to\infty}\frac{L(i,t)}{L(U,t)}>0 \text{ for each connected component $U$ of $V^*$ and for each $i\in U$, and}\\
\label{ratio3}&\liminf_{t\to\infty} \frac{L(U,t)}{L(\widetilde{U},t)}>0 \text{ for any pair of connected components $U$ and $\widetilde{U}$ of $V^*$}.
\end{align}
Combining \eqref{ratio1}, \eqref{ratio2} and \eqref{ratio3}, we point out that $V^*$ is connected. This fact, together with \eqref{ratio1}, is used to infer that $V^*$ contains a single vertex. Hence, there is almost surely a unique vertex with unbounded local time and the process eventually jumps only between this vertex and its neighbours.

\subsection*{AI disclosure}
The first two versions of this manuscript posted on arXiv in 2024 were prepared without the assistance of any AI tools. The latest version was lightly edited and proofread using ChatGPT Education (GPT 5.6 Sol). The authors remain solely responsible for all mathematical content, proofs, and results.

\section{Construction of vertex-reinforced jump process}\label{Se:strongconstr}
 Fix a collection of positive initial local times $\ell=(\ell_v)_{v\in V}$ and a locally bounded Borel measurable reinforcement function $w:(0,\infty)\to (0,\infty)$. Let $\vec{E}=\{(u,v), (v,u) : u, v\in V \text{ and } u\sim v\}$, that is the set of all directed edges induced by $E$. To any directed edge $e\in \vec{E}$, assign a Poisson process $\mathcal{P}(e)$ with rate one. We assume that the Poisson processes are independent. Denote by $(\chi^{e}_n)_{n \ge 1}$ the inter-arrival times for the process  $\mathcal{P}(e)$, i.e. $(\chi^{
e}_n)_{n \in \N}$ are i.i.d. exponential random variables with mean one. We first construct the \lq skeleton\rq\ of the VRJP($\ell,w$) on $\mathcal G=(V,E)$, i.e. a discrete-time process which describes the jumps of the VRJP. 

Let  $\tau_0 = 0$. On the event $\{X_0 =\rho\}$ with some $\rho \in V$, set
$$
\tau_1 := \min_{ v\sim \rho}\frac 1{w(\ell_{v})} \chi^{\ssup {\rho,v}}_1 ,
$$ 
and $L(\rho, \tau_1) := \ell_{\rho} +\tau_1$, and for $x \neq \rho$, let $L(x, \tau_1) = \ell_x$. Moreover set 
$$
{X_{\tau_1} := \arg\min_{v \sim \rho} \frac 1{w(\ell_v)} \chi^{\ssup {\rho,v}}_1.}
$$
For $j\sim\rho$, set
$${\mathfrak{e}}_1(\rho,j):=
\begin{cases}
\chi_2^{(\rho,j)},& j=X_{\tau_1},\\
\chi_1^{(\rho,j)}-w(\ell_j)\tau_1,& j\ne X_{\tau_1},
\end{cases}
$$
and, for $(i,j)\in\vec E$ with $i\ne\rho$, set
${\mathfrak{e}}_1(i,j):=\chi_1^{(i,j)}.$
Suppose we defined $\tau_k, X_{\tau_k},(L(v,\tau_k))_{v\in V}$ and $({\mathfrak{e}}_k(i,j))_{(i,j)\in\vec E}$ for all $k\le n$. 
For each $(i,j)\in \vec{E}$, let
$$\gamma_n(i,j):=1+\sum_{k=1}^n
\1_{\{(X_{\tau_{k-1}},X_{\tau_k})=(i,j)\}}.
$$
On the event $\{X_{\tau_n}=i\}$, for $j\sim i$, let 
 \begin{eqnarray*}
 \tau_{n+1} &:=& \tau_n+\min_{j\sim i}
\frac{{\mathfrak{e}}_n(i,j)}{w(L(j,\tau_n))},\\
 L(i, \tau_{n+1}) &:=& L(i, \tau_n) +  \tau_{n+1} - \tau_n, \quad \mbox{and for $x \neq i$ we set }  L(x,  \tau_{n+1}) = L(x, \tau_n),\\
 X_{\tau_{n+1}} &:=&  \arg\min_{j\sim i}
\frac{{\mathfrak{e}}_n(i,j)}{w(L(j,\tau_n))}.
 \end{eqnarray*}
 For $j\sim i$, set
$${\mathfrak{e}}_{n+1}(i,j):=
\begin{cases}
\chi_{\gamma_n(i,j)+1}^{(i,j)},& j=X_{\tau_{n+1}},\\ {\mathfrak{e}}_n(i,j)-w(L(j,\tau_n))(\tau_{n+1}-\tau_n),& j\ne X_{\tau_{n+1}},
\end{cases}$$
and for $(x,y)\in\vec E$ with $x\ne i$, set
${\mathfrak{e}}_{n+1}(x,y):={\mathfrak{e}}_n(x,y).$
 By recursion, $\tau_n, X_{\tau_n}$,  $(L(v, \tau_n))_{v \in V}$ are defined for all $n\in \N$.  
 

\begin{proposition}\label{VRJP.constr}
The process $\X=(X_t)_t$, defined by $$X_t:=X_{\tau_n} \quad \text{if } t\in[\tau_n,\tau_{n+1})\quad\text{for all $n\in \Z_+$},$$
is a ${\rm VRJP}(\ell,w)$ on $\mathcal G=(V,E)$.
\end{proposition}
      \begin{proof}
We first prove by induction that, for every $n\in\N$, conditionally on $\mathcal F_{\tau_n}$, the random variables
\begin{align}\label{cond.ind}
(\mathfrak e_n(x,y))_{(x,y)\in\vec E}, \quad\big(\chi_{\gamma_n(x,y)+r}^{(x,y)}\big)_{(x,y)\in\vec E,\ r\in\N}
\end{align}
are jointly independent and exponentially distributed with mean one.
Indeed, by the construction, on $\{X_{\tau_1}=j\}$, we have $\mathfrak e_1(\rho,j)=\chi_2^{(\rho,j)}$ and $\gamma_1(x,y)=1+\mathbf1_{\{(x,y)=(\rho,j)\}}$.
Applying Lemma~\ref{lem.exp} (see the appendix) to $(\chi_1^{(\rho,u)})_{u\in\mathcal N_\rho}$, we have that, conditionally on $\mathcal F_{\tau_1}=\sigma(\tau_1,X_{\tau_1})$, the random variables $(\mathfrak e_1(\rho,u))_{u\in\mathcal N_\rho\setminus\{X_{\tau_1}\}}$ are independent and exponentially distributed with mean one.
Using the joint independence of $(\chi_r^{(\rho,u)})_{u\in\mathcal N_\rho,\ r\ge2}$ and $(\chi_r^{(x,y)})_{(x,y)\in\vec E,\ x\ne\rho,\ r\in\N}$ from $(\chi_1^{(\rho,u)})_{u\in\mathcal N_\rho}$, we include these variables in the same conditional independence assertion.
Hence \eqref{cond.ind} holds for $n=1$.
Suppose the assertion holds for some $n\ge 1$. By this induction assertion and Lemma~\ref{lem.exp}, applied conditionally on $\mathcal F_{\tau_n}$ to $(\mathfrak e_n(i,u))_{u\in\mathcal N_i}$, we have that, conditional on $\mathcal F_{\tau_{n+1}}\cap \{X_{\tau_n}=i,X_{\tau_{n+1}}=j\}$, with $i\sim j$,
\begin{align*}
\bigl(\mathfrak e_n(i,u)-w(L(u,\tau_n))(\tau_{n+1}-\tau_n)\bigr)_{u\in\mathcal N_i\setminus\{j\}},\quad (\mathfrak e_n(x,y))_{(x,y)\in\vec E,\ x\ne i},\quad (\chi_{\gamma_n(x,y)+r}^{(x,y)})_{(x,y)\in\vec E,\ r\in\mathbb N}
\end{align*}
are jointly independent exponential random variables with mean one.
Note that $\mathcal F_{\tau_{n+1}}=\mathcal F_{\tau_n}\vee\sigma(\tau_{n+1},X_{\tau_{n+1}})$.
By the construction, on $\{X_{\tau_n}=i,X_{\tau_{n+1}}=j\}$, we have $\gamma_{n+1}(x,y)=\gamma_n(x,y)+\1_{\{(x,y)=(i,j)\}}$ for each $(x,y)\in\vec E$ and $\mathfrak e_{n+1}(i,j)=\chi_{\gamma_n(i,j)+1}^{(i,j)}$.
Moreover, on the same event, $\mathfrak e_{n+1}(x,y)=\mathfrak e_n(x,y)$ for $(x,y)\in\vec E$ with $x\ne i$.
Hence, \eqref{cond.ind} also holds for $n+1$, and thus holds for all $n\ge1$ by induction.

Fix $t\ge0$. On $\{\tau_n\le t<\tau_{n+1},\, X_t=i\}$, the local times of the neighbours of $i$ are constant until $\tau_{n+1}$. Therefore, using \eqref{cond.ind} and the memoryless property, for $j\sim i$, on $\{\tau_n\le t<\tau_{n+1},X_t=i\}$, we have
\begin{align*}
&\P\big(\tau_{n+1}\le t+h,\ X_{\tau_{n+1}}=j\mid\mathcal F_t\big)\\
&\quad=w(L(j,t))\int_0^h\exp\left(-s\sum_{u\in\mathcal N_i}w(L(u,t))\right)\rmd s=w(L(j,t))h+O(h^2).
\end{align*}
Using the fact that $1-e^{-a h}\le a h$ for $0<h\le 1$, on $\{\tau_n\le t<\tau_{n+1},X_t=i\}$, we have
$$ 
\P(\tau_{n+2}\le t+h\mid\mathcal F_t)
\le d^2h^2 \left( \max_{z\in\mathcal B_2(i)} \sup_{L(z,t)\le r\le L(z,t)+1}w(r) \right)^2
=O(h^2),
$$
Here in the first inequality, we used the fact that each local time increases by at most $h\le 1$ during $(t,t+h]$, and in the last estimate we use the fact that $w$ is locally bounded. Consequently, on $\{\tau_n\le t<\tau_{n+1},X_t=i\}$, we obtain
$$
\P\big(\tau_{n+1}\le t+h<\tau_{n+2},\ X_{\tau_{n+1}}=j\mid\mathcal F_t\big)=w(L(j,t))h+o(h).
$$
This defines the conditional transition law of a ${\rm VRJP}(\ell,w)$ up to $\tau_\infty:=\lim_{n\to\infty}\tau_n$.
    \end{proof}

In our next result, we provide a sufficient condition for the non-explosion of vertex-reinforced jump processes. In this context, we work under more general assumptions than the ones in  Theorem \ref{thm:main}. In particular we could allow $\ell_* = 0$ and $\ell^*=\infty$. Moreover, $w$ is not necessarily monotone increasing. It is evident that the condition $0<\ell_*\le \ell^*<\infty$ implies \eqref{non-ex} below. Recall that $\Bcal_r(x)=\{ y\in V: \delta(x,y)\le r\}$ and that $\partial\Bcal_r(x)$ denotes the outer boundary.

\begin{proposition}[Non-explosion]
Suppose that $X_0=\rho$ and  that $w:(0,\infty)\to(0,\infty)$ is locally bounded and  satisfies
\begin{equation}\label{non-ex}\sum_{r=0}^{\infty}
\frac{1}{\max_{v\in \partial\Bcal_r(\rho)}w(\ell_v)}=\infty.
\end{equation}
Then, the process $\X$ is non-explosive, i.e., $\tau_{\infty}=\infty$ almost surely. In other words, for any $t\in \R_+$, there exists a jump after time $t$.
\label{prop:non-exp}
\end{proposition}

\begin{proof}
We reason by contradiction. Let 
$$\tau_{\infty }:= \lim_{n \to \infty } \tau _{n}, \quad A: = \{\tau_{\infty} < \infty \}\quad\text{and}\quad B:= \big\{\lim_{n \to \infty } \delta(X_{\tau _{n}},\rho) = \infty \big\}.$$ Assume that
${\mathbb{P}}(A) > 0$. Observe that 
\begin{align}
\label{eq:int1} L(x, \tau _{\infty }) \le \ell _{x}+ \tau
_{\infty }<\infty \quad\text{on $A$}, \text{ for all } x \in {{V}}.
\end{align}
Set $T_r:=\inf\{ \tau_n : X_{\tau_n}\in \partial\mathcal{ B}_{r}(\rho)\}.$
On $B$, we must have $T_r<\infty$ for all $r\in \N$.
Recall that $d:=\sup_{v\in V}\deg(v)<\infty$, and set
$
w_r:=d\max_{v\in\partial \mathcal B_r(\rho)}w(\ell_v).
$
For $t<T_r$ and $v\in\partial \mathcal B_r(\rho)$, we have
$
L(v,t)=\ell_v.
$
It follows that, before time $T_r$, the total rate at which the process jumps to $\partial \mathcal B_r(\rho)$ is bounded above by
$$
\sum_{\substack{v\sim X_t\\v\in\partial \mathcal B_r(\rho)}}w(L(v,t)) = \sum_{\substack{v\sim X_t\\v\in\partial \mathcal B_r(\rho)}}w(\ell_v) \le w_r.
$$
Consequently, on $\{T_{r-1}<\infty\}$, we have
$$
\P\left(T_r-T_{r-1}>s\mid\mathcal F_{T_{r-1}}\right) \ge e^{-w_rs} \quad\text{for all $s\ge0$.}
$$
Thus, for every $\theta>0$, we have a.s.
$$\E\big[ e^{-\theta(T_r-T_{r-1})}\1_{\{T_r<\infty\}} \mid\mathcal F_{T_{r-1}} \big] \le \frac{w_r}{w_r+\theta}\1_{\{T_{r-1}<\infty\}}.$$
Iterating this inequality, we get
\begin{align}\label{Tr.bound}
\E\left[e^{-\theta T_r}\1_{\{T_r<\infty\}}\right] \le \prod_{k=1}^r\frac{w_k}{w_k+\theta}.
\end{align}
By \eqref{non-ex}, we have $\sum_{k=1}^\infty\frac1{w_k}=\infty$, and thus $\prod_{k=1}^r\frac{w_k}{w_k+\theta}\rightarrow0$ as $r\to\infty$. On $A\cap B$, we have $T_r\le\tau_\infty<\infty$ for every $r$. Hence, by \eqref{Tr.bound},
$$
\E\left[e^{-\theta\tau_\infty}\1_{A\cap B}\right]
\le \E\left[e^{-\theta T_r}\1_{\{T_r<\infty\}}\right] \rightarrow0.
$$
It follows that
$
\P(A\cap B)=0.
$
Hence, on the event $A$, there exists, a.s., a vertex $x$ which is visited infinitely often by $(X_{\tau_n})_{n\in\Z_+}$. Notice that there exists a neighbour $y$ of $x$ such that the process jumps from $x$ to $y$  infinitely often.  Using  \eqref{eq:int1} applied to $L(y, \tau _{\infty })$ and the local boundedness of $w$, we have
that
\begin{align*}
L(x, \tau _{\infty }) \ge \frac{1}{\sup_{t\in [\ell_y, L(y,\tau_{\infty}))}w(t)} \sum_{k=1}^{\infty }\chi ^{\ssup{x,y}}_{k} = \infty\quad \text{a.s. on $A$},
\end{align*}
which contradicts~\eqref{eq:int1}.
\end{proof}

\section{Finite range of strongly vertex-reinforced jump process}\label{sec:finrang}

In this section, we aim to prove the following result which implies Theorem \ref{thm:main}(a).

\begin{theorem}\label{them:loc}
Let $\X=(X_t)_{t\in\R_+}$ be a VRJP$(\ell,w)$ starting from $X_0=\rho$ with some fixed vertex $\rho\in V$. Suppose that  Assumption \ref{assump} is fulfilled. For each $k\in\Z_+$, let
\begin{equation}\label{hitting}
    T_{k} := 
  \inf \big\{t\in\R_+ \colon X_t \in\partial  \mathcal{B}_{k+1}(\rho)\text{ and $X_t$ is incident to $\partial  \mathcal{B}_{k+2}(\rho)$}\big\}.\end{equation}
Then there exists a constant $\gamma\in(0,1)$ such that for each $k\in \N$,
$$\P(T_{3k}<\infty)\le \gamma^{k}.$$
As a result, the process $\X$ gets stuck on a finite number of vertices with probability 1. 
\end{theorem}

Throughout this section we always suppose that 
Assumption \ref{assump} is fulfilled. In order to demonstrate Theorem~\ref{them:loc}, we first prove some preliminary results.  

\begin{proposition}\label{lem:loc1} Assume that $X_0=v$ and $\max_{u\in \mathcal{B}_2(v)}\ell_u\le M$. Let
\begin{align}\label{sigma.G}
\mathcal{G}_v=\sigma \Big(\chi_{n}^{e} : n\in \N,  e\in \vec{E}\setminus\{ (v,u), (u,s) \text{ for } v\sim u \text{ and } s\sim u \} \Big).
\end{align}
 Then there exists a deterministic positive constant $p(M)>0$ depending on only  $M$ such that
$$\P\big( X_t \in \mathcal{B}_1(v), \forall t\in \R_+ \mid \mathcal{G}_v \big)\ge p(M).$$

\end{proposition}

\begin{proof}
Recall that $\mathcal N_v$ is the set of neighbours of $v$. 
For each $n\in\N$, let $\sigma_n(v)$ be the time of the $n$-th visit to $v$, with $\sigma_1(v)=0$. Let $(\eta_n)_{n\in\N}$ be independent exponential random variables
with mean one, independent of all the variables
$(\chi_j^e)_{j\in\N,e\in\vec E}$. 
On $\{\sigma_n(v)=\tau_k<\infty\}$, with some $k\in\Z_+$, set
$$
\xi_n:=
(\tau_{k+1}-\tau_k) \sum_{u\in\mathcal N_v}w(L(u,\sigma_n(v)))
$$
and set $\xi_n:=\eta_n$ on $\{\sigma_n(v)=\infty\}$.

As in the proof of \eqref{cond.ind}, one can prove by induction that, for every $k\in\N$, conditionally on $\mathcal F_{\tau_k}$ and $(\chi_j^{(x,y)})_{j\in\N,\,(x,y)\in\vec E,\,x\ne v}$, the random variables
$$
(\mathfrak e_k(v,u))_{u\in\mathcal N_v},\quad
(\chi_{\gamma_k(v,u)+r}^{(v,u)})_{u\in\mathcal N_v,\,r\in\N}
$$
are jointly independent and exponentially distributed with mean one. By the construction and this conditional independence, conditionally on $\{X_{\tau_k}=v\}$, $\mathcal F_{\tau_k}$ and $(\chi_j^{(x,y)})_{j\in\N,\,(x,y)\in\vec E,\,x\ne v}$, we have that $(\tau_{k+1}-\tau_k)\sum_{u\in\mathcal N_v}w(L(u,\tau_k))$ is exponentially distributed with mean one. By the definition of $\xi_n$, using the independent variable $\eta_n$ when $\sigma_n(v)=\infty$, we therefore have, for every $n\in\N$ and $t\ge0$,
$$
\P\left(\xi_n>t\mid(\xi_k)_{1\le k\le n-1},(\chi_j^{(x,y)})_{j\in\N,\,(x,y)\in\vec E,\,x\ne v}\right)=e^{-t}.
$$
By successive conditioning, we conclude that $(\xi_n)_{n\in\N}$ are independent exponentially distributed random variables with mean one, and that the sequence is jointly independent of $\Gcal_v$ and $(\chi_j^{(u,s)})_{j\in\N,\,u\in\mathcal N_v,\,s\in\mathcal N_u}$.

For $n\in\N$, define
$$S^{\ssup u}_n:=\sum_{j=1}^n \frac{\chi_j^{(u,v)}}{w\left(\ell_*+\frac{1}{w(2M)\deg(v)}\sum_{k=1}^j\xi_k\right)}.$$
Set
$A_n:= \big\{X_{\tau_{2k}}=v\text{ and }X_{\tau_{2k+1}}\in\mathcal N_v\text{ for all }0\le k\le n\}\subset\{X_t\in\Bcal_1(v)\text{ for all }t\le\tau_{2n+1}\big\}$  and
$$\widetilde A_n :=\left\{\sum_{k=1}^n \frac{\chi_k^{(u,v)}}{w(L(v,\tau_{2k-1}))} \le \inf_{s\in\mathcal N_u\setminus\{v\}} \frac{\chi_1^{(u,s)}}{w(2M)} \text{ for all }u\in\mathcal N_v \right\}.$$ 
By the construction in Section~\ref{Se:strongconstr}, we note that
$\widetilde A_n\cap\{L(u,\tau_{2n})\le2M,\ \forall u\in\mathcal N_v\} \subset A_n$.
Additionally, for all $n\in\N$, on the event
$\widetilde A_n\cap\{L(u,\tau_{2n})\le2M,\ \forall u\in\mathcal N_v\}$, we have
\begin{align}
L(v,\tau_{2k+1})\ge\ell_*+ \frac{1}{w(2M)\deg(v)} \sum_{j=1}^{k+1}\xi_j \quad\text{for all }0\le k\le n.\label{eq.Lv}
\end{align}
It follows from \eqref{eq.Lv} that, for all $u\in\mathcal N_v$ and $n\in\N$, 
\begin{align}
L(u,\tau_{2n+2}) \le \ell_u+S_{n+1}^{(u)} \le M+S_{n+1}^{(u)}\quad\text{on $\widetilde A_n\cap\{L(u,\tau_{2n})\le2M,\ \forall u\in\mathcal N_v\}.$}
\label{eq.Lu}
\end{align}
We next show by induction that for all $n\in \N$,
\begin{equation}\label{eq.incl}B_n:=\Big\{ S^{\ssup u}_n \le M \wedge   \inf_{s\in \mathcal{N}_u\setminus\{v\}} \frac{\chi_{1}^{(u,s)}}{w(2M)},\ \forall u\in \mathcal{N}_v \Big\}\subset \widetilde{A}_n\cap \{L(u,\tau _{2n})\le 2M,\ \forall u\in \mathcal{N}_v\}.\end{equation}
Indeed, since $L(u, 0)=\ell_u\le M$ for all $u\in \mathcal{B}_2(v)$, we have  $L(v,\tau_{1})\ge \ell_*+\frac{1}{w(2M)\deg(v)}\xi_1$ and thus $L(u,\tau_2)\le \ell_u  + S_1^{\ssup u}\le M+S_1^{\ssup u}$. It implies that 
$$B_1=\Big\{ S^{\ssup u}_1 \le M \wedge   \inf_{s\in \mathcal{N}_u\setminus\{v\}} \frac{\chi_{1}^{(u,s)}}{w(2M)},\ \forall u\in \mathcal{N}_v \Big\}\subset \widetilde{A}_1\cap\big\{ L(u,\tau_2)\le 2M, \ \forall u\in\mathcal N_v\big\}.$$
Hence \eqref{eq.incl} holds for $n=1$. Assume that \eqref{eq.incl} holds up to some $n\in\N$. We notice that $$B_{n+1}\subset B_n\subset \widetilde{A}_n\cap\big\{ L(u,\tau_{2n})\le 2M, \ \forall u\in \mathcal{N}_v\big\}.$$
In virtue of \eqref{eq.Lu}, we thus have 
$$L(u,\tau_{2n+2})\le M+S_{n+1}^{\ssup u} \le 2M\quad\text{on}\quad B_{n+1}.$$ Furthermore, by reason of \eqref{eq.Lv}, we have that  $$\sum_{k=1}^{n+1}\frac{\chi^{(u,v)}_{k}}{w(L(v,\tau_{2k-1}))}\le S_{n+1}^{\ssup u} \le \inf_{s\in \mathcal{N}_u\setminus\{v\}} \frac{\chi_{1}^{(u,s)}}{w(2M)}\quad\text{on}\quad B_{n+1}.$$
Hence $B_{n+1}\subset \widetilde{A}_{n+1}\cap \big \{L(u,\tau_{2n+2})\le 2M,\ \forall u\in\mathcal N_v\big\}$. By the principle of induction, \eqref{eq.incl} holds for all $n\in\N$.

On the other hand, $(\xi_j)_{j\in\N}$ and
$(\chi_j^{(u,v)})_{j\in\N,u\in\mathcal N_v}$ are independent, and all these random variables are exponentially distributed with mean one. By virtue of Lemma~\ref{lem.sum}, $S_{n}^{\ssup u}$ converges almost surely to a finite random variable $S_{\infty}^{\ssup u}$ as $n\to\infty$. Hence 
$$\Big\{ S^{\ssup u}_{\infty} \le M \wedge \inf_{s\in \mathcal{N}_u\setminus\{v\}} \frac{ \chi_{1}^{(u,s)}}{w(2M)}, \ \forall u\in \mathcal{N}_v \Big\}\subset \{X_t \in \mathcal B_1(v), \forall t\ge 0 \}.$$
We thus have
\begin{align*}
\P\left(X_t\in\Bcal_1(v),\ \forall t\in\R_+\mid\Gcal_v\right)
&\ge \P\Big( S^{\ssup u}_{\infty} \le M \wedge \inf_{s\in \mathcal{N}_u\setminus\{v\}} \frac{ \chi_{1}^{(u,s)}}{w(2M)}, \ \forall u\in \mathcal{N}_v \Big)\\
&= \E\left[ e^{-w(2M)\sum_{u\in\mathcal N_v}(\deg(u)-1)S_\infty^{\ssup u}} \1_{\{S_\infty^{\ssup u}\le M,\ \forall u\in\mathcal N_v\}} \right]\\
&\ge e^{-d(d-1)w(2M)M} \P\left(S_\infty^{\ssup u}\le M,\ \forall u\in\mathcal N_v\right).\end{align*}
Conditionally on $(\xi_j)_{j\in\N}$, the random variables
$(S_\infty^{\ssup u})_{u\in\mathcal N_v}$ are independent and identically distributed. Hence, for any $u\in\mathcal N_v$,
\begin{align*}
\P\left(S_\infty^{\ssup u}\le M,\ \forall u\in\mathcal N_v\right)
&= \E\left[ \P\left(S_\infty^{\ssup u}\le M\mid(\xi_j)_{j\in\N}\right)^{\deg(v)} \right] \ge \P\left(S_\infty^{\ssup u}\le M\right)^d.
\end{align*}
Assuming that $S^{\ssup u}_{\infty}> M$ a.s. for some $u$, by virtue of Lemma~\ref{lem.sum}, we have $$\exp\Big(-w(2M)\deg(v)\int_{\ell_*}^{\infty} \frac{y\rmd x}{w(x)+y} \Big)= \E[e^{-y S^{\ssup u}_{\infty}}]< e^{-yM}, \quad \text{for all}\ y>0.$$ 
This however contradicts the fact that  $\frac{w(2M)\deg(v)}{M}\int_{\ell_*}^{\infty} \frac{\rmd x}{w(x)+y}<1$ for large enough $y$.  It follows that $\P\left(X_t\in\Bcal_1(v),\ \forall t\in\R_+\mid\Gcal_v\right)$ is bounded from below by some constant $p(M)>0$. 
\end{proof}

\begin{definition}\label{def.cutset}
We define a \textbf{cut-set} $\mathcal{C}$ as a finite set of vertices such that $V\setminus\mathcal C$ can be partitioned into two disjoint vertex sets $V^{(0)}$ and $V^{(1)}$, where $V^{(0)}$ is finite and contains $\rho$, $V^{(1)}$ is infinite, and there is no edge between $V^{(0)}$ and $V^{(1)}$. Moreover, we require that each vertex in $\mathcal{C}$ is a neighbour of a
vertex in $V^{(1)}$.
\end{definition}

 \begin{proposition}\label{cutset}
 Assume that $X_0=\rho$ and let $\mathcal{C}$ be a cut-set. Let $V^{(0)}$ and $V^{(1)}$ be respectively the finite and infinite vertex sets obtained by removing $\mathcal{C}$. Set
  $T:= \inf\big\{t\in\R_+\colon  X_t \in  V^{(1)}\big\}.$
Then
$$\P\big({L}(\mathcal{C},{T}) -   L(\mathcal{C},0)> u \mid T<\infty \big) \le e^{-  w(\ell_*) u }.$$
\end{proposition}
 
 \begin{proof}
Since the graph is connected and all the jump rates are positive, we note that $\P(T<\infty)>0$.
Let $\xi(t):=\inf\{s\in\R_+: L(\mathcal{C}, s)-L(\mathcal{C},0)> t\}$ and
$$
R(t) := \P\big(L(\mathcal{C},T) -  L(\mathcal{C},0)>  t\mid T<\infty\big)=\P(\xi(t)<  T\mid T<\infty).
$$
By the definition of $\xi(t)$ and the continuity of $L(\mathcal C,\cdot)$, we  note that $X_{\xi(t)}\in \mathcal{C}$ on $\{\xi(t)<\infty\}$.
 Notice that 
 \begin{align}\label{inq.xi}
\xi(t+h)\ge \xi(t)+h. 
\end{align}
Indeed, this is obvious when $\xi(t+h)=\infty$. If $\xi(t+h)<\infty$ we have $\xi(s)=\int_{0}^{\xi(s)} \1_{\{X_u\in \mathcal{C}^c\}}\rmd u+s$ for all $0\le s\le t+h$. In this case, we have
$$
\xi(t+h)=\int_{0}^{\xi(t+h)} \1_{\{X_u\in \mathcal{C}^c\}}\rmd u+t+h \ge \int_{0}^{\xi(t)} \1_{\{X_u\in \mathcal{C}^c\}}\rmd u +t+h =\xi(t)+h.
$$
Using the above observation, we have 
\begin{align}\nonumber
R(t)-R(t+h)
&=\P\big(\xi(t)< T \le \xi(t+h)\mid T<\infty \big)\\
\label{inq2}
&\ge \P\big(\xi(t)<T\le \xi(t)+h\big)/\P(T<\infty)\\
\nonumber
&=\E\Big[\P\big(T\le \xi(t)+h\mid\mathcal F_{\xi(t)}\big)
\1_{\{\xi(t)<T\}}\Big]/\P(T<\infty).
\end{align}
On the event $\{\xi(t)<T\}$, we have $X_{\xi(t)}\in\mathcal C$, and every vertex in $V^{(1)}$ has not been visited before time $\xi(t)$. Hence, conditionally on $\mathcal F_{\xi(t)}$, the total rate of jumps from $X_{\xi(t)}$ to $V^{(1)}$ is at least $w(\ell_*)$. Therefore, on $\{\xi(t)<T\}$,
\begin{align*}
\P\big(T\le \xi(t)+h\mid\mathcal F_{\xi(t)}\big) &\ge \frac{w(\ell_*)} {\sum_{y\sim X_{\xi(t)}}w(L(y,\xi(t)))} \left( 1-e^{-h\sum_{y\sim X_{\xi(t)}}w(L(y,\xi(t)))} \right)\\
&\ge \frac{w(\ell_*)h} {1+h\sum_{y\sim X_{\xi(t)}}w(L(y,\xi(t)))}.
\end{align*}
Combining this bound with \eqref{inq2} and using the dominated convergence theorem, we get
\begin{align}\label{sub.dR}
 \liminf_{h\downarrow0}
\frac{R(t)-R(t+h)}{h}
\ge
w(\ell_*)R(t).   
\end{align}
Set $G(t):=e^{w(\ell_*)t}R(t)$. Since $R$ is non-increasing and right-continuous, $G$ is lower semicontinuous, and it follows from \eqref{sub.dR} that
\begin{align}\label{sub.dG}
    \limsup_{h\downarrow0}\frac{G(t+h)-G(t)}{h}\le0.
\end{align}
Assume that $G(s)<G(t)$ for some $0\le s<t$. Choose $0<\varepsilon<\frac{G(t)-G(s)}{t-s}$. The lower semicontinuous function $G(u)-\varepsilon u$ attains its minimum on $[s,t]$ at some $u_0<t$. Consequently,
$$\liminf_{h\downarrow0}\frac{G(u_0+h)-G(u_0)}{h}\ge\varepsilon,
$$
which contradicts \eqref{sub.dG}. Thus $G$ is non-increasing, and therefore
$R(t)\le e^{-w(\ell_*)t}R(0)\le e^{-w(\ell_*)t}.$

 \end{proof}

We now turn to the proof of Theorem~\ref{them:loc}.

 \begin{proof}[Proof of Theorem~\ref{them:loc}]
For each $k\in\Z_+$, set
\begin{align}\nonumber
V^{\ssup 1}_k&=\mathcal B_{k+1}(\rho)^c\setminus \left\{ x\in \partial \mathcal{B}_{k+1}(\rho) :   
x\text{ is not incident to any vertex in $\mathcal B_{k+2}(\rho)^c$}\right\} ,\\ \nonumber
\mathcal C_k&=\big\{x\in V\setminus V_k^{\ssup 1}:x\sim y
\text{ for some }y\in V_k^{\ssup 1}\big\},\\
\label{def.V}
 V^{\ssup 0}_k&=(\mathcal{C}_k\cup V^{\ssup 1}_k )^c.
\end{align}
Then $V^{\ssup 0}_k$ is finite and contains $\rho$, while $V^{\ssup 1}_k$ is infinite. Moreover, there is no edge between $V^{\ssup 0}_k$ and $V^{\ssup 1}_k$, and every vertex in $\mathcal C_k$ has a neighbour in $V^{\ssup 1}_k$. Hence $(\Ccal_k)_{k\in\Z_+}$ is a sequence of cut-sets according to Definition~\ref{def.cutset}. The stopping time $T_k$ defined in \eqref{hitting} is the first hitting time of $V^{\ssup 1}_k$.
On $\{T_k<\infty\}$, let $T_k^-$ and $T_k^+$ be respectively the preceding and succeeding jumping times of $T_k$. In other words,
$$
T_k^- = \sup\{\tau_j \colon \tau_j < T_k\} \quad\text{and}\quad
T_k^+ = \inf\{\tau_j \colon \tau_j > T_k\}.
$$
Notice that $X_{T_{k}^-}\in\mathcal{C}_{k}$, while $X_{T_{k}}\in \partial\mathcal{B}_{k+1}(\rho)$ and is incident to a vertex in $\partial\mathcal{B}_{k+2}(\rho)$ which has never been visited by time $T_{k}$. Let 
\begin{align*}
A_k&= \{L(\mathcal C_k, T_{k})-L(\mathcal C_k, 0)\le \ell^*\},\\
B_k&=\{X_{T^+_k}\in \partial\mathcal B_{k+2}(\rho), L(X_{T_k}, T^+_k)\le 2\ell^*\}.
\end{align*}
In virtue of Proposition~\ref{cutset}, we have
$$
\P\big( A_k \mid T_k<\infty \big) \ge 1-e^{-w(\ell_*)\ell^*}.
$$
On $\{T_k<\infty\}\cap A_k$, every neighbour $u$ of $X_{T_k}$ satisfies $L(u,T_k)\le 2\ell^*.$
Since $X_{T_k}$ has a neighbour in $\partial\mathcal B_{k+2}(\rho)$ which has not been visited before $T_k$, we obtain
$$
\P\big( B_k \mid \{T_k<\infty\}\cap A_k \big) \ge \frac{w(\ell_*)}{d w(2\ell^*)} \left(1-e^{-w(\ell_*)\ell^*}\right).
$$
On $\{T_k<\infty\}\cap A_k\cap B_k$, we have
$\max_{u\in\mathcal B_2(X_{T_k^+})}L(u,T_k^+)\le2\ell^*.$
Therefore, by the strong Markov property of $(X_t,L(t))_{t\ge 0}$ and Proposition~\ref{lem:loc1},
\begin{align*}
\P\big( X_t \in \mathcal{B}_1(X_{T_k^+}), \forall t\ge T_k^+ \mid \{T_k<\infty\} \cap A_k\cap B_k \big)\ge p(2\ell^*)>0.
\end{align*}

On $B_{3k}$, we have
$X_{T_{3k}^+}\in\partial\mathcal B_{3k+2}(\rho).$
Hence, if the process remains in $\mathcal B_1(X_{T_{3k}^+})$ after time $T_{3k}^+$, then $T_{3(k+1)}=\infty$. We thus have
$$
\P\left(T_{3(k+1)} = \infty \mid T_{3k}<\infty \right) \ge  \frac{w(\ell_*)}{d w(2\ell^*)} \left(1-e^{-w(\ell_*)\ell^*}\right)^2 p(2\ell^*) =:1-\gamma\in (0,1).
$$
Hence
\begin{align*}
\P(T_{3k}<\infty) &=  \P(T_{3k}<\infty, T_{3(k-1)}<\infty)= \P(T_{3k}<\infty\mid T_{3(k-1)}<\infty) \P(T_{3(k-1)}<\infty)\\
&\le \gamma \P(T_{3(k-1)}<\infty) \le  \gamma^k.
\end{align*}
By the Borel--Cantelli lemma, there almost surely exists $k_0$ such that $T_{3k_0}=\infty$. Since $T_{3k_0}$ is the first hitting time of $V_{3k_0}^{\ssup 1}$, we have a.s.
$$
X_t\in \mathcal C_{3k_0}\cup V_{3k_0}^{\ssup 0} \quad\text{for all }t\in\R_+.
$$
This completes the proof of the theorem.
\end{proof}

\section{Some preliminary results on martingales and Markov chains}

\subsection{Non-convergence theorem for semimartingales}\label{sec:nonconv}
For a càdlàg real-valued process $Z=(Z_t)_{t\in \R_+}$, the \textit{total variation} on an interval $(s,t]$ is denoted $\mathbf{V}(Z,(s,t])$ and defined by
$$
\mathbf{V}(Z,(s,t]) = \sup_{\mathcal{P}} \sum_{i=1}^n |Z_{t_i} - Z_{t_{i-1}}|,
$$
where the supremum is taken over all partitions $\mathcal{P} = \{s = t_0 < t_1 < \cdots < t_n = t\}$ of $(s,t]$.  If $\mathbf{V}(Z,(0,t]) < \infty$ with probability one for every $t$, we say that $Z$ has \textit{finite variation}.

Let $(Z_t)_{t\in\R_+}$ be a c\`adl\`ag finite variation process. For each $t>0$, let $Z_{t-}=\lim_{s\uparrow  t}Z_s$ and $\Delta Z_t=Z_t-Z_{t-}$ be respectively the left limit and the size of the jump of $Z$ at time $t$. 

Let $(M_t)_{t\in\R_+}$ be a c\`adl\`ag locally square-integrable local martingale. We denote by $\langle M\rangle$ the \textit{angle bracket quadratic variation} of $(M_t)_{t\in\R_+}$, i.e. the unique right-continuous predictable increasing process such that $\langle {  M}\rangle_0=0$ and ${ M}^2_t-\langle { M}\rangle_t$ is a local martingale.  For $s<t$, we also write $\langle M\rangle_{s,t}=\langle M\rangle_t-\langle M\rangle_s$.

In Section \ref{sec:f.graph}, we will apply the following non-convergence theorem to strongly VRJPs on finite graphs.

\begin{proposition}[Raimond-Tarres 2023, \cite{RT2023}, Theorem 2.1.1]\label{thm:RT_specialized}
Let $(\mathcal F_t)_{t\ge0}$ be a complete right-continuous filtration, and let
$(Z_t)_{t\ge0}$ be a c\`adl\`ag real-valued process of the form
$$
Z_t-Z_0=\int_0^t \Psi_s\rmd s+M_t+R_t,
$$
where $\Psi$ is progressively measurable, $M$ is a c\`adl\`ag locally square-integrable martingale, and $R$ is a c\`adl\`ag adapted process of finite variation. Let $\Gamma\in\mathcal F_\infty$. Let
$\alpha:[0,\infty)\to(0,\infty)$ be a non-increasing c\`adl\`ag function such that
$\lim_{t\to\infty}\alpha(t)=0$, and let
$\tau:[0,\infty)\to[0,\infty)$ be a non-decreasing c\`adl\`ag function such that
$$
\tau(t)>t\quad\text{for all }t>0
\quad\text{and}\quad
\lim_{n\to\infty}\tau^n(t)=\infty
\quad\text{for all }t>0.
$$
Assume that there exists a constant $\rho>0$, a finite random time $T_0$, and a nonnegative c\`adl\`ag adapted process $(\varsigma_t)_{t\ge0}$ such that, a.s. on $\Gamma$,
\begin{itemize}
\item[(i)]
$Z_{t-} \Psi_t\ge0\quad\text{for all $t\ge T_0$ such that $|Z_{t-}|<\rho$ }$;

\item[(ii)]
$
\alpha(t)^2-\alpha(\tau(t))^2=O\big(\langle M\rangle_{t,\tau(t)}\big)
$ as $t\to\infty$;

\item[(iii)]
$
\langle M\rangle_{t,\infty}= O\big(\alpha(t)^2\big)
$ as $t\to\infty$; 

\item[(iv)]
$\mathbf{V}(R,(t,\infty))=o\big(\alpha(t)\big)$ as $t\to\infty$;

\item[(v)]
for every $t\ge0$ and every stopping time $S\ge t$,
$$
\E\big[(\Delta M_S)^2\1_{\{S<\infty\}}\mid\mathcal F_t\big] \le \varsigma_t\,\P(S<\infty\mid\mathcal F_t),
\quad\text{and}\quad
\varsigma_t=O\big(\alpha(t)^2\big)\quad\text{as $t\to\infty$}.
$$
\end{itemize}
Then
$$
\mathbb P(\Gamma\cap\{Z_t\to0\})=0.
$$
\end{proposition}

It is worth mentioning that there are similar versions of non-convergence theorems for continuous-time processes which have been recently studied in \cite{NR2021} and \cite{CNV2022}.

\subsection{Solution to Poisson's equation of an irreducible Markov chain}\label{sec:poiseq}
Assume that $H=(H_{ij})_{1\le i,j\le n}$ is the infinitesimal generator of an irreducible continuous-time Markov chain on the state space $V=\{1,2,\dots, n\}$ with $n\ge 2$. Denote by $\pi=(\pi_j)_{1\le j\le n}$ the invariant probability measure of $H$, i.e. $\pi H=0$ and $\pi\cdot\mathbf{1}^{\textsf{T}}=1$, where $\mathbf{1}$ is the row vector with all unit entries. Define $\Pi:=\mathbf{1} ^{\textsf{T}}\cdot \pi$, i.e. the matrix where each row is equal to $\pi$.
Let $Q=(Q_{ij})_{1\le i, j\le n}$ be the $n\times n$ matrix given by 
\begin{equation}
Q :=\int_{0}^{\infty}(\Pi-e^{tH})\rmd t.
\end{equation}
We note that
$$
P(t)=e^{tH}:=\sum_{m=0}^{\infty}\frac{(tH)^m}{m!}
$$
converges towards $\Pi$ at an exponential rate. Hence
$Q$ is well-defined. Moreover, $Q$ is a solution to the Poisson equation
\begin{equation}
I-\Pi= QH=HQ,
\end{equation}
and $Q\cdot \mathbf{1}^{\textsf T}=\mathbf{0}$.

\begin{proposition}\label{Pois.thm}
For $i, j \in V$, we have
$$
Q_{ij}=
\displaystyle \pi_j ({\pi}^{(j)}-\mathbf{e}_i^{(j)})\big(H^{(j)}\big)^{-1}\mathbf{1}^{\textsf T},
$$
where ${\pi}^{(j)} =(\pi_{k})_{k\in V\setminus\{j\}}$ and $\mathbf{e}_i^{(j)}=(\1_{k=i})_{k\in V\setminus\{j\}}$ are the row vectors obtained by deleting the $j$-th entry from $\pi$ and $\mathbf{e}_i$ respectively; and $H^{(j)}=(H_{k,h})_{k,h\in V\setminus\{j\}}$ is the matrix obtained by deleting the $j$-th column and the $j$-th row from $H$. In particular, we have
$$
Q_{jj}=
\displaystyle \pi_j {\pi}^{(j)}\big(H^{(j)}\big)^{-1}\mathbf{1}^{\textsf T}.
$$
\end{proposition}

\begin{proof}
Let $K_j\in\R^{n\times n}$ be the matrix such that all the entries on the $j$-th column and the main diagonal are equal to 1 while the others are 0, and let $S_j\in\R^{n\times n}$ be the matrix obtained by swapping the first column and the $j$-th column of the identity matrix, i.e.
\[
K_j=\begin{pmatrix}
1 & 0 & \cdots & 1 & \cdots & 0\\
0 & 1 & \cdots & 1 & \cdots & 0\\
\vdots & \vdots & \ddots & 1 & \ddots & \vdots\\
0 & 0 & \cdots & 1 & \cdots & 0\\
0 & 0 & \cdots & 1 & \cdots & 1
\end{pmatrix}
\quad\text{and}\quad
S_j=\begin{pmatrix}
0 & 0 & \cdots & 1 & \cdots & 0\\
0 & 1 & \cdots & 0 & \cdots & 0\\
\vdots & \vdots & \ddots & \vdots & \ddots & \vdots\\
1 & 0 & \cdots & 0 & \cdots & 0\\
\vdots & \vdots & \ddots & \vdots & \ddots & \vdots\\
0 & 0 & \cdots & 0 & \cdots & 1
\end{pmatrix}.
\]
We have $S_j^{-1}=S_j$, while $K_j^{-1}$ is obtained from the identity matrix by replacing every off-diagonal entry in the $j$-th column by $-1$.

In the computations below, we write $H^{(j)}$, $D_j$, $\pi^{(j)}$ and $\mathbf e_i^{(j)}$ with the indices in $V\setminus\{j\}$ ordered as $(2,\ldots,j-1,1,j+1,\ldots,n)$ for $j\ge2$, and as $(2,\ldots,n)$ for $j=1$.

Notice that for a matrix $A\in\R^{n\times n}$, $K_j^{-1}AK_j$ is the matrix obtained from $A$ by adding all other columns to the $j$-th column and then subtracting the $j$-th row from all other rows. Moreover, $S_j^{-1}AS_j$ is the matrix obtained from $A$ by swapping the $j$-th column with the first column and then swapping the $j$-th row with the first row. We thus have
\begin{align}\label{H.trans}
S_j^{-1}K_j^{-1}HK_jS_j
=\begin{pmatrix}
0 & D_j\\
\mathbf0 & H^{(j)}-\mathbf1^{\textsf T}\cdot D_j
\end{pmatrix},
\end{align}
where $\mathbf0\in\R^{(n-1)\times1}$ is the zero column vector and $D_j=(H_{jk})_{k\in V\setminus\{j\}}$. Since $H$ is the infinitesimal generator of an irreducible finite-state Markov chain, $0$ is a simple eigenvalue of $H$. It follows from \eqref{H.trans} that $H^{(j)}-\mathbf1^{\textsf T}\cdot D_j$ is invertible and that all its eigenvalues have strictly negative real parts.

Applying Lemma~\ref{MatExp} (see Appendix) to the block matrix given by \eqref{H.trans}, we obtain
$$
e^{tH}=K_jS_j\begin{pmatrix}
1 & D_j(e^{t(H^{(j)}-\mathbf1^{\textsf T}\cdot D_j)}-I_{n-1})(H^{(j)}-\mathbf1^{\textsf T}\cdot D_j)^{-1}\\
\mathbf0 & e^{t(H^{(j)}-\mathbf1^{\textsf T}\cdot D_j)}
\end{pmatrix}S_j^{-1}K_j^{-1}.
$$
On the other hand,
$$
\Pi=K_jS_j\begin{pmatrix}
1 & \pi^{(j)}\\
\mathbf0 & O
\end{pmatrix}S_j^{-1}K_j^{-1},
$$
where $\pi^{(j)}=(\pi_i)_{i\in V\setminus\{j\}}$ and $O$ is the $(n-1)\times(n-1)$ matrix with all zero entries. Notice that
$$
\pi^{(j)}H^{(j)}=-\pi_jD_j=(\pi^{(j)}\mathbf1^{\textsf T}-1)D_j
$$
and thus
\begin{equation}\label{eq:pij}
\pi^{(j)}=-D_j\cdot(H^{(j)}-\mathbf1^{\textsf T}\cdot D_j)^{-1}.
\end{equation}
Hence
\begin{align*}
Q&=\int_0^\infty(\Pi-e^{tH})\rmd t\\
&=K_jS_j\begin{pmatrix}
0 & -\pi^{(j)}\cdot(H^{(j)}-\mathbf1^{\textsf T}\cdot D_j)^{-1}\\
\mathbf0 & (H^{(j)}-\mathbf1^{\textsf T}\cdot D_j)^{-1}
\end{pmatrix}S_j^{-1}K_j^{-1}\\
&=K_jS_j\begin{pmatrix}
\pi^{(j)}\cdot(H^{(j)}-\mathbf1^{\textsf T}\cdot D_j)^{-1}\cdot\mathbf1^{\textsf T}
& -\pi^{(j)}\cdot(H^{(j)}-\mathbf1^{\textsf T}\cdot D_j)^{-1}\\
-(H^{(j)}-\mathbf1^{\textsf T}\cdot D_j)^{-1}\cdot\mathbf1^{\textsf T}
& (H^{(j)}-\mathbf1^{\textsf T}\cdot D_j)^{-1}
\end{pmatrix}S_j^{-1}\\
&=S_jK_1\begin{pmatrix}
\pi^{(j)}\cdot(H^{(j)}-\mathbf1^{\textsf T}\cdot D_j)^{-1}\cdot\mathbf1^{\textsf T}
& -\pi^{(j)}\cdot(H^{(j)}-\mathbf1^{\textsf T}\cdot D_j)^{-1}\\
-(H^{(j)}-\mathbf1^{\textsf T}\cdot D_j)^{-1}\cdot\mathbf1^{\textsf T}
& (H^{(j)}-\mathbf1^{\textsf T}\cdot D_j)^{-1}
\end{pmatrix}S_j^{-1}\\
&=S_j\begin{pmatrix}
\pi^{(j)}\cdot(H^{(j)}-\mathbf1^{\textsf T}\cdot D_j)^{-1}\cdot\mathbf1^{\textsf T}
& -\pi^{(j)}\cdot(H^{(j)}-\mathbf1^{\textsf T}\cdot D_j)^{-1}\\
-(I_{n-1}-\Pi^{(j)})(H^{(j)}-\mathbf1^{\textsf T}\cdot D_j)^{-1}\cdot\mathbf1^{\textsf T}
& (I_{n-1}-\Pi^{(j)})(H^{(j)}-\mathbf1^{\textsf T}\cdot D_j)^{-1}
\end{pmatrix}S_j^{-1},
\end{align*}
where $\Pi^{(j)}:=\mathbf1^{\textsf T}\cdot\pi^{(j)}$, and in the fourth identity we used the fact that $K_jS_j=S_jK_1$. It follows that
$$
Q_{ij}=(\pi^{(j)}-\mathbf e_i^{(j)})(H^{(j)}-\mathbf1^{\textsf T}\cdot D_j)^{-1}\cdot\mathbf1^{\textsf T}.
$$
We next notice that $H^{(j)}$ is invertible. Indeed, if $H^{(j)}x=0$, extend $x$ to a vector on $V$ by setting its $j$-th entry equal to $0$. If this vector has a positive maximum at $i\ne j$, then
$$
0=\sum_{h\ne i}H_{ih}(x_h-x_i),
$$
and hence $x_h=x_i$ whenever $H_{ih}>0$. By irreducibility, this implies $x_j=x_i$, a contradiction. Applying the same argument to a negative minimum, we have $x=0$.

We have
$$
(H^{(j)}-\mathbf1^{\textsf T}\cdot D_j)(H^{(j)})^{-1}\mathbf1^{\textsf T}=\left(1-D_j(H^{(j)})^{-1}\mathbf1^{\textsf T}\right)\mathbf1^{\textsf T}.
$$
Since both $H^{(j)}$ and $H^{(j)}-\mathbf1^{\textsf T}\cdot D_j$ are invertible, we have $1-D_j(H^{(j)})^{-1}\mathbf1^{\textsf T}\ne0$. Using the Sherman--Morrison formula (see Lemma \ref{SM.formula} in Appendix), we have that
$$
(H^{(j)}-\mathbf1^{\textsf T}\cdot D_j)^{-1}=(H^{(j)})^{-1}+\frac{(H^{(j)})^{-1}\mathbf1^{\textsf T}D_j(H^{(j)})^{-1}}{1-D_j(H^{(j)})^{-1}\mathbf1^{\textsf T}}.
$$
We thus get
$$
(H^{(j)}-\mathbf1^{\textsf T}\cdot D_j)^{-1}\mathbf1^{\textsf T}=\frac{(H^{(j)})^{-1}\mathbf1^{\textsf T}}{1-D_j(H^{(j)})^{-1}\mathbf1^{\textsf T}}.
$$
Furthermore, using \eqref{eq:pij},
$$
\pi^{(j)}=-D_j\cdot(H^{(j)}-\mathbf1^{\textsf T}\cdot D_j)^{-1}=\frac{-D_j(H^{(j)})^{-1}}{1-D_j(H^{(j)})^{-1}\mathbf1^{\textsf T}}.
$$
Hence
$$
1-D_j(H^{(j)})^{-1}\mathbf1^{\textsf T}=(1-\pi^{(j)}\cdot\mathbf1^{\textsf T})^{-1}=\pi_j^{-1}.
$$
It follows that
\begin{align*}
Q_{ij}
&=\frac{(\pi^{(j)}-\mathbf e_i^{(j)})(H^{(j)})^{-1}\mathbf1^{\textsf T}}{1-D_j(H^{(j)})^{-1}\mathbf1^{\textsf T}}
=\pi_j(\pi^{(j)}-\mathbf e_i^{(j)})(H^{(j)})^{-1}\mathbf1^{\textsf T}.
\end{align*}
\end{proof}

\section{Strongly VRJP on a finite graph}\label{sec:f.graph}

Throughout this section, we assume that the graph $\mathcal G=(V,E)$  is finite.

\subsection{Martingale decomposition for VRJP}
We always assume throughout this subsection that Assumption \ref{assump} is fulfilled and the function $w$ is continuously differentiable. Let $\R_{*}^{V}$ be the set of $|V|$-dimensional vectors $z=(z_v)_{v\in V}$  such that $z_v\ge \ell_*$ for each $v\in V$, where we recall that $\ell_*=\min_{v\in V} \ell_v$.

Set $R(t)=(R_j(t))_{j\in V}$, where for each $j\in V$ and $t\in\R_+$,
\begin{align}
\label{eq.Ht} R_{j}(t)& :=\int_{\ell_j}^{L(j,t)} \frac{\rmd u}{w(u)}-\int_{0}^t \frac{\rmd u}{\sum_{j\in V} w(L(j,u))}.
\end{align}
Notice that if there exists $i,j\in V$ such that $L(i,\infty)=L(j,\infty)=\infty$ then we must have $$R_i({\infty})-R_j({\infty})=\int_{\ell_i}^{\infty}\frac{\rmd u}{w(u)}-\int_{\ell_j}^{\infty}\frac{\rmd u}{w(u)}.$$
In this subsection we establish a martingale decomposition for $(R_i(t)-R_j(t))_{t\in \R_+}$ for each pair of vertices $i,j\in V$.

For each $z=(z_j)_{j\in V}\in \mathbb R_{*}^{V}$, 
let $H(z)=(H_{i,j}(z))_{i,j\in V}$ be an infinitesimal generator matrix defined by
\begin{align}H_{i,j}(z)=\left\lbrace\begin{matrix}{\1}_{j \sim i} w(z_j) & \text{if } j\neq i; \\ \displaystyle - \sum_{k\sim i} w(z_k) & \text{if } j=i.\end{matrix}\right.
\end{align}

Set $L(t)=(L(j,t))_{j\in V }$. Note that $(X_t, L(t))_{t\ge 0}$ is a continuous-time Markov process on the state space $V\times \mathbb R_{*}^{V}$. Let $(\mathcal{T}_{t})_{t\in\R_+}$ and $\mathcal{L}$  be respectively the semigroup associated with this Markov process and its infinitesimal generator. Recall that for each  bounded continuously differentiable function $f:V\times \mathbb{R}_{*}^{V}\to \R$, $k\in V$ and $z\in \R_{*}^{V}$,
\begin{align*}(\mathcal{T}_{t}f)(k,z)&=\E[f(X_t,L(t))|X_0=k,L(0)=z]\quad\text{and}\\
\mathcal{L} f &=\lim_{t\downarrow 0}\frac{\mathcal{T}_tf- f}{t}.\end{align*}

For each $k\in V$ and $z\in \R_{*}^{V}$, we have
\begin{equation}\label{eq:operator}
\begin{aligned}(\mathcal{L} f)(k, z)&=\left(\frac{\partial }{\partial z_k}f\right)(k,z)+\sum_{h\in V} H_{k,h}(z)f(h,z)\\
&=\left(\frac{\partial }{\partial z_k}f\right)(k,z)+\sum_{h\sim k} w(z_h)\big(f(h,z)-f(k,z)\big).\end{aligned}
\end{equation}

For each $z\in\mathbb{R}^{V}_*$, let $\pi(z):=\left(\pi_j(z)\right)_{j\in V}$ with $$\pi_j(z):=\frac{w(z_j)}{\sum_{j\in V} w(z_j)}.$$ Notice that $\pi(z)H(z)=0$ and $\sum_{j\in V} \pi_j(z) =1$. We also define $\Pi(z)=(\Pi_{ij}(z))_{i, j\in  V}$ with $\Pi_{ij}(z)=\pi_j(z)$. Set
\begin{align}\label{sol.Q}
    Q(z):=
    \int_0^\infty \left(\Pi(z)-I\right)e^{s H(z)} \rmd s=
    \int_0^\infty \left(\Pi(z)-e^{s H(z)}\right) \rmd s,\end{align}
where $I$ is the identity matrix in $\R^{V\times V}$.   Recall from Section \ref{sec:poiseq} that  $Q(z)=(Q_{i,j}(z))_{i, j\in  V}$ is a solution to the following Poisson's matrix equation
\begin{equation}\label{eq.Poisson}I-\Pi(z)=H(z)Q(z)=Q(z)H(z).
\end{equation}

For each ordered pair of vertices  $i,j\in V$, not necessarily neighbours, we define function $\phi_{i,j}:  V\times \mathbb{R}_{*}^{V}\to \R$ such that 
\begin{align}\label{def.phi}\phi_{i,j} (k,z) := \frac{Q_{k,i}(z)}{w(z_i)} -\frac{Q_{k,j}(z)}{w(z_j)} .\end{align} 
Recall the notations:  $$\mathbf{e}_{h}^{(j)}=(\1_{k=h})_{k\in V\setminus\{j\}},\quad \pi^{(j)}(z)=(\pi_k(z))_{k\in V\setminus\{j\}}\quad\text{and}\quad H^{(j)}(z)=(H_{k,h}(z))_{k,h\in V\setminus\{j\} }.$$
For $t\in\R_+$, define
$$ W(t):= \sum_{i \in V} w(L(i, t)).
$$

For $i,j,k,h\in V$, $z\in \R_{*}^{V}$ and $t\in\R_{+}$ define
\begin{align} M_{i,j}(t)&:=
\phi_{i,j}(X_t,L(t))- \phi_{i,j}(X_0, \ell)-\int_0^t \mathcal{L} \phi_{i,j}(X_u, L(u))\rmd u \quad\text{and}\\
A_{ij}(t)&:=\int_0^t
\left.\frac{\partial \phi_{ij}(k,z)}{\partial z_k}\right|_{k=X_u, z=L(u)}\rmd u\\
\label{kappa.def}
\kappa_{i,j,h}(k,z)&:=(\mathbf{e}_{k}^{(i)}-\mathbf{e}_h^{(i)})\cdot\left(H^{(i)}(z)\right)^{-1}\cdot\mathbf{1}^{\textsf T}-(\mathbf{e}_{k}^{(j)}-\mathbf{e}_h^{(j)})\cdot\left(H^{(j)}(z)\right)^{-1}\cdot\mathbf{1}^{\textsf T} .
\end{align}

\begin{proposition}\label{mart.M} 
\noindent {\bf a.} For  each $i, j\in V$,  $(M_{i,j}(t))_{t\in\R_+}$ is a càdlàg locally square-integrable martingale  and its angle bracket quadratic variation is given by
\begin{align*}
   \langle M_{i,j}\rangle_t  = \int_{0}^t\Lambda_{i,j}(u)\rmd u,
\end{align*}
where  
$$\Lambda_{i,j}(t):=\sum_{h\sim X_t}\frac{ w(L(h,t))}{W(t)^2} \kappa_{i,j,h}(X_t,L(t))^2.$$
\noindent {\bf b.} For  each $i, j\in V$ and $t\in \R_+$, we have that $A_{ij}(t)$ is equal to
\begin{align*}
    \int_0^t\sum_{h\in V}\frac{ w(L(h,u))}{W(u)^2}\Big( \frac{2w'(L(X_u,u)) }{W(u)}\kappa_{i,j,h}(X_u,L(u))-\frac{\partial \kappa_{i,j,h}(k,z)}{\partial z_k}\Big|_{k=X_u, z=L(u)} \Big) {\rm d} u.
\end{align*}
\end{proposition}

\begin{proof}
{\bf a.} Using Dynkin's formula
and the carré du champ identity (see, e.g., Section~1.4 in \cite{BGL2014}), we notice that $(M_{i,j}(t))_{t\in\R_+}$ is a càdlàg locally square-integrable martingale and its angle bracket quadratic variation is given by
\begin{align*}\langle M_{i,j}\rangle_t
= \int_0^t \left( \mathcal{L} \phi_{i,j}^2-2\phi_{i,j}\mathcal{L} \phi_{i,j}\right)(X_u,L(u))\rmd u.\end{align*}
Using \eqref{eq:operator}, we obtain
\begin{align*}(\mathcal{L}\phi_{i,j}^2)(k,z)&=\left(H(z)\phi_{i,j}^2\right)(k,z)+2\phi_{i,j}(k,z)\frac{\partial \phi_{i,j}(k,z)}{\partial z_k}\quad\text{and}\\
(\phi_{i,j}\mathcal{L}\phi_{i,j})(k,z)&=\phi_{i,j}(k,z)\big( H(z)\phi_{i,j}\big)(k,z)+\phi_{i,j}(k,z)\frac{\partial \phi_{i,j}(k,z)}{\partial z_k}.
\end{align*}
Hence 
\begin{align*}&\left( \mathcal{L} \phi_{i,j}^2-2\phi_{i,j}\mathcal{L} \phi_{i,j}\right) (k,z) =
    \big(H(z)\phi_{i,j}^2\big)(k,z)-2\phi_{i,j}(k,z)\big(H(z)\phi_{i,j}\big)(k,z)\\
    & = \sum_{h\sim k} w(z_h) \left( \Big(\frac{Q_{h,i}(z)}{w(z_i)} -\frac{Q_{h,j}(z)}{w(z_j)}\Big)^2-\Big(\frac{Q_{k,i}(z)}{w(z_i)} -\frac{Q_{k,j}(z)}{w(z_j)}\Big)^2\right) \\
    &- 2 \Big(\frac{Q_{k,i}(z)}{w(z_i)} -\frac{Q_{k,j}(z)}{w(z_j)}\Big) \sum_{h\sim k}w(z_h) \left( \frac{Q_{h,i}(z)}{w(z_i)} -\frac{Q_{h,j}(z)}{w(z_j)}-\frac{Q_{k,i}(z)}{w(z_i)} +\frac{Q_{k,j}(z)}{w(z_j)}\right)\\
    & = \sum_{h\sim k} w(z_{h}) \left(\frac{Q_{h,i}(z)}{w(z_i)} -\frac{Q_{h,j}(z)}{w(z_j)}-\frac{Q_{k,i}(z)}{w(z_i)} +\frac{Q_{k,j}(z)}{w(z_j)} \right)^2.
\end{align*}
On the other hand, in virtue of Proposition \ref{Pois.thm}, we have
\begin{align}\label{Qz}
    Q_{ij}(z)=
\displaystyle \pi_j(z) ({\pi}^{(j)}(z)-\mathbf{e}_i^{(j)})\left(H^{(j)}(z)\right)^{-1}\mathbf{1}^{\textsf T}.   
\end{align}
Combining \eqref{Qz} with the fact that  $\pi_k(z)={w(z_k)}/{\sum_{v\in V} w(z_v)}$ for each $k\in V$, we obtain
$$\frac{Q_{h,i}(z)-Q_{k,i}(z)}{w(z_i)}=\frac{1}{\sum_{v\in V} w(z_v)} (\mathbf{e}_{k}^{(i)}-\mathbf{e}_h^{(i)})\cdot\left(H^{(i)}(z)\right)^{-1}\cdot\mathbf{1}^{\textsf T}.$$
Therefore $\langle M_{i,j}\rangle_t$ is equal to
\begin{align*}
&\int_0^t \sum_{h\sim X_u}\frac{w(L(h,u))}{W(u)^2}\Big( (\mathbf{e}_{X_u}^{(i)}-\mathbf{e}_h^{(i)})\big(H^{(i)}(L(u))\big)^{-1}\mathbf{1}^{\textsf T}-(\mathbf{e}_{X_u}^{(j)}-\mathbf{e}_h^{(j)})\big(H^{(j)}(L(u))\big)^{-1}\mathbf{1}^{\textsf T}\Big)^2\rmd u. \end{align*}
{\bf b.} Using \eqref{Qz} and the fact that  $\pi_k(z)={w(z_k)}/{\sum_{v\in V} w(z_v)}$ for each $k\in V$, we have
\begin{align}
  &\nonumber \phi_{ij}(k,z)= \frac{Q_{k,i}(z)}{w(z_i)}-\frac{Q_{k,j}(z)}{w(z_j)}\\
 \nonumber & = \frac{ \big(  \pi^{(i)}(z)-\mathbf{e}_k^{(i)} \big)\big(H^{(i)}(z)\big)^{-1} \mathbf{1}^{\mathsf{T}}-\big(\pi^{(j)}(z)-\mathbf{e}_k^{(j)} \big)\big(H^{(j)}(z)\big)^{-1} \mathbf{1}^{\mathsf{T}}}{\sum_{v\in V} w(z_v)}\\
 \nonumber  & =\frac{\sum_{h\in V}w(z_h)\left( \big( \mathbf{e}_h^{(i)}-\mathbf{e}_k^{(i)} \big)\big(H^{(i)}(z)\big)^{-1} \mathbf{1}^{\mathsf{T}}-\big(\mathbf{e}_h^{(j)}-\mathbf{e}_k^{(j)} \big)\big(H^{(j)}(z)\big)^{-1} \mathbf{1}^{\mathsf{T}}\right)}{\big(\sum_{v\in V} w(z_v)\big)^2}\\
 \label{phi.def} & = \frac{-\sum_{h\in V\setminus\{k\}}w(z_h) \kappa_{i,j,h}(k,z)}{\big(\sum_{v\in V} w(z_v)\big)^2}.
\end{align}
Hence
\begin{align}\nonumber
    \frac{\partial \phi_{ij}(k,z)}{\partial z_k}&=
    \frac{2w'(z_k)}{\big(\sum_{v\in V} w(z_v)\big)^3}\sum_{h\in V\setminus\{k\}}w(z_h) \kappa_{i,j,h}(k,z)\\ \label{D.phi}
    &-\frac{1}{\big(\sum_{v\in V} w(z_v)\big)^2}\sum_{h\in V\setminus\{k\}}w(z_h)  \frac{\partial }{\partial z_k} \kappa_{i,j,h}(k,z).
\end{align}

\end{proof}

\begin{proposition}\label{decompostion}
For each pair of distinct vertices $i, j\in V$ and $t\in\R_+$, we have
\begin{align}R_i(t)-R_j(t)&=\phi_{ij}(X_t,L(t))-\phi_{ij}(X_0,\ell)-M_{ij}(t)-A_{ij}(t).\end{align}
\end{proposition}

\begin{proof}
By changing variable $u=L(i,t)$, we have
\begin{align*}R_i(t)&=\int_{\ell_i}^{L(i,t)}\frac{{\rm d}u}{w(u)}-\int_0^t \frac{{\rm d}u}{\sum_{k\in V} w(L(k,u))} =\int_0^t\frac{\1_{\{X_u=i\}}{\rm d}u}{w(L(i,u))}-\int_0^t \frac{{\rm d}u}{\sum_{k\in V} w(L(k,u))}\\
&=\int_0^t\frac{1}{w(L(i,u))}\big( I-\Pi(L(u))\big)_{X_u,i}{\rm d}u\\
&=\int_0^t  \frac{1}{w(L(i,u))}\sum_{k\in V} H_{X_u,k}(L(u)){Q_{k,i}(L(u))}{\rm d}u,
\end{align*}
where the second step is a result of  $\pi_i(z)=w(z_i)/{\sum_{v\in V} w(z_v)}$ for each $i\in V$, and the last step follows from the fact that $Q(z)$ is a solution of the Poisson equation \eqref{eq.Poisson}. Hence, by the definitions of  $\mathcal{L}$ and  $\phi_{ij}$ given respectively by \eqref{eq:operator} and \eqref{def.phi}, we obtain
\begin{align*}R_i(t)-R_j(t)& =\int_0^t  \sum_{k\in V} H_{X_u,k}(L(u))\left( \frac{Q_{k,i}(L(u))}{w(L(i,u))}- \frac{Q_{k,j}(L(u))}{w(L(j,u))}\right){\rm d}u\\
&= \int_0^t (\mathcal{L}\phi_{i,j})(X_u,L(u)){\rm d}u - \int_0^t
\left.\frac{\partial \phi_{ij}(k,z)}{\partial z_k}\right|_{k=X_u, z=L(u)}\rmd u\\
&=\phi_{ij}(X_t,L(t))-\phi_{ij}(X_0,\ell)-M_{ij}(t)-A_{ij}(t).
\end{align*}
\end{proof}

\subsection{Non-convergence theorem for strongly VRJP}
We assume throughout this subsection that Assumption \ref{assump} is fulfilled, the function $w$ is continuously differentiable and for each $C>0$ we have
\begin{align}\label{lower.regu}\liminf_{t\to\infty}\frac{w(Ct)}{w(t)}>0.\end{align}

For $i, j\in V$, $i\neq j$ and $t\in\R_+$, let $$Z_{i,j}(t):=M_{ij}(t)+A_{ij}(t).$$
Using Proposition \ref{decompostion}, we have
$$Z_{i,j}(t)=R_j(t)-R_i(t)+ \phi_{ij}(X_t,L(t))-\phi_{ij}(X_0,\ell).$$
By  applying Proposition \ref{thm:RT_specialized} to the process $(Z_{i,j}(t))_{t\ge0}$, with $\delta(i,j)\in \{1,2\}$, we show in this subsection that the event
$$\Gamma_{ij}:=\Big\{\liminf_{t\to\infty} \frac{L(i,t)\wedge L(j,t)}{t}>0\Big\}$$
occurs with probability 0. We then use this result to prove Theorem \ref{localization}.

We first verify the conditions of Proposition  \ref{thm:RT_specialized} by providing certain bounds for $\langle M_{ij}\rangle_t$ and $A_{ij}(t)$. For this purpose, we provide bounds for  $\kappa_{i,j,h}(k,z)$ which is defined in \eqref{kappa.def}. In what follows, we use the matrix-tree theorem for weighted directed graphs to derive a combinatorial representation of $\kappa_{i,j,h}(k,z)$. 

Denote by $\mathfrak{T}$ the set of unrooted spanning trees of $\mathcal{G}=(V,E)$. For each $j\in V$, let $\mathfrak{T}_j$ be the set of all spanning trees rooted at $j$ of $\mathcal{G}=(V,E)$. For two disjoint sequences of vertices $(v_j)_{1\le j\le p}$ and $(u_j)_{1\le j\le q}$, we  denote by $\mathfrak{T}_{v_1v_2\cdots v_p,u_1u_2\cdots v_q}$ the set of all spanning forests of $\mathcal G$ consisting of two components, in which the first component is rooted at $v_1$ and contains $(v_j)_{1\le j\le p}$ and the other is rooted at $u_1$ and contains $(u_j)_{1\le j\le q}$. 
For  example, $\mathfrak{T}_{kh,j}$ is the set of all spanning forests of $\mathcal G$ consisting of two components where the first component is rooted at $k$ and also contains $h$ while the other is rooted at $j$. Note that $\mathfrak{T}_{k,j}=\bigcup_{h\in V\setminus\{j\}}\mathfrak{T}_{kh,j}$ is the set of all spanning forests of $\mathcal{G}$ with two components rooted at $k$ and $j$.

Recall that $\mathbb{R}_{*}^{V}=\{(z_v)_{v\in V}\in \R^V:  z_v\ge \ell_* \text{ for each } v\in V\}$. Fix a vector $z\in \mathbb{R}_{*}^{V}$. For each spanning tree $T\in\mathfrak{T}_j$, set $$w(T, z):=w(z_j) \prod_{v\in V} w(z_v)^{\deg_T(v)-1},$$
and for each spanning forest $F\in \mathfrak{T}_{k,j}$, set
$$w(F,z):=w(z_k)w(z_j)\prod_{v\in V}  w(z_v)^{\deg_F(v)-1}.$$

\begin{proposition} \label{matrix.tree}
We have that 
$$\det(H^{(j)}(z))=(-1)^{|V|-1}\sum_{T\in\mathfrak{T}_j} w(T, z).$$
Furthermore, the cofactor matrix of $H^{(j)}(z)$, denoted by $C^{(j)}(z)=\big(C_{k,h}^{(j)}(z)\big)_{k,h\in V\setminus\{j\}}$, is given by 
 $$C_{k,h}^{(j)}(z)=(-1)^{|V|} \sum_{F\in\mathfrak{T}_{kh,j}}w(F, z).$$
\end{proposition}

\begin{proof}
Recall that $\vec{E}$ is the set of all directed edges induced from $E$. Assign to each directed edge $e=(u,v)$ a weight $w(e):=w(z_v)$. Notice that $-H(z)$ is the outgoing Laplacian matrix of the weighted directed graph $(V, \vec{E}, w)$. 
The result of the proposition follows directly from the matrix-tree theorem for weighted directed graphs. See Section \ref{sec:mat-tree} in Appendix.
\end{proof}

\begin{lemma}\label{lem.kappa1}
For each $z\in \mathbb{R}_{*}^{V}$ and $i,j,k, h\in V$ in which $i\sim j$, we have 
\begin{align}
\label{kappa1.up}|\kappa_{i,j,h}(k,z)|&\le  \frac{2\sum_{v\in V}w(z_v)}{w(z_i)w(z_j)},\\
\label{kappa1.low} |\kappa_{i,j,j}(i,z)|&\ge \frac{\sum_{v\in V}w(z_v)}{w(z_j)\sum_{s\in \mathcal{N}_j} w(z_s)}\quad\text{and}\\
\label{kappa1.prime}
\Big|\frac{\partial}{\partial z_k} \kappa_{i,j,h}(k,z)\Big|&\le 4d \frac{w'(z_k)\sum_{v\in V}w(z_v)}{w(z_k){w(z_i)w(z_j)}},
\end{align}
where we recall that $d=\max_{v\in V}\deg(v).$
\end{lemma}

\begin{proof} We first prove \eqref{kappa1.up}. Note that $H^{(j)}(z)^{-1} = \left( \det(H^{(j)}(z))\right)^{-1}C^{(j)}(z)^{\textsf T}$. Using Proposition \ref{matrix.tree}, we have that for $h,j \in V,$
\begin{align}\nonumber\mathbf{e}_h^{(j)}\cdot\left( H^{(j)}(z)\right)^{-1}\cdot\mathbf{1}^{\textsf T} & = \left( \det(H^{(j)}(z))\right)^{-1} \sum_{r\in V\setminus\{j\}} C_{r,h}^{(j)}(z)\\
 & =-  \frac{\sum_{r\in V\setminus\{j\}} \sum_{F\in\mathfrak{T}_{rh,j}}w(F, z) }{\sum_{T\in\mathfrak{T}_j} w(T, z)}.\label{frac}
\end{align}
For $h, k, j\in  V$, we thus get
\begin{align*}
&(\mathbf{e}_{k}^{(j)}-\mathbf{e}_h^{(j)})\cdot\left(H^{(j)}(z)\right)^{-1}\cdot\mathbf{1}^{\textsf T} = \frac{\displaystyle \sum_{r\in V\setminus\{j\}}\Big(\sum_{F\in\mathfrak{T}_{rh,j}} w(F, z)-\sum_{F\in\mathfrak{T}_{rk,j}} w(F, z)\Big)}{\displaystyle \sum_{T\in\mathfrak{T}_j}  w(T, z)} \\
&= \frac{\displaystyle \sum_{r\in V\setminus\{j\}} w(z_r)\Big(\sum_{F\in\mathfrak{T}_{rh,jk}}\prod_{v\in V}  w(z_v)^{\deg_F(v)-1}-\sum_{F\in\mathfrak{T}_{rk,jh}}\prod_{v\in V}  w(z_v)^{\deg_F(v)-1} \Big)}{\displaystyle \sum_{T\in\mathfrak{T}} \prod_{v\in V} w(z_v)^{\deg_T(v)-1}}.
\end{align*}
Hence, for any $h, k, i, j\in V$,
 \begin{align}\label{kappa.decomp}
   \nonumber \kappa_{i,j,h}(k,z)=& (\mathbf{e}_{k}^{(i)}-\mathbf{e}_h^{(i)})\cdot\left(H^{(i)}(z)\right)^{-1}\cdot\mathbf{1}^{\textsf T}-(\mathbf{e}_{k}^{(j)}-\mathbf{e}_h^{(j)})\cdot\left(H^{(j)}(z)\right)^{-1}\cdot\mathbf{1}^{\textsf T}\\
    & \nonumber = \frac{\displaystyle \sum_{r\in V} w(z_r)\Big[\sum_{F\in\mathfrak{T}_{rjh,ik}}-\sum_{F\in\mathfrak{T}_{rih,jk}}-\sum_{F\in\mathfrak{T}_{rjk,ih}}+\sum_{F\in\mathfrak{T}_{rik,jh}}\Big]\prod_{v\in V}  w(z_v)^{\deg_F(v)-1} }{\displaystyle \sum_{T\in\mathfrak{T}} \prod_{v\in V} w(z_v)^{\deg_T(v)-1}}\\
    & = \sum_{r\in V} w(z_r) \frac{\displaystyle \Big[\sum_{F\in\mathfrak{T}_{jh,ik}}-\sum_{F\in\mathfrak{T}_{ih,jk}}\Big]\prod_{v\in V}  w(z_v)^{\deg_F(v)-1} }{\displaystyle \sum_{T\in\mathfrak{T}} \prod_{v\in V} w(z_v)^{\deg_T(v)-1}}.
\end{align}
For each spanning forest $F=(V, E_F)\in \mathfrak{T}_{jh,ik} \cup \mathfrak{T}_{ih,jk}$, we note that $\{i,j\}$ is not an edge of $F$. Recall that $i\sim j$ on $\mathcal G$. By unrooting $F$ and connecting $i$ with $j$, we obtain an unrooted spanning tree $T=(V, E_T)\in \mathfrak{T}$. Notice also that
\begin{align*}
    \prod_{v\in V}  w(z_v)^{\deg_F(v)-1}=\frac{1}{w(z_i)w(z_j)} \prod_{v\in V}  w(z_v)^{\deg_T(v)-1}.
\end{align*}
Since the map $F\mapsto T$ is injective on
$\mathfrak{T}_{jh,ik}\cup\mathfrak{T}_{ih,jk}$, 
we have that $$|\kappa_{i,j,h}(k,z)| \le \frac{2\sum_{r\in V}w(z_r)}{w(z_i)w(z_j)},$$ 
which proves \eqref{kappa1.up}.

We next prove \eqref{kappa1.low}.  
Substitute $k=i$ and $h=j$ into \eqref{kappa.decomp}, we have
 \begin{align}\label{kappa.decomp22}
   \kappa_{i,j,j}(i,z) 
     = \frac{\displaystyle \sum_{r\in V} w(z_r)\sum_{F\in\mathfrak{T}_{i,j}}\prod_{v\in V}  w(z_v)^{\deg_F(v)-1} }{\displaystyle \sum_{T\in\mathfrak{T}} \prod_{v\in V} w(z_v)^{\deg_T(v)-1}}.
\end{align}
Let us compare the denominator and the numerator of the  fraction in \eqref{kappa.decomp22}. For each spanning tree $T\in\mathfrak T$, let $s\in\mathcal N_j$ be the neighbour of $j$ in the unique simple path connecting $j$ and $i$ in $T$. By deleting the edge $\{j,s\}$ and setting $i$ and $j$ respectively as the roots of the components containing $i$ and $j$, we obtain a spanning forest $F\in\mathfrak T_{i,j}$. The map
$T\mapsto (s,F)$ is injective, since $T$ is recovered from $(s,F)$ by adding the edge $\{j,s\}$.
We also notice that
$$
\prod_{v\in V}w(z_v)^{\deg_T(v)-1} =w(z_j)w(z_s)\prod_{v\in V}w(z_v)^{\deg_F(v)-1}.
$$
Therefore,
\begin{align*} 
\sum_{T\in\mathfrak T} \prod_{v\in V}w(z_v)^{\deg_T(v)-1} &\le w(z_j) \sum_{s\in\mathcal N_j}w(z_s) \sum_{F\in\mathfrak T_{i,j}} \prod_{v\in V}w(z_v)^{\deg_F(v)-1}.
\end{align*}
Combining this with \eqref{kappa.decomp22}, we obtain \eqref{kappa1.low}.

Next, we prove \eqref{kappa1.prime}. For $i,j, h, k\in V$ and $z\in \R_{*}^{V}$, set 
\begin{align*}
&\widetilde{\kappa}_{i,j,h}(k,z)
:= \frac{\displaystyle w(z_k)\Big[\sum_{F\in\mathfrak{T}_{jh,ik}}-\sum_{F\in\mathfrak{T}_{ih, jk}}\Big]\prod_{v\in V}  w(z_v)^{\deg_F(v)-1}}{ \displaystyle \sum_{T\in\mathfrak{T}} \prod_{v\in V} w(z_v)^{\deg_T(v)-1}}\\ 
&+ 
 \sum_{r\in V} w(z_r)\frac{\displaystyle\Big[\sum_{F\in\mathfrak{T}_{jh,ik}}-\sum_{F\in\mathfrak{T}_{ih,jk}}\Big](\deg_F(k)-1)\displaystyle\prod_{v\in V}  w(z_v)^{\deg_F(v)-1}}{ \displaystyle \sum_{T\in\mathfrak{T}} \prod_{v\in V} w(z_v)^{\deg_T(v)-1}}\\
& \text{and} \quad J(k,z):=\frac{\displaystyle \sum_{T\in\mathfrak{T}: \deg_T(k)\ge 2} (\deg_T(k)-1)\prod_{v\in V} w(z_v)^{\deg_T(v)-1} }{\displaystyle \sum_{T\in\mathfrak{T}} \prod_{v\in V} w(z_v)^{\deg_T(v)-1}}.
\end{align*}
We notice that
\begin{align}\label{kappa.prime}
\frac{\partial }{\partial z_k } {\kappa}_{i,j,h}(k,z)=\frac{w'(z_k)}{w(z_k)}\big( \widetilde{\kappa}_{i,j,h}(k,z) - {\kappa}_{i,j,h}(k,z) J(k,z) \big).
\end{align}
Using \eqref{kappa1.up} and the fact that  $$\displaystyle|\widetilde{\kappa}_{i,j,h}(k,z)| \le  \frac{2d\sum_{v\in V} w(z_v)}{w(z_i)w(z_j)}\quad \text{and} \quad |J(k,z)|\le d,$$ we obtain \eqref{kappa1.prime}. 
\end{proof}

Recall that $W(t)=\sum_{v\in V}w(L(v,t))$.

\begin{proposition}\label{thm1} Assume that $i$ and $j$ are two neighbour vertices. We have that a.s.
$$\liminf_{t\to\infty} \frac {w(L(i,t))\wedge w(L(j,t))}{W(t)}=0.$$
\end{proposition}
\begin{proof}
Recall from Proposition \ref{decompostion} that
\begin{align}\label{decompZ}
Z_{ij}(t)=M_{ij}(t)+A_{ij}(t) =R_j(t)-R_i(t)+\phi_{ij}(X_t,L(t))-\phi_{ij}(X_0,\ell).
\end{align}
For $\varepsilon\in(0,1]$, set
$$
\Gamma_{ij}^{(\varepsilon)}:=\Big\{w(L(i,t))\wedge w(L(j,t))\ge \varepsilon W(t) \text{ for all sufficiently large }t\Big\}.
$$
Then
$$
\Big\{\liminf_{t\to\infty}\frac{w(L(i,t))\wedge w(L(j,t))}{W(t)}>0 \Big\}=\bigcup_{n=1}^{\infty}\Gamma_{ij}^{(1/n)}.
$$
It therefore suffices to prove that
$\P(\Gamma_{ij}^{(\varepsilon)})=0$ for every $\varepsilon\in(0,1]$.

In virtue of \eqref{phi.def} and \eqref{kappa1.up}, we have
\begin{align}\label{inq.phi1}
|\phi_{ij}(X_t,L(t))|&\le\frac{\sum_{h\in V\setminus\{X_t\}}w(L(h,t))|\kappa_{i,j,h}(X_t,L(t))|}{W(t)^2}\le\frac{2}{w(L(i,t))w(L(j,t))}.
\end{align}
Since $W(t)\to\infty$, on $\Gamma_{ij}^{(\varepsilon)}$ we have $L(i,\infty)=L(j,\infty)=\infty$. Hence, by \eqref{decompZ} and
\eqref{inq.phi1},
$$
\lim_{t\to\infty}Z_{ij}(t)=\varrho_{ij}\quad\text{on }\Gamma_{ij}^{(\varepsilon)},
$$
where
$$
\varrho_{ij}
:=
\int_{\ell_j}^{\infty}\frac{\rmd u}{w(u)}
-
\int_{\ell_i}^{\infty}\frac{\rmd u}{w(u)}
-
\phi_{ij}(X_0,\ell).
$$

Set $\overline Z_{ij}(t):=Z_{ij}(t)-\varrho_{ij}.$ We apply Proposition~\ref{thm:RT_specialized} to $\overline Z_{ij}$, with $\Psi\equiv0$, martingale part $M_{ij}$ and finite variation part $A_{ij}$.
Condition {\rm (i)} therein is therefore automatic. Using \eqref{kappa1.up} and \eqref{kappa1.low}, we have
\begin{align*}
\frac{\1_{\{X_t=i\}}}{w(L(j,t))\big(\sum_{s\in\mathcal N_j}w(L(s,t))\big)^2}\le \Lambda_{ij}(t)\le\frac{4\sum_{h\sim X_t}w(L(h,t))}{w(L(i,t))^2w(L(j,t))^2}.
\end{align*}
Consequently, on $\Gamma_{ij}^{(\varepsilon)}$, for all sufficiently large $t$,
\begin{align}\label{Lambda.thm1}
\frac{\varepsilon^3\1_{\{X_t=i\}}}{w(L(i,t))^3}\le\Lambda_{ij}(t)\le\frac{4\varepsilon^{-4}}{W(t)^3}.
\end{align}
Using the fact that $w$ is non-decreasing and \eqref{lower.regu}, there exists $K_1>0$ such that, for all sufficiently large $t$,
$$
W(t)\ge w\left(\frac{t+|\ell|}{|V|}\right)\ge K_1w(t+|\ell|),
$$
where $|\ell|=\sum_{v\in V}\ell_v$. Set
\begin{align}\label{alpha.thm1}
\alpha_{ij}(t):=\left(\int_{t+|\ell|}^{\infty}\frac{\rmd u}{w(u)^3}\right)^{1/2}.\end{align}
Then $\alpha_{ij}$ is continuous, decreasing and converges to zero. By the upper bound in \eqref{Lambda.thm1}, on
$\Gamma_{ij}^{(\varepsilon)}$, for all sufficiently large $t$,
\begin{align}\label{bracket.thm1.upper}
\langle M_{ij}\rangle_{t,\infty}
&\le 4\varepsilon^{-4}K_1^{-3}\int_{t+|\ell|}^{\infty}\frac{\rmd u}{w(u)^3}
=C_\varepsilon\alpha_{ij}(t)^2,
\end{align}
where $C_\varepsilon:=4\varepsilon^{-4}K_1^{-3}$. On the other hand, since
$L(i,\infty)=\infty$ on $\Gamma_{ij}^{(\varepsilon)}$, it follows from the lower bound in
\eqref{Lambda.thm1} that
\begin{align}\label{bracket.thm1.lower}
\langle M_{ij}\rangle_{t,\infty}
&\ge
\varepsilon^3
\int_t^\infty
\frac{\1_{\{X_u=i\}}}{w(L(i,u))^3}\rmd u
=
\varepsilon^3
\int_{L(i,t)}^\infty\frac{\rmd u}{w(u)^3}
\nonumber\\
&\ge
\varepsilon^3
\int_{t+|\ell|}^\infty\frac{\rmd u}{w(u)^3}
=
\varepsilon^3\alpha_{ij}(t)^2.
\end{align}

Choose $\delta\in(0,1)$ such that $C_\varepsilon\delta\le\frac{\varepsilon^3}{2}$,
and define $\tau(t)>t$ by
$\alpha_{ij}(\tau(t))^2=\delta\alpha_{ij}(t)^2.$
Since $\alpha_{ij}$ is continuous and strictly decreasing to zero, $\tau$ is
well-defined, continuous and non-decreasing. Moreover, we have
$\alpha_{ij}(\tau^n(t))^2=\delta^n\alpha_{ij}(t)^2$,
and hence $\tau^n(t)\to\infty$ as $n\to\infty$. 

From
\eqref{bracket.thm1.upper} and \eqref{bracket.thm1.lower}, on
$\Gamma_{ij}^{(\varepsilon)}$, for all sufficiently large $t$,
\begin{align*}
\langle M_{ij}\rangle_{t,\tau(t)}=
\langle M_{ij}\rangle_{t,\infty}
-
\langle M_{ij}\rangle_{\tau(t),\infty}\ge
\varepsilon^3\alpha_{ij}(t)^2
-
C_\varepsilon\alpha_{ij}(\tau(t))^2\ge
\frac{\varepsilon^3}{2}\alpha_{ij}(t)^2
\end{align*}
and thus
$\alpha_{ij}(t)^2-\alpha_{ij}(\tau(t))^2=
O\big(\langle M_{ij}\rangle_{t,\tau(t)}\big).$
By \eqref{bracket.thm1.upper}, we also have
$\langle M_{ij}\rangle_{t,\infty}= O\big(\alpha_{ij}(t)^2\big)
$ on
$\Gamma_{ij}^{(\varepsilon)}$. Hence conditions {\rm (ii)} and {\rm (iii)} of
Proposition~\ref{thm:RT_specialized} are satisfied.

From Proposition~\ref{mart.M}.{\rm b}, we have
$A_{ij}(t)=\int_0^tF_{ij}(u)\rmd u$, where
$$
F_{ij}(u)
:=
\sum_{h\in V}
\frac{w(L(h,u))}{W(u)^2}
\left(
\frac{2w'(L(X_u,u))}{W(u)}
\kappa_{i,j,h}(X_u,L(u))
-
\left.
\frac{\partial\kappa_{i,j,h}(k,z)}{\partial z_k}
\right|_{k=X_u,z=L(u)}
\right).
$$
Using \eqref{kappa1.up} and \eqref{kappa1.prime}, on
$\Gamma_{ij}^{(\varepsilon)}$, for all sufficiently large $t$,
\begin{align}\label{F.thm1}
\nonumber
|F_{ij}(t)|
&\le
4\varepsilon^{-2}\frac{W'(t)}{W(t)^3}
+
4d\varepsilon^{-2}
\sum_{v\in V}
\frac{\1_{\{X_t=v\}}w'(L(v,t))}
{W(t)^2w(L(v,t))}\\
&\le
4\varepsilon^{-2}\frac{W'(t)}{W(t)^3}
+
4d\varepsilon^{-2}
\sum_{v\in V}
\frac{\1_{\{X_t=v\}}w'(L(v,t))}
{W(t)^{7/4}w(L(v,t))^{5/4}}.
\end{align}
Consequently,
\begin{align}\label{variation.thm1}
\nonumber
\mathbf{V}(A_{ij},(t,\infty))
&\le
\frac{2\varepsilon^{-2}}{W(t)^2}
+
16d\varepsilon^{-2}
\sum_{v\in V}
\int_t^\infty
\frac{1}{W(u)^{7/4}}
\rmd\left(-\frac{1}{w(L(v,u))^{1/4}}\right)\\ 
&\le
\frac{2\varepsilon^{-2}}{W(t)^2}
+
16d\varepsilon^{-2}
\frac{\sum_{v\in V}w(\ell_v)^{-1/4}}{W(t)^{7/4}}\le
\frac{K_2\varepsilon^{-2}}
{w(t+|\ell|)^{7/4}}
\end{align}
for some constant $K_2>0$. Using \eqref{lower.regu} and the fact that $w$ is non-decreasing, note that
\begin{align}\label{ingt.bound}
    \alpha_{ij}(t)w(t+|\ell|)^{7/4}\ge w(t+|\ell|)^{7/4} \Big(\int_{t+|\ell|}^{2(t+|\ell|)} \frac{\rmd u}{w(u)^3}\Big)^{1/2}\ge (t+|\ell|)^{1/2} \frac{w(t+|\ell|)^{7/4}}{w(2(t+|\ell|))^{3/2}}\to \infty.
    \end{align}
It follows that
$
\mathbf{V}(A_{ij},(t,\infty))=o(\alpha_{ij}(t)),
$ on $\Gamma_{ij}^{(\varepsilon)}$,
which verifies condition {\rm (iv)}.

It remains to verify condition {\rm (v)}. At a jumping time $s$, if
$X_{s-}=k$ and $X_s=h$, then, by the proof of Proposition~\ref{mart.M}{\rm a},
$$
|\Delta M_{ij}(s)|
=
\frac{|\kappa_{i,j,h}(k,L(s))|}{W(s)}
\le
\frac{2}{w(L(i,s))w(L(j,s))}.
$$
Since the local times are non-decreasing, for every $t\ge0$ and every stopping time $S\ge t$,
\begin{align*}
\E\big[(\Delta M_{ij}(S))^2\1_{\{S<\infty\}}\mid\mathcal F_t\big]
&\le
\varsigma_t\,\P(S<\infty\mid\mathcal F_t),
\end{align*}
where
$$
\varsigma_t
:=
\frac{4}{w(L(i,t))^2w(L(j,t))^2}.
$$
On $\Gamma_{ij}^{(\varepsilon)}$, for all sufficiently large $t$,
$$
\varsigma_t
\le
\frac{4\varepsilon^{-4}}{W(t)^4}
\le
\frac{K_3\varepsilon^{-4}}{w(t+|\ell|)^4}
$$
for some constant $K_3>0$. Similarly as in \eqref{ingt.bound}, we have that 
$$w(t+|\ell|)^4 \alpha_{ij}(t)^2 \ge  (t+|\ell|)\frac{w(t+|\ell|)^{4}}{w(2(t+|\ell|))^{3}} \to\infty.$$
Hence  $\varsigma_t=o(\alpha_{ij}(t)^2)$ on $\Gamma_{ij}^{(\varepsilon)}$.
Thus condition {\rm (v)} of Proposition~\ref{thm:RT_specialized} is also satisfied.

We have verified all the assumptions of Proposition~\ref{thm:RT_specialized} for $\overline Z_{ij}$ on $\Gamma_{ij}^{(\varepsilon)}$. Since $\overline Z_{ij}(t)\rightarrow0$ on $\Gamma_{ij}^{(\varepsilon)}$,
we conclude that
$
\P(\Gamma_{ij}^{(\varepsilon)})=0.
$
Taking the countable union over $\varepsilon=1/n$, we obtain the result of the proposition.
\end{proof}

Recall that $\mathcal{N}_v$ is the set of all nearest neighbours of vertex $v$, and that $$\mathbb{R}_{*}^{V}=\{(z_v)_{v\in V}\in \R^V:  z_v\ge \ell_* \text{ for each } v\in V\}.$$

\begin{lemma}\label{lem.kappa2}
Let $\gamma$ be a fixed positive constant. Let $i,j,k,h$ be vertices such that $\delta(i,j)=2$. 
Assume that $z=(z_v)_{v\in V}$ is a vector in $\mathbb{R}_{*}^{V}$ such that
$$
\sum_{s\in \mathcal{N}_i\cup\mathcal{N}_j}w(z_s)\le \gamma.
$$
Then there exists a positive constant $K$ depending only on $\ell_*=\min_{v\in V}\ell_v$ and $d=\max_{v\in V}\deg(v)$ such that
\begin{align}
\label{kappa2.up} |\kappa_{i,j,h}(k,z)| &\le K\frac{\sum_{v\in V}w(z_v)}{w(z_i)\wedge w(z_j)},\\
\label{kappa2.low} \sum_{h\in \mathcal{N}_i}|\kappa_{i,j,h}(i,z)| &\ge \gamma^{-1}\frac{\sum_{v\in V}w(z_v)}{w(z_i)},\\
\label{kappa2.prime} \Big|\frac{\partial}{\partial z_k}\kappa_{i,j,h}(k,z)\Big| &\le K \frac{w'(z_k)}{w(z_k)}
\frac{\sum_{v\in V}w(z_v)}{w(z_i)\wedge w(z_j)}.
\end{align}
\end{lemma}

\begin{proof}
We first prove \eqref{kappa2.up}. Recall from \eqref{kappa.decomp} that 
\begin{align*}
    \kappa_{i,j,h}(k,z)
    & = \frac{\displaystyle \sum_{r\in V} w(z_r)\Big[\sum_{F\in\mathfrak{T}_{jh,ik}}-\sum_{F\in\mathfrak{T}_{ih,jk}}\Big]\prod_{v\in V}  w(z_v)^{\deg_F(v)-1} }{\displaystyle \sum_{T\in\mathfrak{T}} \prod_{v\in V} w(z_v)^{\deg_T(v)-1}}.
\end{align*}
Since $\delta(i,j)=2$, there exists a vertex $s\in V$ such that $\{s,i\}\in E$ and $\{s,j\}\in E$.
For this fixed $s$, consider a spanning forest $F=(V, E_F)\in \mathfrak{T}_{jh,ik}\cup\mathfrak{T}_{ih,jk}$.
Since $i$ and $j$ belong to distinct components of $F$, the vertex $s$ belongs to exactly one of these components. If $s$ belongs to the component containing $i$, then $\{s,j\}\notin E_F$ and adding $\{s,j\}$ to $F$ gives a spanning tree $T\in\mathfrak T$. If $s$ belongs to the component containing $j$, then $\{s,i\}\notin E_F$ and adding $\{s,i\}$ gives a spanning tree $T\in\mathfrak T$. Accordingly,
\begin{align*}
\frac{\displaystyle\prod_{v\in V}w(z_v)^{\deg_F(v)-1}}
{\displaystyle\prod_{v\in V}w(z_v)^{\deg_T(v)-1}}
=
\begin{cases}
\displaystyle\frac{1}{w(z_s)w(z_j)},
&\text{if $s$ belongs to the component containing $i$},\\[2mm]
\displaystyle\frac{1}{w(z_s)w(z_i)},
&\text{if $s$ belongs to the component containing $j$}.
\end{cases}
\end{align*}
Using the assumption that $w(z_s)\ge w(\ell_*)$ for all $s\in  V$, we thus have
$$
\prod_{v\in V}w(z_v)^{\deg_F(v)-1} \le \frac{1}{w(\ell_*)\big(w(z_i)\wedge w(z_j)\big)} \prod_{v\in V}w(z_v)^{\deg_T(v)-1}.
$$
For a fixed spanning tree $T\in\mathfrak T$, any forest $F$ mapped to $T$ by the above construction must be obtained from $T$ by deleting either $\{s,i\}$ or $\{s,j\}$. Hence, for each of the two classes $\mathfrak T_{jh,ik}$ and $\mathfrak T_{ih,jk}$, there are at most two forests $F$ which are mapped to the same spanning tree $T$. Hence, there exists a positive constant $K$ depending only on $\ell_*$ and $d$ such that $$
|\kappa_{i,j,h}(k,z)|
\le
K\frac{\sum_{r\in V}w(z_r)}
{w(z_i)\wedge w(z_j)},
$$
which proves \eqref{kappa2.up}.

We next prove \eqref{kappa2.low}. Substituting $k=i$ into \eqref{kappa.decomp}, we obtain, for $h\in\mathcal N_i$,
\begin{align}
\kappa_{i,j,h}(i,z) =
\frac{\displaystyle \sum_{r\in V}w(z_r) \sum_{F\in\mathfrak{T}_{i,jh}} \prod_{v\in V}w(z_v)^{\deg_F(v)-1}}{\displaystyle \sum_{T\in\mathfrak{T}}\prod_{v\in V}w(z_v)^{\deg_T(v)-1}}.
\label{kappa2.decomp}
\end{align}
For each spanning tree $T\in\mathfrak T$, let $s\in\mathcal N_i$ be the neighbour of $i$ in the unique simple path from $i$ to $j$ in $T$. By deleting the edge $\{i,s\}$ and setting $i$ and $j$ respectively as the roots of the two resulting components, we obtain a spanning forest $F\in\mathfrak T_{i,js}.$ Notice that the map $T\mapsto(s,F)$ is injective, and
$$
\prod_{v\in V}w(z_v)^{\deg_T(v)-1} = w(z_i)w(z_s) \prod_{v\in V}w(z_v)^{\deg_F(v)-1}.
$$
Since $s\in\mathcal N_i$ and $\sum_{r\in\mathcal N_i\cup\mathcal N_j}w(z_r)\le\gamma,$ we have $w(z_s)\le\gamma$. Therefore
\begin{align*}
\sum_{T\in\mathfrak T} \prod_{v\in V}w(z_v)^{\deg_T(v)-1}
&\le \gamma w(z_i)\sum_{h\in\mathcal N_i} \sum_{F\in\mathfrak T_{i,jh}} \prod_{v\in V}w(z_v)^{\deg_F(v)-1}.
\end{align*}
Combining this with \eqref{kappa2.decomp}, we obtain
\begin{align*}
\sum_{h\in\mathcal N_i}
|\kappa_{i,j,h}(i,z)|
&=
\frac{\displaystyle
\sum_{r\in V}w(z_r)
\sum_{h\in\mathcal N_i}
\sum_{F\in\mathfrak T_{i,jh}}
\prod_{v\in V}w(z_v)^{\deg_F(v)-1}}
{\displaystyle
\sum_{T\in\mathfrak T}
\prod_{v\in V}w(z_v)^{\deg_T(v)-1}}\ge
\gamma^{-1}
\frac{\sum_{r\in V}w(z_r)}{w(z_i)},
\end{align*}
which proves \eqref{kappa2.low}.  Finally, using \eqref{kappa.prime} and \eqref{kappa2.up}, we obtain \eqref{kappa2.prime}. 
\end{proof}

Define $$V^*:=\{v\in V : L(v,\infty)=\infty\}.$$
\begin{proposition} \label{thm2} Assume that $i$ and $j$ are two vertices at distance $2$. Then
\begin{align}\label{event.dist2}
\P\Big(
\Big\{\liminf_{t\to\infty}\frac{L(i,t)\wedge L(j,t)}{t}>0\Big\}
\cap
\left\{\delta(u,v)\ge 2 \text{ for all distinct }u,v\in V^*\right\}
\Big)=0.
\end{align}
 \end{proposition}
 \begin{proof}
Fix a deterministic subset $S\subseteq V$ containing $i$ and $j$ such that no two distinct vertices of $S$ are neighbours. Fix also $c\in\mathbb{Q}_+$, $m\in\N$ and $t_0\in\N$. Set
\begin{align*}
\Gamma:=\big\{V^*=S,\ L(V\setminus S,\infty)\le m,\, L(i,t)\wedge L(j,t)\ge ct\text{ for all }t\ge t_0\big\}.
\end{align*}
It suffices to prove that $\P(\Gamma)=0$.

By \eqref{lower.regu}, we notice that, on $\Gamma$, for sufficiently large $t$,
\begin{align}\label{W.polynomial}
K_1^{-1} w(t)
\le
w(L(i,t))\wedge w(L(j,t))
\le
W(t)
\le
K_1 w(t)
\end{align}
for some constant $K_1\in (1,\infty)$. 

Recall that
$Z_{ij}(t)=M_{ij}(t)+A_{ij}(t) =R_j(t)-R_i(t)+\phi_{ij}(X_t,L(t))-\phi_{ij}(X_0,\ell).$
On $\Gamma$, all the neighbours of $i$ and $j$ belong to $V\setminus S$, and hence
$$
\sum_{v\in\mathcal{N}_i\cup\mathcal{N}_j}w(L(v,t)) \le 2d\,w(m)
$$
for all $t\ge0$. It follows from Lemma~\ref{lem.kappa2} and \eqref{W.polynomial} that, on $\Gamma$, for all sufficiently large $t$ and all $k,h\in V$,
\begin{align}
\label{kappa.bound.gamma}
|\kappa_{i,j,h}(k,L(t))|&\le K_2,\\
\label{kappa.prime.gamma}
\Big|\frac{\partial}{\partial z_k}\kappa_{i,j,h}(k,z)\Big|_{z=L(t)}\Big| &\le K_2\frac{w'(L(k,t))}{w(L(k,t))},
\end{align}
for some positive constant $K_2$. In virtue of \eqref{phi.def}, \eqref{kappa.bound.gamma} and \eqref{W.polynomial},
\begin{align}\label{phi.dist2.zero}
|\phi_{ij}(X_t,L(t))| \le \frac{K_2}{W(t)^2}\sum_{h\in V}w(L(h,t)) \le \frac{K_3}{w(t)} \rightarrow0.
\end{align}
Consequently, on $\Gamma$,
\begin{align}\label{Z.dist2.limit}
\lim_{t\to\infty}Z_{ij}(t)= \int_{\ell_j}^{\infty}\frac{\rmd u}{w(u)}-\int_{\ell_i}^{\infty}\frac{\rmd u}{w(u)}-\phi_{ij}(X_0,\ell):= \varrho_{ij}.
\end{align}
Set
\begin{align}\label{stopping.dist2}
\zeta:=\inf\Big\{t\ge t_0:\, L(i,t)\wedge L(j,t)<\frac{c}{2}t \text{ or } L(\mathcal{N}_i\cup\mathcal{N}_j,t)\ge m+1 \Big\}.
\end{align}
Then $\zeta$ is a stopping time and $\zeta=\infty$ on $\Gamma$. Define
$$
\overline Z(t):=Z_{ij}(t\wedge\zeta)-\varrho_{ij},\quad \overline M(t):=M_{ij}(t\wedge\zeta),\quad \overline A(t):=A_{ij}(t\wedge\zeta).
$$
Then
$\overline Z(t)-\overline Z(0)=\overline M(t)+\overline A(t).$

We apply Proposition~\ref{thm:RT_specialized} to $\overline Z$, with $\Psi\equiv0$, martingale part $\overline M$ and finite variation part $\overline A$. Condition {\rm (i)} therein is therefore automatic. Set
\begin{align}\label{eta.tau.dist2}
\eta(t):=\left( \int_{t+|\ell|}^{\infty}\frac{\rmd u}{w(u)^2} \right)^{1/2}.
\end{align}
By Assumption~\ref{assump}, we note that $\eta(t)<\infty$ and $\eta(t)\to0$ as $t\to\infty$. Moreover, by \eqref{lower.regu}, there exists a constant $C\in (1,\infty)$ such that, for all sufficiently large $t$,
\begin{align}\label{eta.basic}
C^{-1}\eta(t)^2 \le \int_t^\infty\frac{\rmd u}{w(u)^2} \le C \eta(t)^2, \quad \frac{t}{w(t)^2}\le C \eta(t)^2\quad \text{and} \quad \frac{1}{w(t)}=o(\eta(t)) 
\end{align}
as $t\to\infty$. 

We first estimate the quadratic variation part. Using \eqref{kappa2.low} with $\gamma:=2d\,w(m)$ and the Cauchy--Schwarz inequality, we have that on $\Gamma$, for all sufficiently large $s$,
\begin{align}
\Lambda_{ij}(s) &\ge \frac{\1_{\{X_s=i\}}}{W(s)^2} \sum_{h\in\mathcal{N}_i} w(L(h,s))\kappa_{i,j,h}(i,L(s))^2 \nonumber\\
&\ge \frac{w(\ell_*)\1_{\{X_s=i\}}}{dW(s)^2} \Big( \sum_{h\in\mathcal{N}_i} |\kappa_{i,j,h}(i,L(s))| \Big)^2 \ge K_4 \frac{\1_{\{X_s=i\}}}{w(L(i,s))^2},
\label{Lambda.dist2.lower}
\end{align}
for some constant $K_4>0$. Since $L(i,\infty)=\infty$ on $\Gamma$, we thus have
\begin{align}
\langle M_{ij}\rangle_{t,\infty} &\ge K_4\int_t^\infty \frac{\1_{\{X_s=i\}}}{w(L(i,s))^2}\rmd s = K_4\int_{L(i,t)}^\infty\frac{\rmd u}{w(u)^2} \ge K_4\int_{t+\ell_i}^\infty\frac{\rmd u}{w(u)^2} \ge K_5\eta(t)^2
\label{bracket.lower.integrated}
\end{align}
for all sufficiently large $t$ on $\Gamma$.
By \eqref{kappa.bound.gamma}, for all sufficiently large $s$ on $\Gamma$,
\begin{align}\label{Lambda.dist2.upper}
\Lambda_{ij}(s) \le \frac{K_6}{W(s)^2} \sum_{h\in\mathcal{N}_{X_s}}w(L(h,s)).
\end{align}
If $X_s=v\in S$, then every neighbour of $v$ belongs to $V\setminus S$. Hence, on $\Gamma$, we have
$\sum_{h\in\mathcal{N}_v}w(L(h,s)) \le d\,w(m).$ Using the fact that $W(s)\ge K_1^{-1}w(s)$, we thus obtain that on $\Gamma$,
\begin{align}\label{bracket.S}
\int_t^\infty \1_{\{X_s\in S\}}\Lambda_{ij}(s)\rmd s \le K_7\int_{t+|\ell|}^\infty\frac{\rmd s}{w(s)^2} = K_7 \eta(t)^2.
\end{align}
Fix $v\in V\setminus S$ and set
$$
B_v(t):= \int_0^t \1_{\{X_s=v\}} \sum_{h\in\mathcal{N}_v}w(L(h,s))\rmd s.
$$
We claim that, for all sufficiently large $t$,
\begin{align}\label{Bv.linear}
B_v(t)\le 2 w(m) t \quad\text{a.s. on $\Gamma$.}
\end{align}
Indeed, for $x,y\in V$ with $x\sim y$ and $t\ge0$, let
$$
N(x,y,t):=\sum_{n:\,\tau_n<t}\1_{\{X_{\tau_{n-1}}=x,\,X_{\tau_n}=y\}},
$$
which denotes the number of jumps from $x$ to $y$ strictly before time $t$. Set
$N_v^{\rm in}(t):=\sum_{h\in\mathcal{N}_v}N(h,v,t)$ and $N_v^{\rm out}(t):=\sum_{h\in\mathcal{N}_v}N(v,h,t).$
Note that on $\Gamma$, the compensator of $N_v^{\rm in}$ is
$\int_0^t\1_{\{X_s\in\mathcal{N}_v\}}w(L(v,s))\rmd s \le w(m)t$. By the strong law for counting processes,  we have $\limsup_{t\to\infty} N_v^{\rm in}(t)/t\le w(m)$  a.s. on $\Gamma$.
Since $|N_v^{\rm in}(t)-N_v^{\rm out}(t)|\le 1,$ we also have $\limsup_{t\to\infty} N_v^{\rm out}(t)/t\le w(m)$ on $\Gamma$. Since $B_v$ is the compensator of $N_v^{\rm out}$, applying again the strong law for counting processes, we obtain \eqref{Bv.linear}.

Combining \eqref{Lambda.dist2.upper}, $W(s)\ge K_1^{-1} w(s)$ and
\eqref{Bv.linear}, and using integration by parts, we obtain that a.s. on $\Gamma$,
\begin{align}
\int_t^\infty \1_{\{X_s=v\}}\Lambda_{ij}(s)\rmd s &\le K_8 \int_t^\infty\frac{\rmd B_v(s)}{w(s)^2}\le K_9\left(\frac{t}{w(t)^2}+\int_t^\infty\frac{\rmd s}{w(s)^2}\right)\le K_{10}\eta(t)^2.
\label{bracket.D}
\end{align}
Since $V$ is finite, \eqref{bracket.S} and \eqref{bracket.D} imply
\begin{align}\label{RT.upper.condition}
\langle\overline M\rangle_{\infty} - \langle\overline M\rangle_t \le  K_{11} \eta(t)^2
\quad\text{a.s. on }\Gamma.
\end{align}

Combining \eqref{bracket.lower.integrated} and \eqref{RT.upper.condition}, there exists a constant $K_{12}\in (1,\infty)$ such that, on $\Gamma$, for all sufficiently large $t$,
$$
K_{12}^{-1}\eta(t)^2 \le \langle M_{ij}\rangle_{t,\infty}\le K_{12}\eta(t)^2.
$$
Choose $\delta=\frac{1}{4K_{12}^2}$ and define $\tau(t)>t$ by $\eta(\tau(t))^2=\delta\eta(t)^2.$ Since $\eta$ is continuous and strictly decreasing to zero, $\tau$ is well-defined, continuous and non-decreasing, and $\eta(\tau^n(t))^2=\delta^n\eta(t)^2,$
so that $\tau^n(t)\to\infty$. Moreover, on $\Gamma$, for all sufficiently large $t$,
\begin{align*}
\langle\overline M\rangle_{\tau(t)} - \langle\overline M\rangle_t
&= \langle M_{ij}\rangle_{t,\infty} - \langle M_{ij}\rangle_{\tau(t),\infty}\ge \frac{1}{K_{12}}\eta(t)^2-K_{12}\eta(\tau(t))^2\ge \frac{1}{2K_{12}}\eta(t)^2.
\end{align*}
Hence
\begin{align}\label{RT.lower.condition}
\eta(t)^2-\eta(\tau(t))^2 = O\left( \langle\overline M\rangle_{\tau(t)} - \langle\overline M\rangle_t \right) \quad\text{a.s. on }\Gamma.
\end{align}
Thus conditions {\rm (ii)} and {\rm (iii)} of Proposition~\ref{thm:RT_specialized} are satisfied.

We now estimate the finite variation part. Recall that $A_{ij}(t)=\int_0^tF_{ij}(s)\rmd s$ with 
$$
F_{ij}(u)
:=
\sum_{h\in V}
\frac{w(L(h,u))}{W(u)^2}
\left(
\frac{2w'(L(X_u,u))}{W(u)}
\kappa_{i,j,h}(X_u,L(u))
-
\left.
\frac{\partial\kappa_{i,j,h}(k,z)}{\partial z_k}
\right|_{k=X_u,z=L(u)}
\right).
$$
Using Proposition~\ref{mart.M}.{\rm b}, together with \eqref{kappa.bound.gamma}, \eqref{kappa.prime.gamma} and \eqref{W.polynomial}, we obtain, on $\Gamma$, for all sufficiently large $s$,
\begin{align}
|F_{ij}(s)| &\le K_{13}\left(\frac{w'(L(X_s,s))}{W(s)^2} + \frac{w'(L(X_s,s))} {w(L(X_s,s))W(s)} \right)\le K_{14} \frac{w'(L(X_s,s))}{w(L(X_s,s))w(s)}.
\label{F.dist2.bound}
\end{align}
For $v\in V$, set
$$G_v(t):=\int_{\ell_v}^{L(v,t)}\frac{w'(u)}{w(u)}\rmd u.$$
Since $w$ is non-decreasing,
\begin{align}\label{Gv.log}
0\le G_v(t) \le \log\frac{w(L(v,t))}{w(\ell_v)}.
\end{align}
It follows from \eqref{F.dist2.bound} that
\begin{align}\label{variation.G}
\mathbf{V}(A_{ij},(t,\infty))\le K_{15}\sum_{v\in V}\int_t^\infty\frac{1}{w(s)}\rmd G_v(s).
\end{align}
If $v\in V\setminus S$, then $L(v,\infty)\le m$ on $\Gamma$, and hence
$G_v(\infty)$ is bounded by a deterministic constant depending on $m$. Therefore
\begin{align}\label{G.D}
\int_t^\infty\frac{1}{w(s)}\rmd G_v(s)
=
O\left(\frac{1}{w(t)}\right)
=
o(\eta(t)),
\end{align}
where the last equality follows from \eqref{eta.basic}. If $v\in S$, we first notice that \eqref{lower.regu} implies that there exist $q,C'>0$ such that
$w(t)\le C'(1+t)^q.$
Since $L(v,t)\le\ell_v+t$, it follows from \eqref{Gv.log} that
$G_v(t)\le K_{16}\log(2+t)$
for all sufficiently large $t$. Using integration by parts, we therefore obtain
\begin{align*}
\int_t^\infty\frac{1}{w(s)}\rmd G_v(s) &\le K_{17}\left(\frac{\log(2+t)}{w(t)}+\int_t^\infty\frac{\rmd s}{(1+s)w(s)}\right).
\end{align*}
By \eqref{eta.basic},
we note that $\frac{\log(2+t)}{w(t)} = o(\eta(t)).
$
By Cauchy--Schwarz inequality, we also have
\begin{align*}
\int_t^\infty \frac{\rmd s}{(1+s)w(s)} &\le \left( \int_t^\infty\frac{\rmd s}{w(s)^2} \right)^{1/2} \left(\int_t^\infty\frac{\rmd s}{(1+s)^2}\right)^{1/2}= o(\eta(t)).
\end{align*}
Combining these estimates with \eqref{variation.G}, we obtain
\begin{align}\label{RT.variation.condition}
\mathbf{V}(\overline A,(t,\infty))=o(\eta(t))
\quad\text{a.s. on }\Gamma,
\end{align}
which verifies condition {\rm (iv)}.

It remains to verify condition {\rm (v)}. At a jumping time $s$, if $X_{s-}=k$ and $X_s=h$, then, by the proof of Proposition~\ref{mart.M}{\rm a},
\begin{align}\label{M.jump.kappa}
|\Delta M_{ij}(s)| = \frac{|\kappa_{i,j,h}(k,L(s))|}{W(s)}.
\end{align}
For $t_0\le s\le\zeta$, by the definition of $\zeta$, we have
$$L(i,s)\wedge L(j,s)\ge\frac{c}{2}s \quad\text{and}\quad \sum_{v\in\mathcal N_i\cup\mathcal N_j}w(L(v,s)) \le 2d\,w(m+1).$$
Hence, by Lemma~\ref{lem.kappa2}, \eqref{lower.regu} and \eqref{M.jump.kappa}, we have that 
$|\Delta M_{ij}(s)| \le \frac{K_{18}}{w(s)}$ for all sufficiently large $s\le\zeta$.  Since the jumps before a fixed deterministic time are uniformly bounded and $\overline M$ has no jumps after $\zeta$,  we obtain that for every jump time $s\ge 0$,
\begin{align}\label{M.jump.bound}
|\Delta\overline M(s)| \le \frac{K_{18}}{w(s+\ell_*)}.
\end{align}
Let $T$ be any stopping time such that $T\ge t$. Since $w$ is non-decreasing,
\begin{align}
\E\left[ (\Delta\overline M(T))^2 \1_{\{T<\infty\}}\mid\mathcal F_t\right] &\le \varsigma_t\,\P(T<\infty\mid\mathcal F_t),
\label{RT.jump.condition}
\end{align}
where we set $\varsigma_t:=\frac{K_{18}^2}{w(t+\ell_*)^2}.$
By \eqref{eta.basic}, we have $\varsigma_t=O\left(\frac{\eta(t)^2}{1+t}\right)=o(\eta(t)^2).$
Thus condition {\rm (v)} of Proposition~\ref{thm:RT_specialized} is also satisfied.

Combining \eqref{RT.lower.condition}, \eqref{RT.upper.condition}, \eqref{RT.variation.condition} and \eqref{RT.jump.condition}, all the conditions of Proposition~\ref{thm:RT_specialized} are fulfilled for $\overline Z$ on $\Gamma$. On the other hand, by \eqref{Z.dist2.limit} and the fact that $\zeta=\infty$ on $\Gamma$, we have 
$$
\lim_{t\to\infty}\overline Z(t)=0 \quad\text{a.s. on }\Gamma.
$$
Therefore, by Proposition~\ref{thm:RT_specialized}, we deduce that $\P(\Gamma)=0.$ Taking the countable union over $S$, $c$, $m$ and $t_0$ proves
\eqref{event.dist2}.
\end{proof}

\subsection{Localisation of VRJP on a finite graph} 
Our main aim in this subsection is to prove  the following theorem:

\begin{theorem}\label{localization} Assume that $\mathcal G=(V,E)$ is a finite graph and  $\X$ is a VRJP$(\ell, w)$ defined on $\mathcal G$ such that Assumption \ref{assump} and the condition of Theorem \ref{thm:main}(b) are fulfilled. 
Then, there exists a.s. a unique vertex $j\in V$ such that \begin{equation*} L(j,\infty)=\infty \quad\text{and}\quad L(v,\infty)<\infty\ \ \text{ for all  } v\in V\setminus\{j\}.\end{equation*}
\end{theorem}

We then use the above result to  prove  Theorem \ref{thm:main}(b). The proofs of Theorem \ref{localization} and Theorem \ref{thm:main}(b) are included at the end of this subsection.

We suppose from now on that Assumption \ref{assump} and the condition of Theorem \ref{thm:main}(b) are fulfilled. We emphasize that in the proofs throughout this subsection, we assume there exist constants $\beta>1$, $C_0>0$ and $\lambda_0\in(0,1]$ such that
\begin{align}\label{super.scale}
w(\lambda t)\le C_0\lambda^\beta w(t)
\end{align}
whenever $t$ is sufficiently large, $0<\lambda\le\lambda_0$ and $\lambda t\ge\ell_*$, and
\begin{align}\label{lower.w}
    \liminf_{t\to\infty}\frac{w(ct)}{w(t)}>0\quad \text{for every fixed $c>0$.}
\end{align}
Let $$V^*:=\{i\in V : L(i,\infty)=\infty\}$$ be the random set of vertices with unbounded local times. 
Recall that we denote by $\mathcal{N}_i$ the set of all nearest neighbours of $i$.

\begin{proposition}\label{thm:local} There exists a deterministic positive constant $\eps_0$ depending only on the graph $\mathcal{G}$ and the function $w$ such that 
for each $i\in V^*$, we have that almost surely
$$\liminf_{t\to\infty} \frac{L(i,t)}{L(\mathcal{N}_i,t)}\ge \eps_0.$$
\end{proposition}

\begin{proof}
For each vertex $i\in V$ and $\eps>0$, let $$\mathcal{A}_{i,\eps}:=\Big\{L(i,\infty)=\infty\text{ and }   \liminf_{t\to\infty} \frac{L(i,t)}{L(\mathcal{N}_i,t)}\le \frac{\eps}{2} \Big\}.$$ 
In order to complete the proof, it suffices to show that $\P(\mathcal A_{i,\eps})=0$ for all $i\in V$ and for all sufficiently small $\eps$.

Fix $i\in V$ and  $\eps>0$.  Let $T_0=0$ and for  $n\in \N$,
$$T_{n}:= \inf\{t\ge T_{n-1}: L(i, t)\le \eps L(\mathcal{N}_i, t), L(\mathcal{N}_i, t) \ge n^2\}.$$
Notice that on $\mathcal{A}_{i,\eps}$, both $L(i,t)$ and $L(\mathcal{N}_i,t)$ tend to infinity, and $T_n<\infty$ for all $n\ge0$. Set $\nu_n=(L(\mathcal{N}_i,T_n))^{1/2}$. Note that $\nu_n\ge n$. For each $k\ge 0$ let $\sigma_k^{(n)}:=\inf\{ t\in \R_+ : L(\mathcal{N}_i, t)= (\nu_n+k)^2\}$. Set $$\gamma_{k}^{(n)}:=\frac{L(i,\sigma_k^{(n)})}{L(\mathcal{N}_i,\sigma_k^{(n)})}=\frac{L(i,\sigma_k^{(n)})}{(\nu_n+k)^2} \quad\text{and}\quad \Delta_k^{(n)}:= L(i,\sigma_{k+1}^{(n)})-L(i,\sigma_k^{(n)}).$$
Both $\gamma_{k}^{(n)}$ and $\Delta_{k-1}^{(n)}$ are  $\mathcal{F}_{\sigma_{k}^{(n)}}$-measurable.
Set $$A_{k}^{(n)}:=\{ \gamma_k^{(n)}\le \eps  \}=\big\{ L(i, \sigma_k^{(n)})\le \eps (\nu_n+k)^{2}\big\}.$$
Note that $\mathcal{A}_{i,\eps}\subset A_0^{(n)}$ for each $n\in\N$. Recall that there exist constants $\beta>1$, $C_0>0$ and $\lambda_0\in(0,1]$ such that $w(\lambda t)\le C_0\lambda^\beta w(t)$ whenever $t$ is sufficiently large, $0<\lambda\le\lambda_0$ and $\lambda t\ge\ell_*$. Set $q:=4C_0 (2\deg(i))^{\beta}$ and assume that
$0<\eps\le \min\left\{\frac{\lambda_0}{2\deg(i)},(2q)^{-\frac{1}{\beta-1}}\right\}.$

\begin{claim}\label{claim1}
There exists a positive constant $c_1\in \R$ such that for all sufficiently large $n$ and all $k\ge 0$, 
\begin{align*}
\P\big(\mathcal{A}_{i,\eps}\cap A_k^{(n)},\Delta_k^{(n)}\ge 2q (\nu_n+k)  (\gamma_k^{(n)})^{\beta}   \big)\le  \exp\{- c_1 (n+k)\}.
\end{align*}
\end{claim}
We first prove Claim \ref{claim1}. Set $$\xi_k^{(n)} := \inf\{ t > \sigma_k^{(n)} \colon L(i, t) - L(i, \sigma_k^{(n)}) = 2q  (\nu_n+k) (\gamma_k^{(n)})^{\beta} \}.$$ Denote by $H_k^{(n)}$ the number of jumps from $i$ to $\mathcal{N}_i$ during the time interval $[\sigma_k^{(n)}, \xi_k^{(n)})$. Notice that there exists a neighbour $v\in \mathcal{N}_i$ such that $L(v,\sigma_k^{(n)})\ge L(\mathcal{N}_i,\sigma_k^{(n)})/{\deg(i)}=(\nu_n+k)^2/{\deg(i)}$.
Hence, conditioning on $\mathcal{F}_{\sigma_{k}^{(n)}}$,  $H_k^{(n)}$ stochastically dominates a Poisson random variable with parameter $$a_k^{(n)}=2q(\nu_n+k)(\gamma_k^{(n)})^{\beta}w\big((\nu_n+k)^2/{\deg(i)}\big).$$
In other words, there exist independent exponential random variables $(\chi_j)_{j\in \N}$ with parameter 1, which are independent of $\mathcal{F}_{\sigma_k^{(n)}}$ such that 
$$H_{k}^{(n)}\ge S(a_{k}^{(n)}), \text{ where we define } S(a):=\inf\Big\{ j: \sum_{s=1}^{j}\chi_{s} \ge  a \Big\}-1.$$
On the other hand, conditioning on $\mathcal{F}_{\sigma_{k}^{(n)}}$, on $A_k^{(n)}$, the amount of local time accumulated on $\mathcal N_i$ during each excursion from $i$ before the next return to $i$ in the interval $[\sigma_k^{(n)}, \xi_k^{(n)})$ stochastically dominates an exponential with rate
\begin{align*}
w\big(L(i,\xi_k^{(n)})\big)=w\big(\gamma_k^{(n)} (\nu_n+k)^2 +2q(\nu_n+k)(\gamma_k^{(n)})^{\beta} \big)\le w\left(2\gamma_k^{(n)} (\nu_n+k)^2\right) := b_k^{(n)},
\end{align*}
where the above inequality follows from the fact that  $A_k^{(n)}=\{\gamma_{k}^{(n)}\le \eps\}$ and  $q\eps^{\beta-1}\le 1/2$.
Hence, there exist independent exponential random variables $(\zeta_j)_{j\in \N}$ with parameter 1, which are independent of $\mathcal{F}_{\sigma_k^{(n)}}$ and $(\chi_j)_{j\in \N}$ such that
\begin{align}\label{sto.dom}
L(\mathcal{N}_i,\xi_k^{(n)})-L(\mathcal{N}_i,\sigma_k^{(n)})\ge \frac{1}{b_{k}^{(n)}}\sum_{j=1}^{S(a_{k}^{(n)})} \zeta_{j} \quad\text{on}\quad A_{k}^{(n)}.
\end{align}
For all sufficiently large $n$, in virtue of \eqref{super.scale}, on $A_k^{(n)}$ we have
\begin{align*}
b_k^{(n)}
&\le C_0\big(2\deg(i)\gamma_k^{(n)}\big)^\beta
w\big((\nu_n+k)^2/\deg(i)\big),
\end{align*}
and therefore
\begin{align}\label{ab.ratio}
\frac{a_k^{(n)}}{b_k^{(n)}}
\ge \frac{2q}{C_0(2\deg(i))^\beta}(\nu_n+k)=8(\nu_n+k).
\end{align}
Using the fact that $S(a_k^{(n)})$ is a Poisson random variable with parameter $a_k^{(n)}$, we have, on $A_k^{(n)}$,
\begin{align}\label{chernoff}
\P\Big(\frac{1}{{b}_{k}^{(n)}}\sum_{j=1}^{S({a}_{k}^{(n)})} \zeta_{j}\le 4(\nu_n+k)  \mid \mathcal{F}_{\sigma_{k}^{(n)}}\Big)
&\le \exp\left\{2b_k^{(n)}(\nu_n+k)-\frac{a_k^{(n)}}{3}\right\}\nonumber\\
&\le e^{-c_1(\nu_n+k)}
\le e^{-c_1(n+k)}
\end{align}
for some positive constant $c_1$, where we used \eqref{ab.ratio} and $b_k^{(n)}\ge w(\ell_*)$.
Combining \eqref{sto.dom} and \eqref{chernoff}, we obtain
 \begin{align} \label{prob.B2}
 \P\left(\mathcal{A}_{i,\eps}\cap A_k^{(n)},L(\mathcal{N}_i, \xi_k^{(n)}) \le (\nu_n+k)^2+4(\nu_n+k)\right)   \le \exp\{- c_1 (n+k)\}.
 \end{align}
Using the fact that $x^2+4x> (x+1)^2$ for  all $x\ge 1$, we thus obtain that for all sufficiently large $n$ and for all $k\ge0$,
 $$
 \P\big(\mathcal{A}_{i,\eps}\cap A_k^{(n)},\sigma_{k+1}^{(n)}\ge \xi_k^{(n)}  \big) \le \exp\{- c_1 (n+k)\}.
 $$
On the event $\{\sigma_{k+1}^{(n)}<\xi_k^{(n)}\}$, we have that $$\Delta_k^{(n)}= L(i, \sigma_{k+1}^{(n)})-L(i, \sigma_k^{(n)}) < 2q (\nu_n+k)(\gamma_k^{(n)})^{\beta}.$$
Hence Claim \ref{claim1} is verified.

Set 
$$D_k^{(n)}:=\Big\{\gamma_{k+1}^{(n)}\le \gamma_k^{(n)} \left(1-\frac{2}{1+\nu_n+k} \Big(1- q(\gamma_k^{(n)})^{\beta-1}\Big)+\frac{1}{(1+\nu_n+k)^2}\right)\Big\}.$$
We next prove the following claim:
\begin{claim}\label{claim2}
On $\mathcal{A}_{i,\eps}$, the event $D_{k}^{(n)}$ holds for sufficiently large $n$ and all $k\ge 0$.
\end{claim}

Indeed, we notice that on the event $\{ \Delta_k^{(n)}\le 2q (\nu_n+k)  (\gamma_k^{(n)})^{\beta} \}$, we have
\begin{align}\label{inq.gamma}
\gamma_{k+1}^{(n)}(\nu_n+k+1)^2\le \gamma_{k}^{(n)}(\nu_n+k)^2 + 2q (\nu_n+k)  (\gamma_k^{(n)})^{\beta}.
\end{align}
Dividing both sides of \eqref{inq.gamma} by $(\nu_n+k+1)^2$, this inequality implies that
\begin{align*}
    \gamma_{k+1}^{(n)}\le \gamma_k^{(n)} \left(1-\frac{2}{1+\nu_n+k} \Big(1- q(\gamma_k^{(n)})^{\beta-1}\Big)+\frac{1}{(1+\nu_n+k)^2}\right).
\end{align*}
Hence, it follows from Claim \ref{claim1} that for all sufficiently large $n$ and all $k\ge 0$,
\begin{align}\label{Prob.D}
\P\Big( (D_{k}^{(n)})^c \cap A_k^{(n)}\cap \mathcal{A}_{i,\eps} \Big)\le \exp\{-c_1(n+k)\}. 
\end{align}
Recall that $A_{k}^{(n)}=\Big\{ \gamma_{k}^{(n)}\le \eps\}$. On $A_{k}^{(n)}\cap D_{k}^{(n)}$, we thus notice that 
$$
\begin{aligned}
\gamma_{k+1}^{(n)}&\le  \gamma_k^{(n)} \left(1-\frac{2}{1+\nu_n+k} \Big(1- q \eps^{\beta-1}\Big)+\frac{1}{(1+\nu_n+k)^2}\right) \le \gamma_{k}^{(n)}\le \eps.
\end{aligned}
$$ 
Hence, for all sufficiently large $n$ and all $k\ge 0$, 
\begin{align}\label{D.incl}
    D_{k}^{(n)}\cap A_{k}^{(n)}\subset A_{k+1}^{(n)}.
\end{align}
Combining the fact that $\mathcal{A}_{i,\eps}\subset A_0^{(n)}$ with \eqref{D.incl} and \eqref{Prob.D}, we have, for all sufficiently large $n$,
\begin{align*}
\P\Big(\mathcal{A}_{i,\eps} \cap \bigcup_{k=0}^{\infty} (D_k^{(n)})^c  \Big)
&\le  \sum_{k=0}^{\infty}\P\Big(  (D_{k}^{(n)})^c \cap A_k^{(n)}\cap \mathcal{A}_{i,\eps} \Big)\\
&\le \sum_{k=0}^{\infty} \exp\{- c_1 (n+k)\}\le e^{-c_2 n},
\end{align*}
where $c_2$ is some positive constant. Using Borel-Cantelli lemma, we have that on $\mathcal{A}_{i,\eps}$, the event $D_{k}^{(n)}$ holds for sufficiently large $n$ and all $k\ge 0$.
Hence, Claim \ref{claim2} is verified.
 
Using Claim \ref{claim2} and applying Lemma \ref{gamma.seq} to the sequence $(\gamma_{k}^{(n)})_{k\ge0}$ with $\alpha=\beta$, $p=2$, $\nu=\nu_n+1$ and $r_k=(1+\nu_n+k)^{-2}$, we infer that, on the event $\mathcal{A}_{i,\eps}$, for sufficiently large $n$, 
 $$L(i,\infty)=\limsup_{k\to\infty} (\nu_n+k)^2\gamma_{k}^{(n)} <\infty\quad\text{almost surely.}$$ 
 This however contradicts $\mathcal{A}_{i,\eps}\subset \{L(i,\infty)=\infty\}$.
It follows that $\P(\mathcal{A}_{i,\eps})=0$. This completes the proof.
\end{proof}

The set $V^*$ is composed of connected components, which we refer to as \textbf{clusters}. By applying Proposition \ref{thm:local}, we obtain the following result that allows us to compare the local time of a vertex with the local time of its corresponding  cluster. 

Recall that, for $x,y\in V$ with $x\sim y$ and $t\ge0$, we denote by $N(x,y,t)$ the number of jumps from $x$ to $y$ strictly before time $t$.

\begin{lemma}\label{lem.jump}
\begin{enumerate} 
\item[a.] Let $x,y\in V$ with $x\sim y$. On the event $\{L(x,\infty)<\infty,\ L(y,\infty)<\infty\}$, we have almost surely $N(x,y,\infty)<\infty.$
\item[b.]  If $V^*$ is composed of more than one cluster, then for any cluster $U$, we have almost surely $\delta(U, V^*\setminus U)=2$. Furthermore, there exists a positive deterministic constant $\gamma_0$, depending only on the graph $\mathcal{G}$ and the function $w$, such that for each cluster $U$ and for any vertex $i \in U$, we have, a.s.
$$\liminf_{t\to\infty}\frac{L(i,t)}{L(U,t)}\ge \gamma_0.$$ 
\item[c.] Let $U$ be a cluster of $V^*$ and let $v\in V\setminus V^*$ be such that $\delta(v,U)=1$. For each $m>0$, there exist positive deterministic constants $c_1$ and $c_2$, depending only on $\mathcal{G}$, $w$ and $m$, such that, on the event $\{L(v,\infty)\le m\},$ almost surely, for all sufficiently large $t$, 
\begin{align*}
c_1L(U,t) \le \sum_{u\in U\cap\mathcal{N}_v}N(u,v,t) \le c_2L(U,t).
\end{align*}
\item[d.] Let $U$ be a cluster of $V^*$ and let $v\in V\setminus V^*$ be such that $\delta(v,U)=1$. There exist positive deterministic constants $c_3$ and $c_4$, depending only on $\mathcal{G}$ and $w$, such that almost surely, for all sufficiently large $t$,
\begin{align*}
c_3 w(L(U,t)) \le \sum_{u\in U\cap\mathcal{N}_v}w(L(u,t)) \le c_4 w(L(U,t)).
\end{align*}
\end{enumerate}
\end{lemma}

\begin{proof}
a. For $x\sim y$, the process
$
N(x,y,t)-\int_0^t\1_{\{X_s=x\}}w(L(y,s))\rmd s
$
is a local martingale with angle bracket
$\int_0^t\1_{\{X_s=x\}}w(L(y,s))\rmd s.$
On the event $\{L(x,\infty)<\infty,\ L(y,\infty)<\infty\}$, using the fact that $w$ is non-decreasing, we have
\begin{align*}
\int_0^\infty\1_{\{X_s=x\}}w(L(y,s))\rmd s &\le w(L(y,\infty))\int_0^\infty\1_{\{X_s=x\}}\rmd s \\
&= w(L(y,\infty))\big(L(x,\infty)-\ell_x\big)<\infty.
\end{align*}
Hence the above local martingale converges almost surely to a finite random variable. It follows that $N(x,y,t)$ converges almost surely to a finite random variable as $t\to\infty$, which proves Part a.

b. 
By Part a. and the fact that $V$ is finite, there exists almost surely a finite random time after which there is no jump between two vertices of $V\setminus V^*$. Hence, for each cluster $U$, we must have $\delta(U, V^* \setminus U) = 2$ if $V^* \setminus U \neq \emptyset$. Recall that $\eps_0$ is the deterministic constant defined in Proposition \ref{thm:local}. Let $i$ be a vertex in cluster $U$. By Proposition \ref{thm:local}, 
  $$L(i,t)\ge \frac{\eps_0}{2}L(\mathcal{N}_i,t)\ge \frac{\eps_0}{2}L(j,t)$$
 for any $j \in U \cap \mathcal{N}_i$ and for sufficiently large $t$. By reiterating, we have $L(i,t) \ge \left(\frac{\eps_0}{2}\right)^{|V|} L(j,t)$ for any $j \in U$ and sufficiently large $t$. This fact completes the result of Part b.

c. Fix a deterministic set $U\subseteq V$ and a deterministic vertex $v\in V\setminus U$ such that $U\cap\mathcal{N}_v\neq\emptyset$. Set
$$
A(t):=\sum_{u\in U\cap\mathcal{N}_v}\int_0^t\1_{\{X_s=u\}}w(L(v,s))\rmd s.
$$
Then
$\sum_{u\in U\cap\mathcal{N}_v}N(u,v,t)-A(t)$
is a local martingale with angle bracket $A(t)$. We now work on the event that $U$ is
a cluster of $V^*$ and $L(v,\infty)\le m$. On this event, we have
\begin{align}\label{cluster.A.upper}
A(t) &\le w(m)\sum_{u\in U\cap\mathcal{N}_v}\big(L(u,t)-\ell_u\big)\le w(m)L(U,t).
\end{align}
Fix $u\in U\cap\mathcal{N}_v$. By Part b., almost surely, $L(u,t)-\ell_u\ge \frac{\gamma_0}{2}L(U,t)$ for all sufficiently large $t$. Therefore
\begin{align}\label{cluster.A.lower}
A(t) &\ge w(\ell_*)\big(L(u,t)-\ell_u\big) \ge \frac{\gamma_0w(\ell_*)}{2}L(U,t).
\end{align}
In particular, $A(t)\to\infty$. By the strong law for local martingales,
$$
\frac{1}{A(t)}\Big(\sum_{u\in U\cap\mathcal{N}_v}N(u,v,t)-A(t)\Big)\rightarrow0
$$
as $t\to\infty$, almost surely. Combining this with \eqref{cluster.A.upper} and \eqref{cluster.A.lower} proves Part c.

d. Fix $u\in U\cap\mathcal{N}_v$. By Part b., almost surely,
$
L(u,t)\ge \frac{\gamma_0}{2}L(U,t)
$
for all sufficiently large $t$. Using \eqref{lower.w}, we thus have
\begin{align*}
\sum_{z\in U\cap\mathcal{N}_v}w(L(z,t)) &\ge w(L(u,t)) \ge c_3w(L(U,t)),\\
\sum_{z\in U\cap\mathcal{N}_v}w(L(z,t))&\le d w(L(U,t)),
\end{align*}
for all sufficiently large $t$.
This proves Part d.
\end{proof}

We next compare the local times of the clusters of $V^*$. For each time $t$, we order the clusters of $V^*$ according to their local times at time $t$, with $\Ccal_1(t)$ being the cluster with the largest local time and $\Ccal_2(t)$ the second largest, and so on. Let $J$ be the number of clusters of $V^*$. Note that $J$ is random but does not depend on $t$.
 
\begin{proposition}\label{dist2.lim} We have a.s.
$$\liminf_{t\to\infty}\frac{L(\mathcal{C}_J(t),t)}{L(\Ccal_{1}(t),t)}>0.$$
\end{proposition}

\begin{proof}
The case $J=1$ is trivial. We work on the event $\{J=j\}$ with $j\ge 2$. By Lemma \ref{lem.jump}(a) and the fact that $V$ is finite, there exists almost surely a finite random time after which there is no jump between two vertices of $V\setminus V^*$. It follows that, if $U_1,\ldots,U_j$ are the clusters of $V^*$ and $I$ is a nonempty proper subset of $\{1,\ldots,j\}$, then
\begin{align}\label{cluster.distance}
\delta\Big(\bigcup_{i\in I}U_i,\bigcup_{i\notin I}U_i\Big)=2.
\end{align}

Fix positive integers $j$ and $k$ such that $1\le k\le j-1$, and fix a collection of pairwise disjoint connected subsets $\mathcal{U}=(U_i)_{1\le i\le j}$ of $V$ such that $\delta\left(U_i,\bigcup_{h\ne i}U_h\right)=2$ for each $1\le i\le j$ and $\delta\left(\bigcup_{h=1}^kU_h,\bigcup_{i=k+1}^jU_i\right)=2.$
Set
$D:=V\setminus\bigcup_{i=1}^jU_i.$
For $t\in\R_+$, set
\begin{align*}
a(t) &:= \max\Big\{ L(U_i,t):  k+1\le i\le j,\,  \delta\Big(U_i,\bigcup_{h=1}^kU_h\Big)=2 \Big\},\quad b(t) := \min_{1\le h\le k} L(U_h,t).
\end{align*}
Fix $m,n\in\N$ with $m\ge\ell^*$. We use $c_i$, $i\in\N$, to denote positive constants depending only on the graph $\mathcal{G}$ and the function $w$. Constants are also allowed to depend on $m$ when this is explicitly indicated. By Lemma \ref{lem.jump}(c)-(d), there exist positive constants $c_1=c_1(m)$, $c_2=c_2(m)$, $c_3$ and $c_4$ such that, on
$\{L(D,\infty)\le m\}\cap\{U_1,\ldots,U_j\text{ are the clusters of }V^*\},$
almost surely, for all sufficiently large $t$,
\begin{align}\label{cluster.jump.bound}
& c_1L(U_i,t) \le \sum_{u\in U_i\cap\mathcal{N}_v}N(u,v,t) \le c_2L(U_i,t),\\
\label{cluster.weight.bound}
&c_3w(L(U_i,t))
\le \sum_{u\in U_i\cap\mathcal{N}_v}w(L(u,t)) \le c_4w(L(U_i,t)),
\end{align}
for every $1\le i\le j$ and every $v\in D$ such that $\delta(v,U_i)=1$.

Let $T$ be the first time $t\ge n$ at which at least one of the following occurs:
\begin{itemize}
    \item $L(D,t)>m$;
    \item the process makes a jump from a vertex in $D$ to another vertex in $D$;
    \item there exist $1\le i\le j$ and $v\in D$ with $\delta(v,U_i)=1$ such that \eqref{cluster.jump.bound} or \eqref{cluster.weight.bound} does not hold at time $t$.
\end{itemize}
We prove that 
\begin{align}\label{cluster.claim}
\liminf_{t\to\infty}\frac{a(t)}{b(t)}>0\quad \text{a.s. on $\{T=\infty\}\cap\{U_1,\ldots,U_j\text{ are the clusters of }V^*\}$.}
\end{align}

Fix $K:=4c_2/c_1.$ Let $\eps\in(0,K^{-1})$, whose value will be chosen sufficiently small below. Let $\mathfrak{t}\ge n$ be a stopping time such that $a(\mathfrak{t})\le\eps b(\mathfrak{t}),$ and set
$$\mathfrak{s}:=\inf\left\{t\ge\mathfrak{t}:a(t)\ge K\eps b(t)\right\}.$$
Recall that $(\tau_r)_{r\ge0}$ are the jumping times of the process. For $n\le s<t<T$, let
\begin{align*}
N_{2,1}(s,t) &:=\sum_{r:\,s<\tau_r<\tau_{r+1}\le t} \1_{\left\{X_{\tau_{r-1}}\in\bigcup_{i=k+1}^jU_i,\,X_{\tau_r}\in D,\,X_{\tau_{r+1}}\in\bigcup_{h=1}^kU_h\right\}},\\
N_{1,2}(s,t) &:=\sum_{r:\,s<\tau_r<\tau_{r+1}\le t}\1_{\left\{X_{\tau_{r-1}}\in\bigcup_{h=1}^kU_h,\, X_{\tau_r}\in D,\, X_{\tau_{r+1}}\in\bigcup_{i=k+1}^jU_i\right\}}
\end{align*}
be the numbers of passages from $\bigcup_{i=k+1}^jU_i$ to $\bigcup_{h=1}^kU_h$ and in the reverse direction, respectively, through a vertex of $D$, where each passage consists of two consecutive jumps with both jump times in $(s,t]$.

We prove \eqref{cluster.claim} and complete the proof of the proposition in six steps.

\textbf{Step 1.} We prove in this step that, for $\eps>0$ sufficiently small and for $q_0=q_0(m)$ sufficiently large, there exist positive constants $c_5,c_6$ and $c_7$ such that
\begin{align}\label{cluster.out}
\P\left( N_{2,1}(\mathfrak{t},\mathfrak{s})\le c_5a(\mathfrak{s}),\ \mathfrak{s}<T \mid\mathcal{F}_\mathfrak{t} \right) \le c_6\exp\{-c_7\eps b(\mathfrak{t})\}\quad\text{on $\{\mathfrak{t}<T,\ b(\mathfrak{t})\ge q_0\}$.}
\end{align}

On $\{\mathfrak{s}<T\}$, by continuity, we have $a(\mathfrak{s})=K\eps b(\mathfrak{s}).$ For each $i$, with $k+1\le i\le j$ and $\delta\Big(U_i,\bigcup_{h=1}^kU_h\Big)=2$, fix vertices $v_i\in D$ such that $\delta(v_i,U_i)=\delta\left(v_i,\bigcup_{h=1}^kU_h\right)=1$. For each $i$ fixed above, enumerate the successive jumping times $(\rho_l^{(i)})_{l\ge1}$ after $\mathfrak{t}$ at which the process jumps from a vertex of $U_i\cap\mathcal{N}_{v_i}$ to $v_i$, and let $\xi_l^{(i)}$ be the indicator that the succeeding jump from $v_i$ is to $\bigcup_{h=1}^kU_h$.
On $\{\mathfrak{s}<T\}\cap \{L(U_i,\mathfrak{s})=a(\mathfrak{s})\}$, we have
$L(U_i,\mathfrak{t}) \le a(\mathfrak{t}) \le\eps b(\mathfrak{t}) \le\eps b(\mathfrak{s}) =K^{-1} L(U_i,\mathfrak{s}),$
and hence, by \eqref{cluster.jump.bound} and the choice of $K$,
\begin{align}\label{cluster.out.count}
\sum_{u\in U_i\cap\mathcal{N}_{v_i}} \big(N(u,v_i,\mathfrak{s})-N(u,v_i,\mathfrak{t})\big) &\ge c_1L(U_i,\mathfrak{s})-c_2L(U_i,\mathfrak{t})\nonumber\\
&\ge \left(c_1-\frac{c_2}{K}\right)L(U_i,\mathfrak{s}) \ge \frac{3c_1}{4}a(\mathfrak{s}).
\end{align}
For $\mathfrak{t}\le t<\mathfrak{s}\wedge T$, we have $a(t)<K\eps b(t)$. By \eqref{cluster.weight.bound}, for each $i$ fixed above, we have
\begin{align*}
\sum_{z\in\mathcal{N}_{v_i}\cap\bigcup_{h=1}^kU_h}w(L(z,t))
&\ge c_3w(b(t)),\\
\sum_{z\in\mathcal{N}_{v_i}\cap\bigcup_{r=k+1}^jU_r}w(L(z,t))
&\le c_4d\,w(a(t))
\le c_8(K\eps)^\beta w(b(t)),
\end{align*}
where the last inequality follows from \eqref{super.scale} and we choose $\eps$ sufficiently small such that $K\eps\le\lambda_0$. Moreover, for $t<T$, we have
$
\sum_{z\in\mathcal{N}_{v_i}\cap D}w(L(z,t))\le dw(m).
$
Consequently, for $\rho_l^{(i)}<\mathfrak{s}\wedge T$,
\begin{align}\label{return.prob}
\E\left(\xi_l^{(i)}\mid\mathcal F_{\rho_l^{(i)}}\right)
&=\frac{\displaystyle\sum_{z\in\mathcal N_{v_i}\cap\bigcup_{h=1}^kU_h}w(L(z,\rho_l^{(i)}))}{\displaystyle\sum_{z\in\mathcal N_{v_i}}w(L(z,\rho_l^{(i)}))}\ge 1-c_8(K\eps)^\beta-\frac{c_9w(m)}{w(b(\rho_l^{(i)}))}.
\end{align}
Choose $\eps>0$ sufficiently small such that $K\eps\le\lambda_0$ and $c_8(K\eps)^\beta\le\frac18$. Then choose $q_0=q_0(m)$ sufficiently large such that $c_9w(m)/w(q_0)\le\frac18.$ On $\{b(\mathfrak{t})\ge q_0\}$, for every $l$ such that $\rho_l^{(i)}<\mathfrak{s}\wedge T$, we thus have
\begin{align}\label{out.success}
\E\left(\xi_l^{(i)}\mid\mathcal F_{\rho_l^{(i)}}\right)\ge\frac34.
\end{align}
We now apply Lemma \ref{lem:Chernoff} (see Appendix). Fix $r\in\N$ and $i$ as above, and consider the event
$$
\{r\le a(\mathfrak{s})<r+1,\, L(U_i,\mathfrak{s})=a(\mathfrak{s}),\, \mathfrak{s}<T\}.
$$
By \eqref{cluster.out.count}, on this event there are at least $\frac{3c_1}{4}r-1$ jumps from $U_i\cap\mathcal N_{v_i}$ to $v_i$ during $(\mathfrak{t},\mathfrak{s})$.
At most one of these jumps has its succeeding jump after $\mathfrak{s}$. Hence, choosing $c_5>0$ sufficiently small and $q_0$ sufficiently large, by Lemma \ref{lem:Chernoff} and \eqref{out.success}, there exist positive constants $c_{10},c_{11}$ such that
\begin{align}\label{N21.inq}
&\P\left(N_{2,1}(\mathfrak{t},\mathfrak{s})\le c_5a(\mathfrak{s}),\ r\le a(\mathfrak{s})<r+1,\, L(U_i,\mathfrak{s})=a(\mathfrak{s}),\, \mathfrak{s}<T\mid\mathcal F_\mathfrak{t}\right)\le c_{10}e^{-c_{11}r}
\end{align}
on $\{\mathfrak{t}<T,\ b(\mathfrak{t})\ge q_0\}$. Since, on $\{\mathfrak{s}<T\}$,
$a(\mathfrak{s})=K\eps b(\mathfrak{s})\ge K\eps b(\mathfrak{t}),$
we have $r\ge K\eps b(\mathfrak{t})-1$ whenever $r\le a(\mathfrak{s})<r+1$. Summing \eqref{N21.inq} over $r\ge K\eps b(\mathfrak{t})-1$ and over the finitely many possible choices of $i$, we obtain \eqref{cluster.out}.

\textbf{Step 2.}  We prove in this step that, for $\eps>0$ sufficiently small and for $q_0=q_0(m)$ sufficiently large, there exist positive constants $c_{12},c_{13},c_{14}$ such that
\begin{align}\label{cluster.in}
\P\left(N_{1,2}(\mathfrak{t},\mathfrak{s})\ge c_{12}\eps^{\beta-1}a(\mathfrak{s}),\ \mathfrak{s}<T\mid\mathcal{F}_\mathfrak{t}\right)
\le c_{13}\exp\{-c_{14}\eps^\beta b(\mathfrak{t})\}\quad\text{on $\{\mathfrak{t}<T,\ b(\mathfrak{t})\ge q_0\}$.}
\end{align}

Fix $1\le h\le k$, $k+1\le i\le j$ and $v\in D$ such that $\delta(v,U_h)=\delta(v,U_i)=1.$ Let $(\rho_l)_{l\ge1}$ be the successive jumping times after $\mathfrak{t}$ at which the process jumps from a vertex of $U_h\cap\mathcal{N}_v$ to $v$. Let $\xi_l^{(h,i,v)}$ be the indicator that the succeeding jump from $v$ is to $U_i$.
On $\{\mathfrak{s}<T\}$, for every $\rho_l<\mathfrak{s}$, by \eqref{cluster.weight.bound}, the total rate from $v$ towards $U_i$ is bounded above by $c_4w(L(U_i,\rho_l))$, while the rate from $v$ towards $U_h$ is bounded below by $c_3w(L(U_h,\rho_l))$. Moreover, $L(U_i,\rho_l)\le a(\rho_l)<K\eps b(\rho_l)\le K\eps L(U_h,\rho_l)$. Hence, since $K\eps\le\lambda_0$, by \eqref{super.scale}, we have
\begin{align}\label{cluster.forward.prob}
\E(\xi_l^{(h,i,v)}\mid\mathcal{F}_{\rho_l})
\le c_{15}\left(\frac{L(U_i,\rho_l)}{L(U_h,\rho_l)}\right)^\beta.
\end{align}
Choosing $q_0$ sufficiently large, we may assume that \eqref{super.scale} applies whenever $L(U_h,t)\ge b(t)\ge q_0$.

For $\mathfrak{t}\le t\le\mathfrak{s}$, we have $L(U_i,t)\le a(t)\le K\eps b(t)\le K\eps L(U_h,t)$. Moreover, on $\{\mathfrak{s}<T\}$, by \eqref{cluster.jump.bound}, we have
$$
\sum_{u\in U_h\cap\mathcal{N}_v} \big(N(u,v,t)-N(u,v,\mathfrak{t})\big) \le c_2L(U_h,t).
$$
Applying Lemma \ref{lem.integral} (see Appendix) with $\alpha:=\beta$, $f(t):=L(U_i,t)$, $g(t):=L(U_h,t)$ and
$N(t):=\sum_{u\in U_h\cap\mathcal{N}_v}\big(N(u,v,t)-N(u,v,\mathfrak{t})\big)$, we obtain
\begin{align}\label{cluster.in.comp}
\sum_{l:\,\rho_l<\mathfrak{s}}
\left(\frac{L(U_i,\rho_l)}{L(U_h,\rho_l)}\right)^\beta
&\le \int_{(\mathfrak{t},\mathfrak{s}]}
\left(\frac{L(U_i,t)}{L(U_h,t)}\right)^\beta\rmd N(t)\nonumber\\
&\le c_{16}\eps^{\beta-1}L(U_i,\mathfrak{s})
\le c_{16}\eps^{\beta-1}a(\mathfrak{s}).
\end{align}
Combining \eqref{cluster.forward.prob} and \eqref{cluster.in.comp}, we get
\begin{align}\label{cluster.forward.sum}
\sum_{l:\,\rho_l<\mathfrak{s}}\E(\xi_l^{(h,i,v)}\mid\mathcal{F}_{\rho_l})
\le c_{17}\eps^{\beta-1}a(\mathfrak{s}).
\end{align}
Fix $r\in\N$ and consider the event $r\le\eps^{\beta-1}a(\mathfrak{s})<r+1$. 
On this event, \eqref{cluster.forward.sum} is bounded above by $2c_{17}r$. Since every passage counted by $N_{1,2}(\mathfrak{t},\mathfrak{s})$ corresponds to one of the finitely many possible triples $(h,i,v)$, choosing $c_{12}>0$ sufficiently large and applying Lemma \ref{lem:Chernoff} to each such triple, we obtain positive constants $c_{18},c_{19}$ such that
$$
\P\left(N_{1,2}(\mathfrak{t},\mathfrak{s})\ge c_{12}\eps^{\beta-1}a(\mathfrak{s}),\
r\le\eps^{\beta-1}a(\mathfrak{s})<r+1,\ \mathfrak{s}<T\mid\mathcal F_\mathfrak{t}\right)
\le c_{18}e^{-c_{19}r}.
$$
On $\{\mathfrak{s}<T\}$, we have $\eps^{\beta-1}a(\mathfrak{s})=K\eps^\beta b(\mathfrak{s})\ge K\eps^\beta b(\mathfrak{t})$, and hence $r\ge K\eps^\beta b(\mathfrak{t})-1$. Summing over $r$, we obtain \eqref{cluster.in}.

\textbf{Step 3.} We prove in this step that there exist positive constants $c_{20},c_{21}$ such that
\begin{align}\label{cluster.upcross}
\P\left(\mathfrak{s}<T\mid\mathcal{F}_\mathfrak{t}\right) \le c_{20}\exp\{-c_{21}\eps^\beta b(\mathfrak{t})\}\quad\text{on $\{\mathfrak{t}<T,\ b(\mathfrak{t})\ge q_0\}$.}
\end{align}
Choose $\eps\in (0,1)$ sufficiently small, in addition to the conditions imposed in Steps 1 and 2, such that 
$c_{12}\eps^{\beta-1}<\frac{c_5}{2}$. Before time $T$, there is no jump between two vertices of $D$. Hence passages between the two collections of clusters alternate, apart from at most one unfinished passage at each endpoint of the time interval. Therefore
\begin{align}\label{cluster.balance}
\left|N_{2,1}(\mathfrak{t},\mathfrak{s})-N_{1,2}(\mathfrak{t},\mathfrak{s})\right|\le2.
\end{align}
On $\{\mathfrak{s}<T\}$, if
$N_{2,1}(\mathfrak{t},\mathfrak{s})>c_5a(\mathfrak{s})$ and $N_{1,2}(\mathfrak{t},\mathfrak{s})<c_{12}\eps^{\beta-1}a(\mathfrak{s})$, then by the choice of $\eps$, we have
$$
N_{2,1}(\mathfrak{t},\mathfrak{s})-N_{1,2}(\mathfrak{t},\mathfrak{s})
>\frac{c_5}{2}a(\mathfrak{s}).
$$
Since $a(\mathfrak{s})=K\eps b(\mathfrak{s})\ge K\eps b(\mathfrak{t})$,  choosing $q_0$ sufficiently large such that $\frac{c_5}{2}K\eps q_0>2$, this contradicts \eqref{cluster.balance}. Hence, by \eqref{cluster.out} and \eqref{cluster.in}, we obtain
\begin{align*}
\P\left(\mathfrak{s}<T\mid\mathcal F_\mathfrak{t}\right)
&\le
\P\left(N_{2,1}(\mathfrak{t},\mathfrak{s})\le c_5a(\mathfrak{s}),\ \mathfrak{s}<T\mid\mathcal F_\mathfrak{t}\right)+
\P\left(N_{1,2}(\mathfrak{t},\mathfrak{s})\ge c_{12}\eps^{\beta-1}a(\mathfrak{s}),\ \mathfrak{s}<T\mid\mathcal F_\mathfrak{t}\right)\\
&\le
c_6\exp\{-c_7\eps b(\mathfrak{t})\}
+c_{13}\exp\{-c_{14}\eps^\beta b(\mathfrak{t})\}.
\end{align*}
This implies \eqref{cluster.upcross}.

\textbf{Step 4.} We show in this step that, almost surely, on $\{T=\infty\}\cap\{U_1,\ldots,U_j\text{ are the clusters of }V^*\}$, it is impossible that
\begin{align}\label{cluster.stay} a(t)<K\eps b(t)\quad\text{for all sufficiently large }t.
\end{align}

Fix a deterministic integer $s\ge n$ and consider the event
$$\{T=\infty\}
\cap\{U_1,\ldots,U_j\text{ are the clusters of }V^*\} \cap\{b(s)\ge q_0,\ a(t)<K\eps b(t)\text{ for all }t\ge s\}.
$$ For every $i$, with $k+1\le i\le j$ and
$\delta\left(U_i,\bigcup_{h=1}^kU_h\right)=2,$
choose $v_i\in D$ such that $\delta(v_i,U_i)=1$ and $v_i$ has a neighbour in $\bigcup_{h=1}^kU_h$. By \eqref{cluster.jump.bound},
$$
\sum_{u\in U_i\cap\mathcal{N}_{v_i}} \big(N(u,v_i,t)-N(u,v_i,s)\big) \ge c_1L(U_i,t)-c_2L(U_i,s).
$$
Since $L(U_i,t)\to\infty$, for all sufficiently large $t$,
$$
\sum_{u\in U_i\cap\mathcal{N}_{v_i}} \big(N(u,v_i,t)-N(u,v_i,s)\big) \ge \frac{c_1}{2}L(U_i,t).
$$
Let $(\rho_r^{(i)})_{r\ge1}$ be the successive jumping times after $s$ at which the process jumps from a vertex of $U_i\cap\mathcal{N}_{v_i}$ to $v_i$, and let $\xi_r^{(i)}$ be the indicator that the succeeding jump from $v_i$ is to $\bigcup_{h=1}^kU_h$. By the same argument leading to \eqref{return.prob}, we have
$$
\E(\xi_r^{(i)}\mid\mathcal{F}_{\rho_r^{(i)}})\ge\frac34
$$
for every $r\ge1$.
Let $M_i(s,t)$ denote the number of these arrivals before time $t$ and let $Q_i(s,t)$ be the sum of the corresponding indicators
$\xi_r^{(i)}$. Then the partial sums
$$
\sum_{r=1}^n\left(\xi_r^{(i)}-\E(\xi_r^{(i)}\mid\mathcal{F}_{\rho_r^{(i)}})\right) \quad\text{for } n\ge1,
$$
form a martingale with predictable quadratic variation bounded by $n$. Hence, by the martingale strong law, we have $Q_i(s,t)\ge\frac34M_i(s,t)+o(M_i(s,t)).$
At most one of the arrivals counted by $M_i(s,t)$ has its succeeding jump after $t$. Consequently, for every such $i$, we have
$N_{2,1}(s,t)\ge Q_i(s,t)-1\ge c_{22}L(U_i,t)$
for all sufficiently large $t$. Hence, for all sufficiently large $t$,
\begin{align}\label{cluster.stay.out}
N_{2,1}(s,t)\ge c_{22}a(t).
\end{align}

For each possible triple $(h,i,v)$, let $P_{h,i,v}(s,t)$ be the sum of conditional expectations in \eqref{cluster.forward.sum},
with $\mathfrak{t}$ and $\mathfrak{s}$ replaced by $s$ and $t$,
and let $Q_{h,i,v}(s,t)$ be the sum of the corresponding indicators.
By \eqref{cluster.forward.prob} and the argument for \eqref{cluster.in.comp}, we have
$$
P_{h,i,v}(s,t) \le c_{23}\eps^{\beta-1}L(U_i,t) \le c_{23}\eps^{\beta-1}a(t).
$$
If $P_{h,i,v}(s,\infty)<\infty,$ then by the conditional Borel--Cantelli lemma, we have
$Q_{h,i,v}(s,\infty)<\infty$ almost surely.
If $P_{h,i,v}(s,t)\to\infty,$ then
$Q_{h,i,v}(s,t)-P_{h,i,v}(s,t)$
is a martingale whose predictable quadratic variation is bounded by $P_{h,i,v}(s,t)$, and hence, by the martingale strong law, we have
$$
Q_{h,i,v}(s,t) = P_{h,i,v}(s,t)+o(P_{h,i,v}(s,t)).
$$
Since $a(t)\to\infty$, summing over the finitely many possible triples $(h,i,v)$, we get
\begin{align}\label{cluster.stay.in}
N_{1,2}(s,t) \le c_{24}\eps^{\beta-1}a(t)+o(a(t))
\end{align}
almost surely. Choose $\eps$ sufficiently small that
$c_{24}\eps^{\beta-1}<\frac{c_{22}}{2}.$
Since $a(t)\to\infty$, \eqref{cluster.stay.out} and \eqref{cluster.stay.in} contradict
$|N_{2,1}(s,t)-N_{1,2}(s,t)|\le2.$
Thus \eqref{cluster.stay} cannot occur.

\textbf{Step 5.} We now prove \eqref{cluster.claim}. For each $q\in\N$, set
$\mathfrak{t}_q:= \inf\left\{ t\ge n: b(t)\ge q\text{ and }a(t)\le\eps b(t) \right\}$
and, on $\{\mathfrak{t}_q<\infty\}$, set
$
\mathfrak{s}_q:= \inf\left\{ t\ge\mathfrak{t}_q: a(t)\ge K\eps b(t) \right\}.
$
By \eqref{cluster.upcross}, for all sufficiently large $q$, we have
\begin{align*}
\P(\mathfrak{t}_q<T,\ \mathfrak{s}_q<T) &= \E\left[ \1_{\{\mathfrak{t}_q<T\}} \P(\mathfrak{s}_q<T\mid\mathcal{F}_{\mathfrak{t}_q}) \right]\le c_{20}e^{-c_{21}\eps^\beta q}.
\end{align*}
Hence
$$
\sum_{q=1}^\infty\P(\mathfrak{t}_q<T,\ \mathfrak{s}_q<T)<\infty.
$$
By the Borel--Cantelli lemma, almost surely only finitely many of the events $\{\mathfrak{t}_q<T,\ \mathfrak{s}_q<T\}$ occur.
Suppose, on the contrary, that on $\{T=\infty\}\cap\{U_1,\ldots,U_j\text{ are the clusters of }V^*\}$,
$$
\liminf_{t\to\infty}\frac{a(t)}{b(t)}=0.
$$
On this event, $a(t)\to\infty$ and $b(t)\to\infty$. Hence $\mathfrak{t}_q<\infty$ for all sufficiently large $q$. Since $T=\infty$, and since \eqref{cluster.stay} cannot occur, we also have $\mathfrak{s}_q<\infty$ for all sufficiently large $q$. This contradicts the Borel--Cantelli lemma and proves \eqref{cluster.claim}.

\textbf{Step 6.} We now complete the proof of the proposition. We reason by contradiction and assume that
\begin{align}\label{cluster.contradiction}
\liminf_{t\to\infty}\frac{L(\Ccal_J(t),t)}{L(\Ccal_1(t),t)}=0.
\end{align}
On $\{J=j\}$, with $j\ge2$, there exists a sequence $(t_r)_{r\ge1}$ tending to infinity such that $L(\Ccal_j(t_r),t_r)/{L(\Ccal_1(t_r),t_r)}\rightarrow0.$
Since
$$
\frac{L(\Ccal_j(t_r),t_r)}{L(\Ccal_1(t_r),t_r)}=\prod_{i=1}^{j-1}\frac{L(\Ccal_{i+1}(t_r),t_r)}{L(\Ccal_i(t_r),t_r)},
$$
there exist $1\le k\le j-1$ and, after taking a subsequence if necessary, we have
\begin{align}\label{cluster.final.ratio}
\frac{L(\Ccal_{k+1}(t_r),t_r)}{L(\Ccal_k(t_r),t_r)}\rightarrow0.
\end{align}
There are only finitely many possible clusters. Hence, after taking a further subsequence, we may assume that there exist pairwise disjoint connected subsets $U_1,\ldots,U_j$ such that
$\Ccal_i(t_r)=U_i$ for $1\le i\le j$, for every $r$.

By \eqref{cluster.distance},
$
\delta\left(\bigcup_{i=1}^kU_i,\bigcup_{i=k+1}^jU_i\right)=2.
$
For the corresponding functions $a$ and $b$ defined above, we have
$
a(t_r)\le L(\Ccal_{k+1}(t_r),t_r)
$
and
$
b(t_r)=L(\Ccal_k(t_r),t_r).
$
It follows from \eqref{cluster.final.ratio} that
$$
\frac{a(t_r)}{b(t_r)}\rightarrow0.
$$
On the event that $U_1,\ldots,U_j$ are the clusters of $V^*$, we have
$L(D,\infty)<\infty.$
By Lemma \ref{lem.jump}(a), almost surely there are only finitely many jumps between two vertices of $D$. Furthermore, by Lemma \ref{lem.jump}(c)-(d), almost surely \eqref{cluster.jump.bound}-\eqref{cluster.weight.bound} holds for every $1\le i\le j$ and every $v\in D$ with $\delta(v,U_i)=1$ for all sufficiently large $t$. Consequently, almost surely there exist $m,n\in\N$, with $m\ge\ell^*$, such that $T=\infty$.

For each fixed $j,k,\mathcal{U},m,n$, \eqref{cluster.claim} implies
$$
\P\left(
T=\infty,\ U_1,\ldots,U_j\text{ are the clusters of }V^*,\ 
\liminf_{t\to\infty}\frac{a(t)}{b(t)}=0
\right)=0.
$$
There are only finitely many possible values of $j,k$ and $\mathcal{U}$, while $m,n\in\N$. Taking the corresponding countable union contradicts \eqref{cluster.contradiction}. Therefore
$$
\liminf_{t\to\infty}\frac{L(\Ccal_J(t),t)}{L(\Ccal_1(t),t)}>0 \quad\text{a.s.}
$$
\end{proof}

We combine the above results to demonstrate Theorem \ref{localization}.

\begin{proof}[Proof of Theorem \ref{localization}]
In virtue of Proposition~\ref{dist2.lim}, for each cluster $U$ of $V^*$, we have almost surely
$$
\liminf_{t\to\infty}\frac{L(U,t)}{L(\Ccal_1(t),t)}>0,
$$
where we recall that $\Ccal_1(t)$ is the cluster of $V^*$ having the largest local time at time $t$. Moreover,
$$
\sum_{U:\,U\text{ is a cluster of }V^*}L(U,t)
=L(V^*,t) =t+\sum_{v\in V}\ell_v-L(V\setminus V^*,t).
$$
Since $L(V\setminus V^*,\infty)<\infty$ and there are at most $|V|$ clusters, we have
$$
\liminf_{t\to\infty}\frac{L(\Ccal_1(t),t)}{t}\ge \frac{1}{|V|}>0.
$$
Consequently, for every cluster $U$ of $V^*$,
\begin{align}\label{liminf}
\liminf_{t\to\infty}\frac{L(U,t)}{t}>0.
\end{align}
Combining \eqref{liminf} with Lemma \ref{lem.jump}(b), for each vertex $i\in V^*$, we have almost surely
\begin{align}\label{liminf.vertex}
\liminf_{t\to\infty}\frac{L(i,t)}{t}>0.
\end{align}

We next show that every cluster of $V^*$ consists of a single vertex. Suppose otherwise. Since a cluster is connected, there exist $i,j\in V^*$ belonging to the same cluster such that $i\sim j$. By \eqref{liminf.vertex}, there exists a random constant $c>0$ such that, for all sufficiently large $t$,
$$
L(i,t)\wedge L(j,t)\ge ct.
$$
Using this bound and the fact that $W(t)\le |V| w(t+|\ell|)$, we obtain, for all sufficiently large $t$,
$$
\frac{w(L(i,t))\wedge w(L(j,t))}{W(t)} \ge \frac{w(ct)}{|V| w(t+|\ell|)}
$$
Using \eqref{lower.w}, it follows that
$$
\liminf_{t\to\infty}
\frac{w(L(i,t))\wedge w(L(j,t))}{W(t)}>0,
$$
which contradicts Proposition~\ref{thm1}. Hence every cluster of $V^*$ consists of a single vertex. In particular,
\begin{align}\label{Vstar.no.neighbour}
\delta(u, v)\ge 2
\quad\text{for all distinct }u,v\in V^*.
\end{align}

Assume now that $V^*$ contains more than one vertex. By Lemma \ref{lem.jump}(b), for any cluster $U$ of $V^*$, we have
$
\delta(U,V^*\setminus U)=2.
$
Since every cluster is a singleton, there exist $i,j\in V^*$ such that $\delta(i,j)=2$. Combining \eqref{liminf.vertex} and \eqref{Vstar.no.neighbour}, we obtain
$$
\liminf_{t\to\infty}\frac{L(i,t)\wedge L(j,t)}{t}>0 \quad\text{and}\quad u\not\sim v\text{ for all distinct }u,v\in V^*.
$$
This contradicts Proposition~\ref{thm2}. Therefore $V^*$ contains a single vertex, which proves the theorem.
\end{proof}

We now turn to the proof of our main theorem.
 
\begin{proof}[Proof of Theorem \ref{thm:main}(b)]
Let $\Xi=\{\exists! v\in \mathcal{G} : L(v,\infty)=\infty\}$. Recall the definitions of  $V^{\ssup 0}_{k}$ and $V^{\ssup 1}_{k}$ from \eqref{def.V}. 
Let $\mathcal G_k$ be the connected component containing $\rho$ of the subgraph induced by $V_{3k}^{\ssup 0}$. Denote by $\vec{E}_k$ the set of all induced directed edges of $\mathcal{G}_k$ and let $\widetilde\ell$ be the restriction of $\ell$ to its vertex set.
Recall that $(\chi^{e}_j)_{j\in \N, e\in \vec{E}}$ is the collection of exponential random variables used in the construction of the process $\X=(X_t)_{t\in \R_+}$ mentioned in Section \ref{Se:strongconstr}. 
Let $\widetilde\X=(\widetilde X_t)_{t\ge0}$ be the ${\rm VRJP}(\widetilde\ell,w)$ on $\mathcal G_k$, started at $\widetilde X_0=\rho$ and constructed using $(\chi_j^e)_{j\in\N,e\in\vec E_k}$. Note that $X_t=\widetilde{X}_{t}$ for all $0\le t< \widehat{T}_{3k}$, where $\widehat{T}_{3k}$ is the first time when the process $\X$ hits  $(V_{3k}^{\ssup 0})^c$.
Hence, on $\Xi^c\cap\{\widehat T_{3k}=\infty\}$, there exist distinct vertices of $\mathcal G_k$ with unbounded local times for $\widetilde\X$. Applying Theorem \ref{localization} to the process $\widetilde\X$, we infer that \begin{align}\label{prob.inq1}\P(\Xi^c\cap \{\widehat{T}_{3k}=\infty\})=0.\end{align}
Recall that $T_{k}$ is the first hitting time to $V^{\ssup 1}_{k}$. In virtue of Theorem \ref{them:loc}, we have that for each $k\ge 2 $, \begin{align}\label{prob.inq2}
    \P(\widehat{T}_{3k}=\infty)\ge \P(T_{3(k-1)}=\infty) \ge 1- \gamma^{k-1}
    \end{align}
    with some $\gamma\in (0,1)$. 
Combining  \eqref{prob.inq1} and \eqref{prob.inq2}, we have $$\P(\Xi\cap \{\widehat{T}_{3k}=\infty\})=\P(\widehat{T}_{3k}=\infty)\ge 1- \gamma^{k-1}.$$ Taking $k\to\infty$ and using the continuity of the probability measure, we conclude that $\P(\Xi)=1$.

%
\end{proof}

\section{Appendix}

\subsection{Matrix-tree theorem for weighted directed graphs}\label{sec:mat-tree}
Let $G=(V,\vec{E},w)$ be a weighted directed graph, where $V=\{1,2,\dots, n\}$ is the set of vertices, $\vec{E}$ is the set of  directed edges and $w:\vec{E}\to (0,\infty)$ is a weight function. We assume that $G$ is connected and loopless.

Let $L=(L_{ij})_{1\le i, j\le n}$ be the outgoing Laplacian matrix of $G$, i.e.
$$
L_{ij}=\left\{\begin{matrix}
- w(i,j)& \text{if  $(i,j)\in \vec{E}$}, \\
\displaystyle\sum_{k\; :\; (i,k)\in \vec{E}}w(i,k) & \text{if $i=j$},\\
0 & \text{otherwise.}
\end{matrix}\right.$$

For each $j\in V$, we say that a directed subgraph $T=(V_T,\vec{E}_T,w)$ is a weighted incoming directed spanning tree rooted at vertex $j$ if $\vec{E}_T$ is a minimal subset of $\vec{E}$ such that $V_T=V$ and there is a directed path from every vertex to the root $j$. Similarly, one can also define a weighted incoming directed spanning forest rooted at certain vertices in $V$. 

Denote by $\vec{\mathfrak{T}}_j$ the set of all weighted incoming directed spanning trees rooted at $j$.
For $i, j, k \in V$ and $k\notin\{i,j\}$, we denote by $\vec{\mathfrak{T}}_{ij,k}$ the set of all weighted incoming directed spanning forest which consists of two trees rooted at $i$ and $k$ such that the first one must contain $j$. We use the convention that $\vec{\mathfrak{T}}_{ij,k}=\emptyset$ if $k\in\{i,j\}$.

The weight $w(H)$ of a weighted
directed graph $H=(V_H, \vec{E}_H, w)$ is defined as the product of the weights of its directed edges, i.e.
$w(H)=\prod_{e\in \vec{E}_H}w(e).$

\begin{proposition}[Matrix-tree theorem for weighted directed graphs]\label{mat-tree.thm}
Let $L(i,j)\in \mathbb{R}^{(n-1)\times (n-1)}$ be the matrix obtained from $L$ by deleting the $i$-th row and the $j$-th column. Then 
$$\det(L(i,j))=(-1)^{i+j}\sum_{T\in \vec{\mathfrak{T}}_i} w(T).$$
Furthermore, for each $k\neq i$ and $h\neq j$, the $(k,h)$-cofactor of the matrix $L(i,j)$ is given by
$$C_{k,h}(i,j)=(-1)^{i+j}\left(
\sum_{F\in \vec{\mathfrak{T}}_{kh,i}}w(F)-\sum_{F\in \vec{\mathfrak{T}}_{kj,i}}w(F)\right).$$
\end{proposition}

See Section VI.5 and Section VI.6 in \cite{Tutte1984} for more details.

\subsection{Other useful results}

\begin{lemma}\label{MatExp}
Let $A=\begin{pmatrix}A_{11} & A_{12}\\0 & A_{22}\end{pmatrix}$, where $A_{11}$ and $A_{22}$ are square blocks. We have that $$e^{tA}=\begin{pmatrix}e^{tA_{11}} & B(t)\\0 & e^{tA_{22}}\end{pmatrix},$$ where
$$B(t)=\int_0^te^{(t-s)A_{11}}\cdot A_{12}\cdot e^{sA_{22}}\rmd s.$$
In particular, when $A_{11}=0$ and $A_{22}$ is non-singular, 
$$e^{tA}=\begin{pmatrix}I_1 & A_{12}\cdot (e^{tA_{22}}-I_2)\cdot A_{22}^{-1} \\0 & e^{tA_{22}}\end{pmatrix},$$
where $I_1$ and $I_2$ are identity matrices. 
\end{lemma}

\begin{proof}
 Differentiating $e^{tA}$ w.r.t. $t$, we obtain that 
$$
\begin{pmatrix}A_{11}e^{tA_{11}}&B'(t)\\ 0 & A_{22}e^{tA_{22}}\end{pmatrix}=\frac{\rmd}{\rmd t}e^{tA}=Ae^{tA}=\begin{pmatrix}A_{11}&A_{12}\\ 0&A_{22}\end{pmatrix}\begin{pmatrix} e^{tA_{11}}&B(t)\\ 0&e^{tA_{22}}\end{pmatrix}.
$$
Therefore
$B(t)$ is the solution to the initial value problem: $B(0)=0$, and
$B'(t)=A_{11}B(t)+A_{12}e^{tA_{22}}.$
\end{proof}

\begin{lemma}[Sherman-Morrison formula]\label{SM.formula}
Suppose ${A\in \mathbb {R} ^{n\times n}}$ is a non-singular square matrix and ${ u,v\in \mathbb {R} ^{n}}$ are column vectors. Then
\begin{align*} \det \left( {A} + {uv} ^{\textsf {T}}\right)&=\left(1+ {v} ^{\textsf {T}} {A} ^{-1} {u} \right)\,\det \left( {A} \right). \end{align*}
Furthermore, if $1+v^{\textsf {T}}A^{-1}u\neq 0$ then 
\begin{align*}
 \left(A+uv^{\textsf {T}}\right)^{-1}&=A^{-1}-\frac{A^{-1}uv^{\textsf {T}}A^{-1}}{1+v^{\textsf {T}}A^{-1}u}.\end{align*}
\end{lemma}

See, e.g. \cite{Hager1989}.

The two next lemmas were mentioned by Stanislav  Volkov (private communication).

\begin{lemma}\label{lem.sum} Assume that $\int_c^{\infty}\frac{\rmd u}{w(u)}<\infty$ and $w_*:=\inf_{u\ge c}w(u)>0$ with some deterministic constant $c>0$. Let $\lambda>0$ be a deterministic constant, and let $\xi_j,\eta_j$ with $j\in \N$ be i.i.d. exponential random variables with parameter 1. Then
$$S_n:= \sum_{j=1}^n \frac{\eta_j}{w\left(c+\lambda^{-1}\sum_{k=1}^j\xi_k \right)}$$
converges a.s. to a finite random variable $S_{\infty}$ with $$\E[e^{y S_{\infty}}]=\exp\left( \lambda y \int_c^{\infty} \frac{\rmd x}{w(x)-y} \right) \quad\text{for $y\in (-\infty,w_*)$}\quad  \text{and}\quad \E[S_{\infty}]=\lambda\int_c^{\infty} \frac{\rmd x}{w(x)}.$$
\end{lemma}
\begin{proof}
Set $T_j:=\sum_{k=1}^j\xi_k$, $j\ge1$. Then $(T_j)_{j\ge1}$ are the points of a Poisson process with rate 1 on $\R_+$. Since $(\eta_j)_{j\ge1}$ are independent of $(\xi_j)_{j\ge1}$, by the marking theorem (see, e.g., Theorem 5.6 in \cite{LP2018}),  $(T_j,\eta_j)_{j\ge1}$ form a Poisson process on $\R_+^2$ with intensity
$\rmd t\, e^{-u}\rmd u$.
Since all terms are non-negative, we have $S_n$ tends almost surely to
$$
S_{\infty}:=\sum_{j=1}^{\infty}\frac{\eta_j}{w(c+\lambda^{-1}T_j)}.
$$
By Campbell's formula (see Proposition 2.7 in \cite{LP2018}), we have
\begin{align*}
\E[S_{\infty}]
&=\int_0^\infty\int_0^\infty \frac{u}{w(c+\lambda^{-1}t)}e^{-u}\rmd u\rmd t\\
&=\int_0^\infty\frac{\rmd t}{w(c+\lambda^{-1}t)} =\lambda\int_c^\infty\frac{\rmd x}{w(x)}<\infty.
\end{align*}
Hence $S_{\infty}<\infty$ almost surely. Fix $y\in(-\infty,w_*)$. For every $t\ge0$, we notice that
\begin{align*}
\E\left[\exp\left(\frac{y\eta_1}{w(c+\lambda^{-1}t)}\right)\right] &=\int_0^\infty \exp\left(\frac{yu}{w(c+\lambda^{-1}t)}\right)e^{-u}\rmd u\\
&=\frac{1}{1-\frac{y}{w(c+\lambda^{-1}t)}}.
\end{align*}
Moreover,
$$
\int_c^\infty\left|\frac{y}{w(x)-y}\right|\rmd x<\infty.
$$
Indeed, if $0\le y<w_*$, then
$\frac{1}{w(x)-y}\le\frac{w_*}{w_*-y}\frac{1}{w(x)}$, while if $y<0$, then
$\frac{1}{w(x)-y}\le\frac{1}{w(x)}.$
Therefore, by the exponential formula for Poisson processes (see Theorem 3.9 and Exercise 3.4 in \cite{LP2018}),
\begin{align*}
\E[e^{yS_{\infty}}] &=\exp\left\{\int_0^\infty \left( \E\left[\exp\left(\frac{y\eta_1}{w(c+\lambda^{-1}t)}\right)\right]-1 \right)\rmd t\right\}\\
&=\exp\left\{\int_0^\infty \frac{y}{w(c+\lambda^{-1}t)-y}\rmd t\right\}=\exp\left\{\lambda y\int_c^\infty\frac{\rmd x}{w(x)-y}\right\}.
\end{align*} 
\end{proof}

\begin{lemma}\label{gamma.seq}
Let $\alpha>1, \nu>0, p>1$ and $q>0$ be fixed real numbers and let $(a_k)_{k\ge 0}$ be a positive sequence such that 
\begin{itemize}
    \item $0<a_0<q^{-1/(\alpha-1)}$ and
    \item for all $k\ge 0$,
    \begin{align}\label{def.a}
    a_{k+1}\le a_k\left(1-\frac{p}{\nu+k}\left( 1-q a_k^{\alpha-1}\right)+r_k\right),
    \end{align}
    where $(r_k)_{k\ge0}$ is some sequence such that $\sum_{k=0}^{\infty}r_k<\infty$ and $r_k\le \frac{p}{\nu+k}\left( 1-q a_0^{\alpha-1}\right)$ for all $k\ge 0$.
\end{itemize}
Then $\limsup_{k\to\infty} (\nu+k)^p a_k<\infty.$ 
\end{lemma}

\begin{proof}
By the assumptions of the lemma, the sequence $(a_k)_{k\ge 0}$ is positive and non-increasing. Indeed, if $a_k\le a_0$, then
$$
-\frac{p}{\nu+k}\left(1-q a_k^{\alpha-1}\right)+r_k \le -\frac{p}{\nu+k}\left(1-q a_0^{\alpha-1}\right)+r_k \le0.
$$
Hence, by induction,
\begin{align}\label{a.seq}
    0<a_k\le a_0<q^{-1/(\alpha-1)}\quad \text{for all $k\ge 0$}.
\end{align}
Iterating \eqref{def.a}, we obtain
 \begin{align*}
     a_k\le a_0 \prod_{j=0}^{k-1}\left(1-\frac{p}{\nu+j}\left( 1-q a_j^{\alpha-1}\right)+r_j\right).
 \end{align*}
Using \eqref{a.seq} and the fact that $1+x\le e^{x}$ for all $x\in \R$, we thus have
\begin{align*}
    a_k     &\le a_0 \exp\left(-p\left(1-qa_0^{\alpha-1}\right)\sum_{j=0}^{k-1}\frac{1}{\nu+j}+\sum_{j=0}^{k-1}r_j\right)\le K_1(\nu+k)^{-\eps},
\end{align*}
where we set $\eps:=p(1-qa_0^{\alpha-1})>0$ and $K_1$ is some positive constant.
Hence
\begin{align*}
a_k &\le a_0 \exp\left(-p\sum_{j=0}^{k-1}\frac{1}{\nu+j} +pq\sum_{j=0}^{k-1}\frac{a_j^{\alpha-1}}{\nu+j} +\sum_{j=0}^{k-1}r_j\right)\\
&\le a_0 \exp\left(-p\sum_{j=0}^{k-1}\frac{1}{\nu+j} +pqK_1^{\alpha-1}\sum_{j=0}^{k-1}(\nu+j)^{-1-\eps(\alpha-1)} +\sum_{j=0}^{k-1}r_j\right)\\
&\le K_2(\nu+k)^{-p},
\end{align*}
where $K_2$ is some positive constant, since $\eps(\alpha-1)>0$ and $\sum_{j=0}^{\infty}r_j<\infty$.
This ends the proof of the lemma.
\end{proof}

\begin{lemma}\label{lem.integral}
Let $\alpha>1$, let $0\le s<t$, and let $f,g:[s,t]\to(0,\infty)$ be non-decreasing absolutely continuous functions. Let $N:[s,t]\to\R_+$ be a non-decreasing c\`adl\`ag function with $N(s)=0$. Assume that, for some $\eta, c>0$,
$$
f(r)\le\eta g(r) \quad\text{and}\quad N(r)\le cg(r) \quad\text{for all }r\in[s,t].
$$
Then
\begin{align}\label{cluster.integral}
\int_{(s,t]} \left(\frac{f(r)}{g(r)}\right)^\alpha\rmd N(r) \le c\left(1+\frac{\alpha^2}{\alpha-1}\right) \eta^{\alpha-1}f(t).
\end{align}
\end{lemma}

\begin{proof}
Set $H(r):=(f(r)/g(r))^\alpha$. Since $H$ is continuous, taking integration by parts, we have
$$
\int_{(s,t]}H(r)\rmd N(r) = H(t)N(t)-\int_s^tN(r)\rmd H(r).
$$
Moreover,
$
\rmd H(r) = \alpha f(r)^{\alpha-1}g(r)^{-\alpha}\rmd f(r) -\alpha f(r)^\alpha g(r)^{-\alpha-1}\rmd g(r).
$
Since $f$ and $N$ are non-decreasing and $N(r)\le cg(r)$, we have
\begin{align*}
\int_{(s,t]}H(r)\rmd N(r) &\le
H(t)N(t) + \alpha c\int_s^t \left(\frac{f(r)}{g(r)}\right)^\alpha\rmd g(r).
\end{align*}
By the assumptions,
$$
H(t)N(t) \le c\left(\frac{f(t)}{g(t)}\right)^{\alpha-1}f(t) \le c\eta^{\alpha-1}f(t).
$$
Furthermore, since $f$ is non-decreasing,
$
\left(\frac{f(r)}{g(r)}\right)^\alpha \le \eta^\alpha\wedge\frac{f(t)^\alpha}{g(r)^\alpha}.
$
Therefore
\begin{align*}
\int_s^t
\left(\frac{f(r)}{g(r)}\right)^\alpha\rmd g(r) &\le \int_{g(s)}^{g(t)} \left(\eta^\alpha\wedge\frac{f(t)^\alpha}{x^\alpha}\right)\rmd x\\
&\le
\int_0^\infty \left(\eta^\alpha\wedge\frac{f(t)^\alpha}{x^\alpha}\right)\rmd x= \frac{\alpha}{\alpha-1}\eta^{\alpha-1}f(t).
\end{align*}
Combining the above estimates proves \eqref{cluster.integral}.
\end{proof}

\begin{lemma}\label{lem.exp}
Let $E_1,\ldots,E_m$ be independent exponential random variables with mean one, and let $a_1,\ldots,a_m$ be deterministic positive constants. Set
$$T:=\min_{1\le i\le m}\frac{E_i}{a_i}, \quad J:=\arg\min_{1\le i\le m}\frac{E_i}{a_i}.$$
Then, conditionally on $(T,J)$, the random variables
$(E_i-a_iT)_{1\le i\le m, i\ne J}$ are independent exponentially distributed with mean one.
 \end{lemma}

\begin{proof}
For $t\ge0$, $j\in\{1,\ldots,m\}$ and $x_i\ge0$ with $1\le i\le m$, we have
    \begin{align*}
&\P\big(T>t,\ J=j,\ E_i-a_iT>x_i\text{ for all }i\ne j\big)\\
& \quad =\P \left(\frac{E_j}{a_j}>t,\, E_i>\frac{a_i}{a_j}E_j+x_i\ \text{for every }i\ne j\right). \\
&\quad=
\int_t^\infty a_j e^{-a_js}
\prod_{i\ne j}e^{-a_is-x_i}\rmd s= \frac{a_j}{\sum_{i=1}^m a_i}
e^{-t\sum_{i=1}^m a_i}
\prod_{i\ne j}e^{-x_i}.
\end{align*}
In particular, when $x_i=0$ for $1\le i\le m$, we get
$$\P\big(T>t,\ J=j\big)=\frac{a_j}{\sum_{i=1}^m a_i}
e^{-t\sum_{i=1}^m a_i}.$$
Using the conditional probability formula, we obtain the desired result. 
\end{proof}

\begin{lemma}\label{lem:Chernoff}
There exists a universal constant $c>0$ such that the following holds.
Let $(\xi_n)_{n\in\N}$ be a sequence of $[0,1]$-valued random variables adapted to a filtration $(\mathcal{G}_n)_{n\in\mathbb Z_+}$, and set $p_n:=\E(\xi_n\mid\mathcal{G}_{n-1})$ for $n\in\N$.
Then, for any finite stopping times $K\le L$ and any $x>0$,
\begin{align}
\P\Bigg( \sum_{n=K+1}^{L}\xi_n \le \frac12\sum_{n=K+1}^{L}p_n,\ \sum_{n=K+1}^{L}p_n\ge x \mid\mathcal{G}_K \Bigg)
&\le e^{-cx},
\label{Chernoff.lower}\\
\P\Bigg( \sum_{n=K+1}^{L}\xi_n \ge 2x,\ \sum_{n=K+1}^{L}p_n\le x \mid \mathcal{G}_K \Bigg) &\le e^{-cx}.
\label{Chernoff.upper}
\end{align}
\end{lemma}

\begin{proof}
For any $\theta\in\R$, using $0\le\xi_n\le1$ and the convexity of the function
$u\mapsto e^{\theta u}$, we have
\begin{align*}
\E\left(e^{\theta\xi_n}\mid\mathcal{G}_{n-1}\right) &\le \E\left(1+\xi_n(e^\theta-1)\mid\mathcal{G}_{n-1}\right)= 1+p_n(e^\theta-1)\le \exp\{p_n(e^\theta-1)\}.
\end{align*}
Consequently, $\exp\left\{ \theta\sum_{n=1}^{k}\xi_n -(e^\theta-1)\sum_{n=1}^{k}p_n
\right\}$, for $k\in\mathbb Z_+$, is a non-negative supermartingale.

To prove \eqref{Chernoff.lower}, take $\theta=-\log2$. On the event
$\left\{ \sum_{n=K+1}^{L}\xi_n \le \frac12\sum_{n=K+1}^{L}p_n,\ \sum_{n=K+1}^{L}p_n\ge x \right\},$
we have
\begin{align*}
\theta\sum_{n=K+1}^{L}\xi_n -(e^\theta-1)\sum_{n=K+1}^{L}p_n
&\ge \left( \frac{\theta}{2}-(e^\theta-1) \right) \sum_{n=K+1}^{L}p_n\\
&= \frac{1-\log2}{2} \sum_{n=K+1}^{L}p_n\ge \frac{1-\log2}{2}x.
\end{align*}
Hence, by the optional sampling theorem and Markov's inequality,
$$
\P\Bigg( \sum_{n=K+1}^{L}\xi_n \le \frac12\sum_{n=K+1}^{L}p_n,\ \sum_{n=K+1}^{L}p_n\ge x \mid \mathcal{G}_K \Bigg) \le \exp\left\{-\frac{1-\log2}{2}x\right\}.
$$

Similarly, to prove \eqref{Chernoff.upper}, take $\theta=\log2$. On the event $\left\{ \sum_{n=K+1}^{L}\xi_n\ge2x,\ \sum_{n=K+1}^{L}p_n\le x \right\},$
we have
\begin{align*}
\theta\sum_{n=K+1}^{L}\xi_n -(e^\theta-1)\sum_{n=K+1}^{L}p_n &\ge (2\log2-1)x.
\end{align*}
Therefore,
$$
\P\Bigg( \sum_{n=K+1}^{L}\xi_n \ge2x,\ \sum_{n=K+1}^{L}p_n\le x \mid \mathcal{G}_K \Bigg) \le e^{-(2\log2-1)x}.
$$
This completes the proof.
\end{proof}
{\bf Acknowledgements.}  This  work was partially supported by Australian Research Council grant  ARC DP230102209.

\bibliographystyle{amsplain}
\bibliography{mybib}

@book {LP2018,
    AUTHOR = {Last, G\"unter and Penrose, Mathew},
     TITLE = {Lectures on the {P}oisson process},
    SERIES = {Institute of Mathematical Statistics Textbooks},
    VOLUME = {7},
 PUBLISHER = {Cambridge University Press, Cambridge},
      YEAR = {2018},
     PAGES = {xx+293},
      ISBN = {978-1-107-45843-7; 978-1-107-08801-6},
   MRCLASS = {60G55 (60D05 60F05 60G57 60H05)},
  MRNUMBER = {3791470},
MRREVIEWER = {Christoph\ Th\"ale},
}

@article {BRS2013,
    AUTHOR = {Benaim, Michel and Raimond, Olivier and Schapira, Bruno},
     TITLE = {Strongly vertex-reinforced-random-walk on a complete graph},
   JOURNAL = {ALEA Lat. Am. J. Probab. Math. Stat.},
  FJOURNAL = {ALEA. Latin American Journal of Probability and Mathematical
              Statistics},
    VOLUME = {10},
      YEAR = {2013},
    NUMBER = {2},
     PAGES = {767--782},
   MRCLASS = {60G50 (60K35 62L20)},
  MRNUMBER = {3125746},
MRREVIEWER = {Dimitri Petritis},
}

@article {DV2002,
    AUTHOR = {Davis, Burgess and Volkov, Stanislav},
     TITLE = {Continuous time vertex-reinforced jump processes},
   JOURNAL = {Probab. Theory Related Fields},
  FJOURNAL = {Probability Theory and Related Fields},
    VOLUME = {123},
      YEAR = {2002},
    NUMBER = {2},
     PAGES = {281--300},
      ISSN = {0178-8051},
   MRCLASS = {60G17 (60G20)},
  MRNUMBER = {1900324},
MRREVIEWER = {Randall James Swift},
       DOI = {10.1007/s004400100189},
       URL = {https://doi.org/10.1007/s004400100189},
}

@article {DV2004,
    AUTHOR = {Davis, Burgess and Volkov, Stanislav},
     TITLE = {Vertex-reinforced jump processes on trees and finite graphs},
   JOURNAL = {Probab. Theory Related Fields},
  FJOURNAL = {Probability Theory and Related Fields},
    VOLUME = {128},
      YEAR = {2004},
    NUMBER = {1},
     PAGES = {42--62},
      ISSN = {0178-8051},
   MRCLASS = {60J75 (60G17)},
  MRNUMBER = {2027294},
MRREVIEWER = {Aleksey B. Piunovski\u{\i}},
       DOI = {10.1007/s00440-003-0286-y},
       URL = {https://doi.org/10.1007/s00440-003-0286-y},
}

@article {NR2021,
    AUTHOR = {Raimond, Olivier and Nguyen, Tuan-Minh},
     TITLE = {Strongly vertex-reinforced jump process on a complete graph},
   JOURNAL = {Ann. Inst. Henri Poincar\'{e} Probab. Stat.},
  FJOURNAL = {Annales de l'Institut Henri Poincar\'{e} Probabilit\'{e}s et
              Statistiques},
    VOLUME = {57},
      YEAR = {2021},
    NUMBER = {3},
     PAGES = {1549--1568},
      ISSN = {0246-0203},
   MRCLASS = {60J55 (60J76)},
  MRNUMBER = {4291445},
       DOI = {10.1214/20-aihp1115},
       URL = {https://doi.org/10.1214/20-aihp1115},
}

@article {ST2015,
    AUTHOR = {Sabot, Christophe and Tarr\`es, Pierre},
     TITLE = {Edge-reinforced random walk, vertex-reinforced jump process
              and the supersymmetric hyperbolic sigma model},
   JOURNAL = {J. Eur. Math. Soc. (JEMS)},
  FJOURNAL = {Journal of the European Mathematical Society (JEMS)},
    VOLUME = {17},
      YEAR = {2015},
    NUMBER = {9},
     PAGES = {2353--2378},
      ISSN = {1435-9855},
   MRCLASS = {60K35 (60K37 81Q35 81T25 81T60)},
  MRNUMBER = {3420510},
MRREVIEWER = {Daniel Boivin},
       DOI = {10.4171/JEMS/559},
       URL = {https://doi.org/10.4171/JEMS/559},
}

@book {Tutte1984,
    AUTHOR = {Tutte, W. T.},
     TITLE = {Graph theory},
    SERIES = {Encyclopedia of Mathematics and its Applications},
    VOLUME = {21},
 PUBLISHER = {Addison-Wesley Publishing Company, Advanced Book Program,
              Reading, MA},
      YEAR = {1984},
     PAGES = {xxi+333},
      ISBN = {0-201-13520-5},
   MRCLASS = {05-01 (05Cxx)},
  MRNUMBER = {746795},
MRREVIEWER = {Frederick Hoffman},
}

@article {Hager1989,
    AUTHOR = {Hager, William W.},
     TITLE = {Updating the inverse of a matrix},
   JOURNAL = {SIAM Rev.},
  FJOURNAL = {SIAM Review. A Publication of the Society for Industrial and
              Applied Mathematics},
    VOLUME = {31},
      YEAR = {1989},
    NUMBER = {2},
     PAGES = {221--239},
      ISSN = {0036-1445},
   MRCLASS = {65Fxx},
  MRNUMBER = {997457},
       DOI = {10.1137/1031049},
       URL = {https://doi.org/10.1137/1031049},
}

@article{CNV2022,
    AUTHOR = {Collevecchio, Andrea and Nguyen, Tuan-Minh and Volkov,
              Stanislav},
     TITLE = {Vertex-reinforced jump process on the integers with nonlinear
              reinforcement},
   JOURNAL = {Ann. Appl. Probab.},
  FJOURNAL = {The Annals of Applied Probability},
    VOLUME = {32},
      YEAR = {2022},
    NUMBER = {4},
     PAGES = {2671--2705},
      ISSN = {1050-5164},
   MRCLASS = {60G17 (60G20 60K35)},
  MRNUMBER = {4474517},
       DOI = {10.1214/21-aap1743},
       URL = {https://doi.org/10.1214/21-aap1743},
}

@article {SZ2019,
    AUTHOR = {Sabot, Christophe and Zeng, Xiaolin},
     TITLE = {A random {S}chr\"{o}dinger operator associated with the vertex
              reinforced jump process on infinite graphs},
   JOURNAL = {J. Amer. Math. Soc.},
  FJOURNAL = {Journal of the American Mathematical Society},
    VOLUME = {32},
      YEAR = {2019},
    NUMBER = {2},
     PAGES = {311--349},
      ISSN = {0894-0347},
   MRCLASS = {60K35 (60K37 81T25 81T60 82B44)},
  MRNUMBER = {3904155},
MRREVIEWER = {Dimitri Petritis},
       DOI = {10.1090/jams/906},
       URL = {https://doi.org/10.1090/jams/906},
}

@article {ACK2014,
    AUTHOR = {Angel, Omer and Crawford, Nicholas and Kozma, Gady},
     TITLE = {Localization for linearly edge reinforced random walks},
   JOURNAL = {Duke Math. J.},
  FJOURNAL = {Duke Mathematical Journal},
    VOLUME = {163},
      YEAR = {2014},
    NUMBER = {5},
     PAGES = {889--921},
      ISSN = {0012-7094},
   MRCLASS = {60K35 (60K37)},
  MRNUMBER = {3189433},
MRREVIEWER = {Andrew R. Wade},
       DOI = {10.1215/00127094-2644357},
       URL = {https://doi.org/10.1215/00127094-2644357},
}

@book {BGL2014,
    AUTHOR = {Bakry, Dominique and Gentil, Ivan and Ledoux, Michel},
     TITLE = {Analysis and geometry of {M}arkov diffusion operators},
    SERIES = {Grundlehren der mathematischen Wissenschaften [Fundamental
              Principles of Mathematical Sciences]},
    VOLUME = {348},
 PUBLISHER = {Springer, Cham},
      YEAR = {2014},
     PAGES = {xx+552},
      ISBN = {978-3-319-00226-2; 978-3-319-00227-9},
   MRCLASS = {60J25 (58J65 60J35 60J60)},
  MRNUMBER = {3155209},
MRREVIEWER = {Ming Liao},
       DOI = {10.1007/978-3-319-00227-9},
       URL = {https://doi.org/10.1007/978-3-319-00227-9},
}

@article {Davis1990,
    AUTHOR = {Davis, Burgess},
     TITLE = {Reinforced random walk},
   JOURNAL = {Probab. Theory Related Fields},
  FJOURNAL = {Probability Theory and Related Fields},
    VOLUME = {84},
      YEAR = {1990},
    NUMBER = {2},
     PAGES = {203--229},
      ISSN = {0178-8051},
   MRCLASS = {60J15 (60G42 60J10)},
  MRNUMBER = {1030727},
MRREVIEWER = {Robin Pemantle},
       DOI = {10.1007/BF01197845},
       URL = {https://doi.org/10.1007/BF01197845},
}

@article {Sellke2008,
    AUTHOR = {Sellke, T.},
     TITLE = {Reinforced random walk on the {$d$}-dimensional integer
              lattice},
   JOURNAL = {Markov Process. Related Fields},
  FJOURNAL = {Markov Processes and Related Fields},
    VOLUME = {14},
      YEAR = {2008},
    NUMBER = {2},
     PAGES = {291--308},
      ISSN = {1024-2953},
   MRCLASS = {60K35 (60G50)},
  MRNUMBER = {2437533},
}

@article {CT2017,
    AUTHOR = {Cotar, Codina and Thacker, Debleena},
     TITLE = {Edge- and vertex-reinforced random walks with super-linear
              reinforcement on infinite graphs},
   JOURNAL = {Ann. Probab.},
  FJOURNAL = {The Annals of Probability},
    VOLUME = {45},
      YEAR = {2017},
    NUMBER = {4},
     PAGES = {2655--2706},
      ISSN = {0091-1798},
   MRCLASS = {60G50 (60J10 60K35)},
  MRNUMBER = {3693972},
MRREVIEWER = {Elcio Lebensztayn},
       DOI = {10.1214/16-AOP1122},
       URL = {https://doi.org/10.1214/16-AOP1122},
}

@article {LT2007,
    AUTHOR = {Limic, Vlada and Tarr\`es, Pierre},
     TITLE = {Attracting edge and strongly edge reinforced walks},
   JOURNAL = {Ann. Probab.},
  FJOURNAL = {The Annals of Probability},
    VOLUME = {35},
      YEAR = {2007},
    NUMBER = {5},
     PAGES = {1783--1806},
      ISSN = {0091-1798},
   MRCLASS = {60G50 (60J10 60K35)},
  MRNUMBER = {2349575},
MRREVIEWER = {Endre Cs\'{a}ki},
       DOI = {10.1214/009117906000001097},
       URL = {https://doi.org/10.1214/009117906000001097},
}

@incollection {LT2008,
    AUTHOR = {Limic, Vlada and Tarr\`es, Pierre},
     TITLE = {What is the difference between a square and a triangle?},
 BOOKTITLE = {In and out of equilibrium. 2},
    SERIES = {Progr. Probab.},
    VOLUME = {60},
     PAGES = {481--495},
 PUBLISHER = {Birkh\"{a}user, Basel},
      YEAR = {2008},
   MRCLASS = {60G50 (60J10 60K35)},
  MRNUMBER = {2477395},
       DOI = {10.1007/978-3-7643-8786-0\_22},
       URL = {https://doi.org/10.1007/978-3-7643-8786-0_22},
}

@article {Tarres2004,
    AUTHOR = {Tarr\`es, Pierre},
     TITLE = {Vertex-reinforced random walk on {$\mathbb{Z}$} eventually gets
              stuck on five points},
   JOURNAL = {Ann. Probab.},
  FJOURNAL = {The Annals of Probability},
    VOLUME = {32},
      YEAR = {2004},
    NUMBER = {3B},
     PAGES = {2650--2701},
      ISSN = {0091-1798},
   MRCLASS = {60G17 (34F05 60C05 60G50 60J10 60J20)},
  MRNUMBER = {2078554},
MRREVIEWER = {Endre Cs\'{a}ki},
       DOI = {10.1214/009117907000000694},
       URL = {https://doi.org/10.1214/009117907000000694},
}

@article {Volkov2001,
    AUTHOR = {Volkov, Stanislav},
     TITLE = {Vertex-reinforced random walk on arbitrary graphs},
   JOURNAL = {Ann. Probab.},
  FJOURNAL = {The Annals of Probability},
    VOLUME = {29},
      YEAR = {2001},
    NUMBER = {1},
     PAGES = {66--91},
      ISSN = {0091-1798},
   MRCLASS = {60K37 (60G50)},
  MRNUMBER = {1825142},
MRREVIEWER = {Michel Benaim},
       DOI = {10.1214/aop/1008956322},
       URL = {https://doi.org/10.1214/aop/1008956322},
}

@misc{1,
    title={Random walk with reinforcement},
    author={Coppersmith, D. and Diaconis, P.},
    year={1986},
    note={Unpublished manuscript},
}

@article {P1992,
    AUTHOR = {Pemantle, Robin},
     TITLE = {Vertex-reinforced random walk},
   JOURNAL = {Probab. Theory Related Fields},
  FJOURNAL = {Probability Theory and Related Fields},
    VOLUME = {92},
      YEAR = {1992},
    NUMBER = {1},
     PAGES = {117--136},
      ISSN = {0178-8051,1432-2064},
   MRCLASS = {60J10 (60K99)},
  MRNUMBER = {1156453},
MRREVIEWER = {David\ J.\ Aldous},
       DOI = {10.1007/BF01205239},
       URL = {https://doi.org/10.1007/BF01205239},
}

@article {DST2015,
    AUTHOR = {Disertori, Margherita and Sabot, Christophe and Tarr\`es,
              Pierre},
     TITLE = {Transience of edge-reinforced random walk},
   JOURNAL = {Comm. Math. Phys.},
  FJOURNAL = {Communications in Mathematical Physics},
    VOLUME = {339},
      YEAR = {2015},
    NUMBER = {1},
     PAGES = {121--148},
      ISSN = {0010-3616,1432-0916},
   MRCLASS = {60J05 (05C81 60G50)},
  MRNUMBER = {3366053},
MRREVIEWER = {Barbara\ Martinucci},
       DOI = {10.1007/s00220-015-2392-y},
       URL = {https://doi.org/10.1007/s00220-015-2392-y},
}

@article {P2022,
    AUTHOR = {Poudevigne-Auboiron, R\'emy},
     TITLE = {Monotonicity and phase transition for the {VRJP} and the
              {ERRW}},
   JOURNAL = {J. Eur. Math. Soc. (JEMS)},
  FJOURNAL = {Journal of the European Mathematical Society (JEMS)},
    VOLUME = {26},
      YEAR = {2024},
    NUMBER = {3},
     PAGES = {789--816},
      ISSN = {1435-9855,1435-9863},
   MRCLASS = {60K35 (60K37 82B44)},
  MRNUMBER = {4721024},
       DOI = {10.4171/jems/1298},
       URL = {https://doi.org/10.4171/jems/1298},
}

@ARTICLE{RT2023,
       author = {{Raimond}, Olivier and {Tarres}, Pierre},
        title = "{Non-convergence to unstable equilibriums for continuous-time and discrete-time stochastic processes}",
      journal = {arXiv e-prints},
         year = 2023,
        month = nov,
          eid = {arXiv:2311.02978},
        pages = {arXiv:2311.02978},
          doi = {10.48550/arXiv.2311.02978},
archivePrefix = {arXiv},
       eprint = {2311.02978},
 primaryClass = {math.PR},
       adsurl = {https://ui.adsabs.harvard.edu/abs/2023arXiv231102978R}
}

@article {BSS2014,
    AUTHOR = {Basdevant, Anne-Laure and Schapira, Bruno and Singh, Arvind},
     TITLE = {Localization on 4 sites for vertex-reinforced random walks on
              {$\mathbb{Z}$}},
   JOURNAL = {Ann. Probab.},
  FJOURNAL = {The Annals of Probability},
    VOLUME = {42},
      YEAR = {2014},
    NUMBER = {2},
     PAGES = {527--558},
      ISSN = {0091-1798},
   MRCLASS = {60K35},
  MRNUMBER = {3178466},
       DOI = {10.1214/12-AOP811},
       URL = {https://doi.org/10.1214/12-AOP811},
}

@article {Schapira2021,
    AUTHOR = {Schapira, Bruno},
     TITLE = {Localization on 5 sites for vertex reinforced random walks:
              towards a characterization},
   JOURNAL = {Ann. Appl. Probab.},
  FJOURNAL = {The Annals of Applied Probability},
    VOLUME = {31},
      YEAR = {2021},
    NUMBER = {4},
     PAGES = {1774--1786},
      ISSN = {1050-5164},
   MRCLASS = {60K35 (60G50)},
  MRNUMBER = {4312846},
MRREVIEWER = {Malin P. Forsstr\"{o}m},
       DOI = {10.1214/20-aap1633},
       URL = {https://doi.org/10.1214/20-aap1633},
}
\end{document}